\documentclass[11pt,a4paper]{article}
\usepackage[utf8]{inputenc}
\usepackage{amsmath}
\usepackage{amsfonts}
\usepackage{amssymb}
\usepackage{amsthm}
\usepackage{setspace}
\usepackage{hyperref}
\usepackage{graphicx}
\usepackage{stmaryrd}
\usepackage{mathtools}
\usepackage{subcaption}
\usepackage[T1]{fontenc}
\usepackage{xcolor}
\usepackage{float}
\usepackage{tikz-cd}
\usepackage[noadjust]{cite}
\usepackage{cite}
\usepackage{ragged2e}
\usepackage[margin=1.2in]{geometry}
\graphicspath{ {./images/} }
\newtheorem{theorem}{Theorem}[section]
\newtheorem{proposition}[theorem]{Proposition}
\newtheorem{lemma}[theorem]{Lemma}
\newtheorem{corollary}[theorem]{Corollary}

\theoremstyle{definition}
\newtheorem{definition}[theorem]{Definition}
\newtheorem{notation}[theorem]{Notation}

\newtheorem{remark}[theorem]{Remark}
\newtheorem{conjecture}[theorem]{Conjecture}

\newtheorem*{mainques}{Main Question}

\newcommand{\mcg}{\operatorname{Mod}}

\newcommand{\lk}{\operatorname{lk}}
\newcommand{\qi}{\operatorname{QI}}
\newcommand{\aut}{\operatorname{Aut}}
\newcommand{\fix}{\operatorname{Fix}}
\newcommand{\actson}{\curvearrowright}

\begin{document}
	
	\begin{center}
		{\large \textbf{Rigidity and classification results for large-type Artin groups}}
	\end{center}
	
	\begin{center}
		Jingyin Huang, Damian Osajda and Nicolas Vaskou
	\end{center}
	
	\begin{abstract}
		\centering \justifying We compute the automorphism group of the intersection graph of many large-type Artin groups. This graph is an analogue of the curve graph of mapping class groups but in the context of Artin groups. As an application, we deduce a number of rigidity and classification results for these groups, including computation of outer automorphism groups, commensurability classification, quasi-isometric rigidity, measure equivalence rigidity, orbit equivalence rigidity, rigidity of lattice embedding, and rigidity of cross-product von Neumann algebra.
	\end{abstract}

	\noindent \rule{7em}{.4pt}\par
	
	\small
	
	
	\normalsize
	
	\tableofcontents
	
	\section{Introduction}
	
	\subsection{Background and motivation}
	
	Let $\Gamma$ be a finite simplicial graph, with each edge labelled by an integer which is $\ge 2$. The \emph{Artin group with defining graph $\Gamma$}, denoted $A_{\Gamma}$, is the group defined by the following presentation: its generators are the vertices of $\Gamma$, and its relators are given by $$\underbrace{aba\cdots}_{m}=\underbrace{bab\cdots}_{m}$$ whenever $a$ and $b$ span an edge with label $m$. 
	Some well-known subclasses of Artin groups include braid groups, and \emph{right-angled} Artin groups, where all label of edges are $2$ (hence all relators are commutators).
	
	We are motivated by taking the class of Artin groups as a playground of understanding different forms of rigidity properties of groups (as mentioned in the abstract); and looking for new rigidity phenomenon. 
	Many earlier works in this direction are on right-angled Artin groups \cite{bestvina2008asymptotic,behrstock2008quasi,huang2016groups,huang2018commensurability,margolis2020quasi,horbez2022measure}. It was understood that several rigidity properties of more classical classes of groups, like higher rank lattices and mapping class groups of surfaces, either break down for right-angled Artin groups, or they still hold, but for a different reason. From the viewpoint of quasi-isometric rigidity and measure equivalence rigidity, right-angled Artin groups are actually quite ``flexible'', and known examples that demonstrate the ``flexibility'' of right-angled Artin groups crucially exploit the commuting relations of these Artin groups \cite[Section 11]{bestvina2008asymptotic}, \cite[Example 4.19]{MR3692971}, \cite[Section 4]{horbez2022measure}. This makes it also interesting to the study of Artin groups on the other extreme - \emph{large-type} Artin groups, where all edges have label $\ge 3$ (hence there are no commuting relations).
	
	Some classes of large-type Artin groups indeed satisfy much stronger rigidity properties compared to right-angled Artin groups, and there are even examples of large-type Artin groups satisfying form of rigidity which are still open for higher rank lattices and mapping class groups of surfaces \cite[Section 11]{horbez2020boundary}. However, a general picture of rigidity and classification of large-type Artin groups is still missing, as these results relies on a foundational work of Crisp \cite{crisp2005automorphisms}, which requires the defining graph to be \emph{triangle free} (i.e. $\Gamma$ does not contain 3-cycles). As Crisp already explained in his article, his method breaks down outside the triangle free situation. Currently, we do not even know in general what could be the generating set of the automorphism group of a large-type Artin group. Very recently, the third author initiated an in-depth study of several fundamental properties of large-type Artin groups whose defining graphs allow 3-cycles inside \cite{vaskou2023automorphisms}.
	
	Several aspects in the study of large-type Artin groups, including commensurability classification and computation of automorphism groups, as well as understanding advanced form of rigidity like quasi-isometric rigidity and measure equivalence rigidity, all reduce to the study of a single object, namely the \emph{intersection graphs} of large-type Artin groups, by previous works \cite{huang2017quasi,horbez2020boundary}. More precisely, we will study the Main Questions in Section~\ref{subsec:intro1}.
	
	\subsection{Main result: automorphism group of intersection graph}
	\label{subsec:intro1}
	
	\begin{definition}
		Suppose $A_\Gamma$ is a large-type Artin group. The \emph{intersection graph} of $A_\Gamma$, denoted by $I_\Gamma$, is a simplicial graph whose vertices are in 1-1 correspondence with $\mathbb Z$-subgroups of $A_\Gamma$ whose centraliser is not virtually abelian, and two vertices are adjacent if the associated $\mathbb Z$-subgroups commute.
	\end{definition}
	
	\begin{remark}
		The intersection graph is an analogue of the curve graph of mapping class groups in the setting of large-type Artin groups. In a few sporadic cases when the Artin groups are commensurable to some mapping class groups, the intersection graphs and the curve graphs are isomorphic.
		Intersection graphs were first introduced by Crisp \cite{crisp2005automorphisms}, under the name of $\Theta$-graph, in the special case of triangle free large-type Artin groups. Crisp's definition requires an adjustment for studying more general large-type Artin groups, and this is done in \cite{huang2017quasi}, where the intersection graph is actually defined for all 2-dimensional Artin groups, which contain large-type Artin groups as a subclass. It is worth mentioning that for studying right-angled Artin groups, there is also an analogue of curve graph, called the extension graph, introduced by Kim and Koberda \cite{kim2013embedability}.
	\end{remark}
	
	The graph automorphism group of $I_\Gamma$, denoted by $\aut(I_\Gamma)$, is endowed with the topology of pointwise convergence, which makes it a Polish group. Note that $I_\Gamma$ is locally infinite except for a few more or less trivial cases, so a priori $\aut(I_\Gamma)$ could possibly be a huge Polish group. Nevertheless, we ask the following.
	
	\begin{mainques}
		Suppose $A_\Gamma$ is of large-type. When is the Polish group $\aut(I_\Gamma)$ locally compact? When is $\aut(I_\Gamma)$ discrete?
	\end{mainques}
	
	The graph $I_\Gamma$ is a common invariant for studying all the forms of rigidity/classification aspects mentioned in the abstract, which motivates the main question.
	Very roughly speaking, the ``smaller'' $\aut(I_\Gamma)$ is, the more ``rigid'' $A_\Gamma$ is. 
	
	There is an obstruction to $\aut(I_\Gamma)$ being locally compact, which has its trace in the outer automorphism group of $A_\Gamma$. Recall that if a group $G$ admits a splitting $G=A*_CB$, for an element $z\in A$ that centralises $C$, the \emph{twist} by $z$ near $A$ is defined as the automorphism of $G$ which is the identity on $A$ and is the conjugation by $z$ on $B$ (see \cite{levitt2005automorphisms} for more general definition of twists). The most obvious twist for a large-type Artin group $A_\Gamma$ happens when $\Gamma$ has a separating vertex or a separating edge and consider a splitting of $A_\Gamma$ induced by this separating vertex or edge. Hence we will say $\Gamma$ is \emph{twistless} if $\Gamma$ does not contain separating vertex or edge.

	\begin{conjecture}
		\label{conj:main}
		Suppose $A_\Gamma$ is of large-type. Then
		\begin{enumerate}
			\item $\aut(I_\Gamma)$ is locally compact if and only if $\Gamma$ is connected and twistless, moreover, in this case, $\aut(I_\Gamma)$ is isomorphic to the cellular automorphism of the Cayley complex of $A_\Gamma$;
			\item $\aut(I_\Gamma)$ is discrete if and only if $\Gamma$ is connected, twistless and star rigid, moreover, in this case the natural action of $A_\Gamma$ on $I_\Gamma$ gives an embedding $A_\Gamma\to \aut(I_\Gamma)$ with finite index image.
		\end{enumerate}
	\end{conjecture}

	Recall that $\Gamma$ is \emph{star rigid} if any label-preserving automorphism of $\Gamma$ fixing the star of a vertex pointwise must be identity. The only if the direction of the conjecture is not hard, proved in Proposition~\ref{lem:onlyif}. The other direction is much more interesting, which we manage to prove in the following case. 
	
	\begin{definition} \label{DefiTwistlessHierarchy}
		An \emph{admissible decomposition of $\Gamma$} is made of a pair of full subgraphs $\Gamma_1$ and $\Gamma_2$ of $\Gamma$ such that $\Gamma=\Gamma_1\cup\Gamma_2$. This decomposition is \emph{twistless} if $\Gamma_1\cap\Gamma_2$ is not the empty set, or a vertex, or an edge. We say $\Gamma$ has a \emph{hierarchy} terminating in a class of graphs $\mathcal C$, if it is possible to start with $\Gamma$ and keep performing admissible decomposition for finitely many times until we end up with a collection of graphs in $\mathcal C$.
		
		To see why such hierarchy is a natural thing to consider,
		note that any graph has a hierarchy terminating in complete graphs, which corresponds to any Artin groups being an iterated amalgamated product starting from Artin groups with complete defining graphs. This is a standard consideration in the study of Artin groups.

		The hierarchy is \emph{twistless} if each step of the decomposition is twistless. We say $\Gamma$ is a \emph{twistless star}, if $\Gamma$ is the star of a vertex $v\in \Gamma$, and $\Gamma$ is twistless. Note that every complete graph is a twistless star. We will be considering graphs admitting hierarchies terminating on twistless stars, instead of graphs admitting hierarchies terminating on complete graphs.
	\end{definition}
	
	\begin{theorem}(=Theorem~\ref{thm:main})
		\label{thm:mainintro}
		Suppose $A_\Gamma$ is of large-type. If $\Gamma$ admits a twistless hierarchy terminating in twistless stars, then Conjecture~\ref{conj:main} holds.
	\end{theorem}
	
	As an example, if $\Gamma$ is the 1-skeleton of a triangulation of a closed $n$-manifold with $n\ge 2$, then $\Gamma$ satisfies the assumption of Theorem~\ref{thm:mainintro}, hence $\aut(I_\Gamma)$ is locally compact, see Remark~\ref{rmk:example}.

	By contrast, the analogue object in the world of right-angled Artin groups, namely the extension graph, never has locally compact automorphism group.
	
	Theorem~\ref{thm:mainintro} relies on a more general (and more technical) theorem converting Conjecture~\ref{conj:main} to a property of the modified Deligne complexes associated with the Artin groups, see Corollary~\ref{cor:convert} for a precise statement, which we believe could be a useful tool towards fully resolve Conjecture~\ref{conj:main}.
	
	While we prove that this property holds true in the case where $\Gamma$ admits a twistless hierarchy terminating in twistless stars, the authors of \cite{blufstein2024homomorphisms} prove that provided $A_{\Gamma}$ is of XXXL-type (i.e. all coefficient are $\geq 6$), this property holds true whenever $\Gamma$ is twistless, which provide another strong evidence for Conjecture~\ref{conj:main}.
	
	In particular, from Corollary~\ref{cor:convert}, we also deduce the following.
	
	\begin{theorem}
		Suppose that $A_{\Gamma}$ is of XXXL-type and $\Gamma$ is connected and twistless. Then Conjecture~\ref{conj:main} holds.
	\end{theorem}
	
	In the following sections (Section~\ref{Section:isomorphism} to Section~\ref{Section:W*}), we discuss a variety of applications of the main result. And in Section~\ref{Secfurtherresult}, we discuss further results on a purely combinatorial criterion on how to generalize all the applications we have. We will discuss most applications under the assumption of twistless hierarchy, then discuss how to adjust them for XXXL case, as well as some more general criterion which could potentially give the optimal results.
	
	\subsection{Applications to automorphisms and isomorphism classification} \label{Section:isomorphism}
	
	The question of determining the full automorphism group of an Artin group is wide open in general, although it has been solved in specific cases, such as for right-angled Artin groups \cite{servatius1989automorphisms}. In the world of large-type Artin groups, a generating set has been found for so-called CLTTF Artin groups (\cite{crisp2005automorphisms}), and more recently, a presentation has been given for large-type free-of-infinity Artin groups (\cite{vaskou2023automorphisms}).
	
	For $A_\Gamma$, we let $\aut_{\Gamma}(A_{\Gamma})$ be the subgroup of the automorphism group of $A_\Gamma$ generated by inner automorphisms, the global inversion that sends every generator to its inverse, and automorphisms induced by label-preserving automorphisms of $\Gamma$. This subgroup was extensively studied and showed to have a very natural interpretation in \cite{jones2024fixed}.
	
	While the general case of large-type Artin groups remains open, we extend the result of \cite{vaskou2023automorphisms} and compute $\aut(A_{\Gamma})$ for many new Artin groups:
	
	\begin{theorem}(=Corollary~\ref{cor:qi}) \label{ThmAutomorphismsIntro}
		Suppose $A_\Gamma$ is a large-type Artin group such that $\Gamma$ admits a twistless hierarchy terminating in twistless stars. Then $\aut(A_{\Gamma}) = \aut_{\Gamma}(A_{\Gamma})$. In particular, the outer automorphism group of $A_\Gamma$ is finite, and it is generated by the global inversion and label-preserving automorphisms of the defining graph $\Gamma$.
	\end{theorem}
	
	The isomorphism problem, that of determining whether two presentation graphs $\Gamma$ and $\Gamma'$ give rise to isomorphic Artin groups, is also open in general. A conjecture from \cite{brady2002rigidity} is that $A_{\Gamma}$ and $A_{\Gamma'}$ are isomorphic if and only if $\Gamma'$ can be obtained from $\Gamma$ by a series of diagram twists. This conjecture has been proved true for all large-type Artin groups (see \cite{vaskou2023isomorphism}). In \cite{martin2023characterising}, it was further proved true while only assuming that one of the two Artin groups is large-type.
	
	Another natural question is to ask when two Artin groups $A_{\Gamma}$ and $A_{\Gamma'}$ can be commensurable (i.e. they have isomorphic finite index subgroups). The only results concerning large-type Artin groups is that of Crisp (\cite{crisp2005automorphisms}) for some triangle-free Artin groups.
	
	In this paper, we recover a solution to the isomorphism problem for many large-type Artin groups, and we provide a solution to the question of commensurability for many new Artin groups:

	\begin{theorem}(=Corollary~\ref{cor:classification1}) \label{ThmIsomorphismsIntro}
		Suppose $A_\Gamma$ and $A_{\Gamma'}$ are two large-type Artin groups such that $\Gamma$ admits a twistless hierarchy terminating in twistless stars. 
		Then the following are equivalent:
		\begin{enumerate}
			\item $A_\Gamma$ and $A_{\Gamma'}$ are isomorphic;
			\item $A_\Gamma$ and $A_{\Gamma'}$ are commensurable;
			\item there is a label-preserving isomorphism between $\Gamma$ and $\Gamma'$.
		\end{enumerate}
	\end{theorem}

	\subsection{Applications to QI rigidity and classification}
	\label{subsec:intro2}
	We say two groups are \emph{virtually isomorphic}, if they are isomorphic up to quotienting by finite normal subgroups and passing to finite index subgroups.
	
	Let $\Delta_{333}$ be a triangle such that each edge is labelled by $3$.
	For the study of quasi-isometric rigidity of large-type Artin groups $A_\Gamma$, it is only interesting to consider the case when $\Gamma\neq\Delta_{333}$, as when $\Gamma=\Delta_{333}$, $A_\Gamma$ is a finite index subgroup of the mapping class group of the $5$-punctured sphere, whose quasi-isometric rigidity is already known \cite{hamenstaedt2005geometry,behrstock2008quasi}.
	
	\begin{theorem}(=Corollary~\ref{cor:qi})
		\label{theo:qi}
		Suppose $A_\Gamma$ is a large-type Artin group with $\Gamma\neq \Delta_{333}$ such that $\Gamma$ admits a twistless hierarchy terminating in twistless stars.  Then 
		\begin{enumerate}
			\item any self quasi-isometry of $A_\Gamma$ is at bounded distance from an automorphism of the Cayley complex $C_\Gamma$;
			\item the quasi-isometry group of $A_\Gamma$ is isomorphic to $\aut(C_\Gamma)$;
			\item any finitely generated group quasi-isometric to $A_\Gamma$ is virtually isomorphic to a uniform lattice in the locally compact topological group $\aut(C_\Gamma)$.
		\end{enumerate}
		If in addition $\Gamma$ is star rigid, then 
		\begin{enumerate}
			\item any self quasi-isometry of $A_\Gamma$ is at bounded distance from an automorphism of $A_\Gamma$;
			\item the quasi-isometry group of $A_\Gamma$ is isomorphic to $\aut_{\Gamma}(A_{\Gamma})$ ;
			\item any finitely generated group quasi-isometric to $A_\Gamma$ is virtually isomorphic to $A_\Gamma$.
		\end{enumerate}
	\end{theorem}
	
	Quasi-isometric rigidity of some triangle-free large-type Artin groups were proved in \cite{huang2017quasi}. The simplest examples which are covered here, but not in \cite{huang2017quasi}, is that $\Gamma$ being the 1-skeleton of a triangulation of some closed $n$-manifold, with $n\ge 2$.
	
	Note that there are two forms of quasi-isometric rigidity in Theorem~\ref{theo:qi}, corresponding the two scenarios in Conjecture~\ref{conj:main}. The first form of rigidity is that any self quasi-isometry of the group $G$ is at bounded distance from an isometry of a canonical space $X$ associated with this group, hence we can ``uniformize'' the class of all groups quasi-isometric to $G$ in the sense that they can be put virtually as lattices in the same canonical locally compact topological group, which is the isometry group of $X$. However, isometry group of $X$ could be rather large. This form of rigidity is quite similar to the quasi-isometric rigidity of uniform higher rank lattices \cite{kleiner1997rigidity,eskin1997quasi}.
	
	The second form of rigidity, which is stronger, is that the quasi-isometry group of $G$ has a finite index subgroup such that any self quasi-isometry of $G$ inside this subgroup is at bounded distance from a left translation of $G$, hence any finitely generated group quasi-isometric to $G$ is virtually isomorphic to $G$. This form of rigidity is quite similar to the quasi-isometric rigidity of mapping class groups of surfaces \cite{hamenstaedt2005geometry,behrstock2012geometry}.
	
	It is natural to ask how far away are we from a complete understanding of these two forms of rigidity in the class of large-type Artin groups. To this end, we propose the following conjecture.
	\begin{conjecture}
		\label{conj:qi}
		Suppose $A_\Gamma$ is a large-type Artin groups with $\Gamma\neq\Delta_{333}$. Then
		\begin{enumerate}
			\item any self quasi-isometry of $A_\Gamma$ is at bounded distance from an automorphism of the Cayley complex $C_\Gamma$ if and only if $\Gamma$ is connected and twistless;
			\item any self quasi-isometry of $A_\Gamma$ is at bounded distance from an automorphism of $A_\Gamma$ if and only if $\Gamma$ is connected, twistless and star rigid.
		\end{enumerate}
	\end{conjecture}
	
	We have reduced Conjecture~\ref{conj:qi} to a purely combinatorial statement about the modified Deligne complex of a twistless large-type Artin group, see Theorem~\ref{thm:qi reduction}. Again, the only if direction is not hard, and the if direction is more interesting. Conjecture~\ref{conj:main} and Conjecture~\ref{conj:qi} are closely related in the sense that they are reduced to the same combinatorial statement about modified Deligne complexes. 
	
	By contract, right-angled Artin groups never satisfy any of the forms of quasi-isometric rigidity mentioned in Conjecture~\ref{conj:qi}.
	
	We also have a classification result.
	
	\begin{theorem}(=Corollary~\ref{cor:classification1}) \label{ThmQiClassIntro}
		Suppose $A_\Gamma$ and $A_{\Gamma'}$ are two large-type Artin groups such that $\Gamma$ admits a twistless hierarchy terminating in twistless stars. 
		Then  $A_\Gamma$ and $A_{\Gamma'}$ are quasi-isometric if and only if  there is a label-preserving isomorphism between $\Gamma$ and $\Gamma'$. 
	\end{theorem}
	
	\subsection{Applications to ME/OE rigidity and classification}
	Recall that two countable groups $\Gamma_1,\Gamma_2$ are \emph{measure equivalent} (a notion introduced by Gromov \cite{gromov1993geometric} as a measurable analogue of quasi-isometry) if there exists a standard measure space $\Sigma$, equipped with an action of the product $\Gamma_1\times\Gamma_2$ by measure-preserving Borel automorphisms, such that for every $i\in\{1,2\}$, the action of the factor $\Gamma_i$ on $\Sigma$ is free and has a fundamental domain of finite measure. An important example is that any two lattices in the same locally compact second countable group $G$ are always measure equivalent, through their left/right action by multiplication on $G$, preserving the Haar measure. Another example to have in mind, if that if two countable groups are virtually isomorphic, then they are measure equivalent.
	Some notable measure equivalence rigidity results are proved for higher rank lattices \cite{furman1999gromov}, mapping class groups of surfaces \cite{kida2010measure} and $Out(F_n)$ \cite{guirardel2021measure}. See \cite{furman2009survey} for a nice survey.

	We can obtain measure equivalence rigidity results under almost the same assumption as Theorem~\ref{theo:qi}, with an extra requirement as follows. Recall that an Artin group is of \emph{hyperbolic-type}, if its associated Coxeter group is word hyperbolic. hyperbolic-type Artin groups are never word-hyperbolic unless they are free groups.
	A large-type Artin group $A_\Gamma$ is of hyperbolic-type, if $\Gamma$ does not contain a triangle with all of its edges labelled by $3$.
	
	\begin{theorem}(=Corollary~\ref{cor:ME})
		\label{theo:ME}
		Suppose $A_\Gamma$ is a large-type Artin group of hyperbolic-type and suppose that $\Gamma$ admits a twistless hierarchy terminating in twistless stars. Then any countable group measure equivalent to $A_\Gamma$ is virtually a lattice in the locally compact topological group $\aut(C_\Gamma)$.
		
		If in addition the graph $\Gamma$ is star rigid, then any countable group measure equivalent to $A_\Gamma$ is virtually isomorphic to $A_\Gamma$. 
	\end{theorem}
	
	Again, similar to the situation of quasi-isometric rigidity, here we see two different forms of measure equivalence rigidity, corresponding to the two scenarios in Conjecture~\ref{conj:main}. The first form of rigidity is closer to what happens to higher rank lattices \cite{furman1999gromov}, and the second form of rigidity is closer to what happens to mapping class groups of surfaces \cite{kida2010measure}.

	Measure-equivalence rigidity of some triangle-free large-type Artin groups were proved in \cite{horbez2020boundary}. The simplest examples which are covered here, but not in \cite{horbez2020boundary}, is that $\Gamma$ being the 1-skeleton of a triangulation of some closed $n$-manifold, with $n\ge 2$.
	
	Next we recall the notion of \emph{stably orbit equivalence}, which is closely related to measure equivalence.
	Let $G_1$ and $G_2$ be two countable groups, and for every $i\in\{1,2\}$, let $X_i$ be a standard probability space equipped with an essentially free measure-preserving action of $G_i$ by Borel automorphisms. The actions $G_1\actson X_1$ and $G_2\actson X_2$ are \emph{stably orbit equivalent} if there exist Borel subsets $Y_1\subseteq X_1$ and $Y_2\subseteq X_2$ of positive measure and a measure-scaling isomorphism $f:Y_1\to Y_2$ such that for almost every$y\in Y_1$, one has $$f((G_1\cdot y)\cap Y_1)=(G_2\cdot f(y))\cap Y_2.$$

	The actions $G_1\actson X_1$ and $G_2\actson X_2$ are \emph{virtually conjugate} if there exist short exact sequences $1\to F_i\to G_i\to Q_i\to 1$ with $F_i$ finite for every $i\in\{1,2\}$, finite-index subgroups $Q_1^0\subseteq Q_1$ and $Q_2^0\subseteq Q_2$, and conjugate actions $Q_1^0\actson X'_1$ and $Q_2^0\actson X'_2$ such that for every $i\in\{1,2\}$, the $Q_i$-action on $X_i/F_i$ is induced from the $Q_i^0$-action on $X'_i$.  Virtually conjugate actions are always stably orbit equivalent.

	It turns out that two countable groups are measure equivalent if and only if they admit essentially free ergodic measure-preserving actions by Borel automorphisms on standard probability spaces which are stably orbit equivalent (see e.g.\ \cite[Theorem~2.5]{furman2009survey}).
	
	\begin{theorem}(=Corollary~\ref{cor:ME})
		\label{theo:OE}
		Suppose $A_\Gamma$ is a large-type Artin group such that it is also of hyperbolic-type and $\Gamma$ admits a twistless hierarchy terminating in twistless stars. Also suppose that $\Gamma$ is star-rigid.
		
		Then for any ergodic measure-preserving essentially free action $A_\Gamma\actson X$ on a probability measure space $X$, and any countable group $G'$ with an ergodic measure-preserving essentially free $G'$-action on a standard probability space $X'$, as long as actions $G\actson X$ and $G'\actson X'$ are stably orbit equivalent, then they are virtually conjugate.
	\end{theorem}
	
	We also obtain classification results as follows.
	
	\begin{theorem}(=Corollary~\ref{cor:classification2}) \label{ThmMEClassIntro}
		Suppose $A_\Gamma$ and $A_{\Gamma'}$ are two large-type and hyperbolic-type Artin groups such that $\Gamma$ admits a twistless hierarchy terminating in twistless stars. Then the following are equivalent:
		\begin{enumerate}
			\item $A_\Gamma$ and $A_{\Gamma'}$ are measure-equivalent;
			\item $A_\Gamma$ and $A_{\Gamma'}$ admit essentially free ergodic measure-preserving actions by Borel automorphisms on standard probability spaces which are stably orbit equivalent;
			\item there is a label-preserving isomorphism between $\Gamma$ and $\Gamma'$. 
		\end{enumerate}
	\end{theorem}
	
	\subsection{Application to lattice envelope}
	
	Recall a \emph{lattice envelope} of a discrete group $G$ is a locally compact second countable topological group where $G$ embeds as a lattice. The study of lattice envelope was initiated in \cite{bader2020lattice}.
	Using the measure-equivalence rigidity results and \cite{horbez2020boundary}, we characterise all lattice envelopes of certain large-type Artin groups, as well as their finite index subgroups.
	
	\begin{theorem}(=Theorem~\ref{theo:lattice-embeddings1})
		\label{theo:lattice-embeddings}
		Suppose $A_\Gamma$ is a large-type Artin group of hyperbolic-type and suppose that $\Gamma$ admits a twistless hierarchy terminating in twistless stars. Suppose $G'$ is a finite index subgroup of $A_\Gamma$ which embeds into a locally compact second countable group $\mathsf{H}$ as a lattice via $f:G'\to \mathsf{H}$.
		
		Then there exists a continuous homomorphism $\psi:\mathsf{H}\to \aut(C_\Gamma)$ with compact kernel such that $\psi\circ f$ coincides with the natural map $G'\to\aut(C_\Gamma)$. 
	\end{theorem}

	If we assume in addition that $\Gamma$ is star rigid, then the lattice envelope of $A_\Gamma$ is more or less trivial.
	
	\begin{theorem}(=Theorem~\ref{theo:lattice-embeddings2})
		\label{theo:lattice-embeddings-2}
		Suppose $A_\Gamma$ is a large-type Artin group of hyperbolic-type and suppose that $\Gamma$ admits a twistless hierarchy terminating in twistless stars. Also suppose that $\Gamma$ is star-rigid. Then there exists $G_0$ which is virtually isomorphic to $G$ such that the following holds true.
		
		Suppose $A_\Gamma$ or one of its finite index subgroup embeds into a locally compact second countable group $\mathsf{H}$ as a lattice. Then there exists a continuous homomorphism $g:\mathsf{H}\to G_0$ with compact kernel.
	\end{theorem}

	\subsection{Application to $W^*$-rigidity} \label{Section:W*}
	
	Given a countable group $G$ and a standard probability space $(X,\mu)$ equipped with an ergodic measure-preserving free $G$-action $G\actson (X,\mu)$ by Borel automorphisms, a celebrated construction of Murray and von Neumann \cite{MvN} associates a cross-product von Neumann algebra $L(G\actson X)$ to the $G$-action on $X$. The actions $G\actson X$ and $H\actson Y$ are $W^*$-equivalent if $L(G\actson X)$ and $L(H\actson Y)$ are isomorphic: this is a weaker notion than orbit equivalence. They are \emph{stably} $W^*$-equivalent if the associated von Neumann Algebras have isomorphic amplifications.
	
	As a consequence of previous discussion and \cite{horbez2020boundary}, 
	we can obtain $W^*$-rigidity results under almost the same assumption as Theorem~\ref{theo:OE}, with an extra requirement that $\Gamma$ is not a complete graph.
	
	\begin{theorem}(=Corollary~\ref{cor:von-neumann})
		\label{theo:W}
		Let $A_\Gamma$ be a large-type Artin group of hyperbolic-type. Suppose that
		\begin{enumerate}
			\item $\Gamma$ admits a twistless hierarchy terminating in twistless stars. 
			\item $\Gamma$ is star rigid, and $\Gamma$ is not a complete graph.
		\end{enumerate}
		Let $H$ be a countable group, and let $A_\Gamma\actson X$ and $H\actson Y$ be free, ergodic, measure-preserving actions by Borel automorphisms on standard probability spaces.  
		
		If these two actions are stably $W^*$-equivalence, then they are virtually conjugate (in particular $A_\Gamma$ and $H$ are virtually isomorphic).
	\end{theorem}
	
	Popa's deformation/rigidity theory has led to many examples of $W^*$-superrigid actions, starting with the works of Peterson \cite{peterson2010examples} and Popa and Vaes \cite{popa2010group}; but knowing that all actions of a given group are $W^*$-superrigid is a phenomenon that has only
	been established in very few cases so far \cite{houdayer2013class,chifan2015w,horbez2020boundary,horbez2022measure}. Our $W^*$-superrigidity
	theorems are obtained as a consequence of orbit equivalence rigidity and of the uniqueness of $L^\infty(X)$ as a Cartan subalgebra of $L(A_\Gamma\actson X)$, up to unitary conjugacy. The
	latter is obtained from the existence of decompositions of $A_\Gamma$ as an amalgamated free
	product, using a theorem of Ioana \cite{ioana2015cartan}. The extra assumption that $\Gamma$ is not a complete graph is used to produce splittings of $A_\Gamma$ that satisfy the assumptions in Ioana's theorem.
	
	Note that many previous rigidity results (quasi-isometric, measure equivalent etc.) were already know in the context of more classical objects, like higher rank lattices or mapping class groups of surfaces. However, the form of rigidity in Theorem~\ref{theo:W} (i.e. $W^*$-rigidity) is still open for higher rank lattices or mapping class groups of surfaces.
	
	We also have a classification result.
	\begin{theorem}
		\label{thm:classification2}
		Suppose $A_\Gamma$ and $A_{\Gamma'}$ are two large-type and hyperbolic-type Artin groups such that $\Gamma$ admits a twistless hierarchy terminating in twistless stars. If in addition both $\Gamma$ and $\Gamma'$ are not complete graphs, then the following are equivalent. 
		\begin{enumerate}
			\item there is a label-preserving isomorphism between $\Gamma$ and $\Gamma'$ ;
			\item	the cross product von-Neumann algebras $L(A_\Gamma\actson X)$ and $L(A_{\Gamma'}\actson Y)$ have isomorphic amplifications, for some free, ergodic, measure-preserving actions by Borel automorphisms on standard probability spaces $A_\Gamma\actson X$ and $A_{\Gamma'}\actson Y$.
		\end{enumerate}
	\end{theorem}
	
	\subsection{Further results} \label{Secfurtherresult}
	
	As previously mentioned, Theorem \ref{thm:mainintro} and all the applications can be reduced to a combinatorial property concerning the intersection pattern of some families of standard trees in the modified Deligne complex. The presentation graphs $\Gamma$ satisfying that property are called “fundamental” (see Definition \ref{DefiFundGraphI}). Basically, a graph $\Gamma$ is fundamental if every subgraph $G$ of $I_{\Gamma}$ that is isomorphic to the barycentric subdivision of $\Gamma$ corresponds to a unique fundamental domain in the Deligne complex $D_{\Gamma}$ (equivalently, see Definition \ref{DefiFundGraphII}).
	
	The following follows from our work:
	
	\begin{theorem} \label{ThmExtendingAllResults}
		Let $\mathcal{C}$ be a class of presentation graphs that correspond to large-type Artin groups. If every graph in $\mathcal{C}$ is fundamental, then we recover the results of Theorems \ref{ThmAutomorphismsIntro}, \ref{ThmIsomorphismsIntro}, \ref{theo:qi} and \ref{ThmQiClassIntro} if we replace the condition “$\Gamma$ (resp. $\Gamma$ and $\Gamma'$) admits a twistless hierarchy terminating in twistless stars” with the condition “$\Gamma$ (resp. $\Gamma$ and $\Gamma'$) belongs to $\mathcal{C}$”.
		
		Moreover, if the graphs in $\mathcal{C}$ are also of hyperbolic-type, then the above also works for Theorems \ref{theo:ME}, \ref{theo:OE}, \ref{ThmMEClassIntro}, \ref{theo:lattice-embeddings}, \ref{theo:lattice-embeddings-2}, \ref{theo:W} and \ref{thm:classification2}.
	\end{theorem}
	
	This theorem is a combination of Proposition~\ref{prop:qi}, Proposition~\ref{prop:ME}, Proposition~\ref{prop:SOE}, Proposition~\ref{prop:von-neumann-strong} and Lemma~\ref{lem:classification}.
	\medskip
	
	In this paper, we prove that any large-type graph $\Gamma$ that admit a twistless hierarchy terminating in twistless stars is fundamental (see Proposition \ref{prop:fund}). 
	
	In \cite{blufstein2024homomorphisms}, the authors also study the property of being fundamental, albeit not under that name. They show that every graph $\Gamma$ that is connected, twistless, and whose labels are at least $6$, is fundamental. In particular, the following holds. 
	
	\begin{theorem} \label{ThmExtendingAllResults1}
		Theorems \ref{ThmAutomorphismsIntro}, \ref{ThmIsomorphismsIntro}, \ref{theo:qi}, \ref{ThmQiClassIntro}, \ref{theo:ME}, \ref{theo:OE}, \ref{ThmMEClassIntro}, \ref{theo:lattice-embeddings}, \ref{theo:lattice-embeddings-2}, \ref{theo:W} and \ref{thm:classification2} hold if we replace the condtion “$\Gamma$ (resp. $\Gamma$ and $\Gamma'$) admits a twistless hierarchy terminating in twistless stars” with the condition “$\Gamma$ (resp. $\Gamma$ and $\Gamma'$) is XXXL, connected, and twistless”.
	\end{theorem}
	
	We believe that the following conjecture holds true:
	
	\begin{conjecture}
		Let $A_{\Gamma}$ be any large-type Artin group. If $\Gamma$ is connected and twistless, then it is fundamental.
	\end{conjecture}

	\subsection{Discussion of proofs}
	We will only discuss the proof of the main theorem Theorem~\ref{thm:mainintro} and some of its generalizations, as the passage from the main theorem to all the applications in rigidity and classification is standard.
	
	Let $A_{\Gamma}$ be a large-type Artin group. Our goal is to compute the automorphism group of the associated intersection of graph $I_{\Gamma}$, eventually prove that under appropriate assumptions, the automorphism group $\aut(I_\Gamma)$ of $I_\Gamma$ is isomorphic to the automorphism group $\aut(D_\Gamma)$ of the associated modified Deligne complex $D_\Gamma$. Then we relate $\aut(D_\Gamma)$ to the automorphism group of the Cayley complexes or automorphism group of $A_\Gamma$. We will only discuss how we relate $\aut(I_\Gamma)$ and $\aut(D_\Gamma)$, which is the most interesting part of the proof.
	
	We recall that the vertices of $I_{\Gamma}$ are in $1-1$ correspondence with $\mathbb{Z}$-subgroups whose centraliser is not virtually abelian. Despite the description of the vertices of $I_\Gamma$ looking uniform, our first step is show the automorphism group of $I_\Gamma$ cannot act transitively on the vertex set of $I_\Gamma$, except when $\Gamma=\Delta_{333}$ is a triangle with all edges labelled by 3. In the special case $\Gamma=\Delta_{333}$, the corresponding Artin group $A_{\Gamma}$ injects with finite index in the mapping class group $\mcg^\pm(\Sigma_5)$ of the $5$-punctured sphere. In particular, the intersection graph $I_{\Gamma}$ is isomorphic to the curve graph $\mathcal{C}_5$ of the $5$-punctured sphere (see Section \ref{SectionI=CurveGraph}). As such, it is vertex-transitive, and the vertices are indistinguishable from one another.
	
	In contrast to what happens in curve graphs, transitivity fails whenever $\Gamma\neq\Delta_{333}$. This lack of symmetry is actually a crucial point for us. It turns out that the graph $I_{\Gamma}$ contains three naturally different types of vertices, that we call vertices of type $T$, of type $D$ and of type $E$. These different types of vertices correspond to the different types of non-virtually abelian centralisers appearing in Proposition \ref{PropositionCentralisers}. It is possible to build connection between these vertices with certain convex subcomplexes of $D_\Gamma$. The simplest case is vertices of type $D$ in $I_\Gamma$, which actually corresponds to certain type of vertices in $D_\Gamma$ (we call them \emph{type 2 vertices}). Vertices of type $T$ in $I_\Gamma$ corresponding to certain collections of trees in $D_\Gamma$ (we call them \emph{standard trees}); and vertices of type $E$ are the most complicated ones, corresponding to more intricate subcomplexes of $D_\Gamma$.

	First we try to distinguish vertices of type $E$ in the intersection graph $I_\Gamma$. To this end, we carefully add certain 2-cells to $I_\Gamma$ such that automorphisms of $I_\Gamma$ also permutes these 2-cells. Now vertices of type $E$ in $I_\Gamma$ can be distinguished from $T$ and $D$ using the local structure of this 2-complex. This is done in Section \ref{SectionIsolatingExoticVertices}, where we prove the following.
	
	\begin{proposition} (=Corollary~\ref{CorollaryExoticSentToExotic})
		Let $A_{\Gamma}$ and $A_{\Gamma'}$ be two large-type Artin groups where both $\Gamma$ and $\Gamma'$ are connected, and $\Gamma \neq \Delta_{333} \neq \Gamma'$. Then every isomorphism $\psi : I_{\Gamma} \rightarrow I_{\Gamma'}$ sends vertices of type $E$ to vertices of type $E$.    
	\end{proposition}
	
	The above proposition shows that every automorphism of $I_{\Gamma}$ reduces to an automorphism of the subgraph $I_{\Gamma}^{TD}$ spanned by the vertices of type $T$ and $D$.
	
	In Section \ref{subsecTD} we strengthen the previous result provided $\Gamma$ is not a tree. Our arguments rely on understanding intersections of standard trees in the Deligne complex. We prove the following:
	
	\begin{proposition} (=Proposition~\ref{PropAutomorphismsPreserveType})
		Let $A_{\Gamma}$ and $A_{\Gamma'}$ be two large-type Artin groups where both $\Gamma$ and $\Gamma'$ are connected but not trees, and $\Gamma \neq \Delta_{333} \neq \Gamma'$. Then every isomorphism $\psi : I_{\Gamma} \rightarrow I_{\Gamma'}$ preserves the type of the vertices ($T$, $D$ or $E$).
	\end{proposition}
	
	A consequence of this proposition is that an automorphism $\psi \in \aut(I_{\Gamma})$ naturally sends (infinite) standard trees of the Deligne complex $D_{\Gamma}$ to other (infinite) standard trees, and similarly, it sends type $2$ vertices of $D_{\Gamma}$ to other type $2$ vertices. This is the kind of information we need to be able to construct a simplicial automorphism of the Deligne complex from $\psi$.
	
	To extract the key property of constructing a simplicial automorphism of $D_\Gamma$ from an automorphism of $I_{\Gamma}$, we then introduce the notion of fundamental subgraphs, and we prove the following general consequence of being fundamental:
	
	\begin{theorem} (=Theorem~\ref{thm:2ndiso}) \label{ThmComparisonMapIntro}
		Let $A_{\Gamma}$ be a large-type Artin group such that $\Gamma$ is connected and does not have leaf vertices. If $\Gamma$ is fundamental, then $\aut(I^{TD}_\Gamma)$ is isomorphic to the group $\aut(D_\Gamma)$ of simplicial automorphisms of the Deligne complex.
	\end{theorem}
	
	We also show in Section~\ref{subsecFund} that the following class of graphs $\Gamma$ are fundamental.
	
	\begin{theorem} (=Proposition~\ref{prop:fund})
		Suppose $\Gamma$ admits a twistless hierarchy terminating in twistless stars. Then $\Gamma$ is fundamental.
	\end{theorem}
	
	Other examples of graphs that are fundamental are given \cite{blufstein2024homomorphisms}, see Section \ref{Secfurtherresult}.
	
	Obtaining various rigidity results from Theorem \ref{ThmComparisonMapIntro} is done throughout Section \ref{SectionProvingMainResults}.

	\section*{Acknowledgements}
	
	The third author is supported by the Postdoc Mobility $\sharp$P500PT$\_$210985 of the Swiss National Science Foundation. The first author is partially supported by a Sloan fellowship and NSF DMS-2305411. The first author thanks Camille Horbez for several helpful discussions.

	\section{Preliminaries}
	
	In this section we recall several known facts about Artin groups. We also introduce objects central to this paper, such as the Deligne complex or the intersection graph. Finally, we prove such basic results about the intersection graph of a large-type Artin group.

	\subsection{General things about Artin groups}
	
	It is a common fact about Artin groups (see \cite{van1983homotopy}) that for any Artin group $A_{\Gamma}$ and any full subgraph $\Gamma' \subseteq \Gamma$ the natural morphism $A_{\Gamma'} \hookrightarrow A_{\Gamma}$ is an injection. Consequently, we will write $A_{\Gamma'}$ to refer to both the Artin group over $\Gamma'$ and the subgroup of $A_{\Gamma}$ spanned by the standard generators associated with the vertices of $\Gamma'$. If $\Gamma'$ is a single edge connecting $a$ and $b$, we will often write $A_{ab}$ instead of $A_{\Gamma'}$.
	
	he subgroups of $A_{\Gamma}$ of the form $A_{\Gamma'}$ are called \emph{standard parabolic subgroups}, and their conjugates are called \emph{parabolic subgroups}. Parabolic subgroups play a central role in the study of Artin groups. One key feature of their combinatorics is the construction of the celebrated (modified) Deligne complex introduced in \cite{charney1995k}.
	
	Thereafter we define this complex in the context of \emph{2-dimensional} Artin groups, which are Artin groups in which for every triangle $T = T(a,b,c)$ in $\Gamma$ we have $\frac{1}{m_{ab}} + \frac{1}{m_{ac}} + \frac{1}{m_{bc}} \leq 1$.
	
	\begin{definition}[Modified Deligne complex]
		\label{def:md}
		Let $A_\Gamma$ be a 2-dimensional Artin group with defining graph $\Gamma$. The \emph{(modified) Deligne complex} $D_\Gamma$ is the geometric realization of the poset (ordered by inclusion) of all cosets of $A_\Gamma$ of the form $gA_{\Gamma'}$, where $\Gamma'$ is either the empty subgraph (in which case $A_{\Gamma'}$ is the trivial subgroup), a vertex of $\Gamma$, or an edge of $\Gamma$. The Artin group $A_{\Gamma}$ act on $D_{\Gamma}$ by left-multiplication, and we denote by $K_{\Gamma}$ the fundamental domain of this action. A vertex of $D_\Gamma$ is of type $i$ if it corresponds to $gA_{\Gamma'}$ such that $\Gamma'$ has $i$ vertices.
	\end{definition}
	
	The Deligne complex is generally well-understood for $2$-dimensional Artin groups, thanks to the following:
	
	\begin{theorem} \label{Thm:DeligneCAT(0)} \cite{charney1995k}
		The Deligne complex $D_{\Gamma}$ associated with any $2$-dimensional Artin group $A_{\Gamma}$ is $2$-dimensional, and it is CAT(0) when given the Moussong metric.
	\end{theorem}
	
	We define the Moussong metric thereafter:
	
	\begin{notation}
		In the Deligne complex $D_{\Gamma}$, we will sometimes denote the vertices $\{1\}$, $\langle a \rangle$ and $A_{ab}$ by $v_{\emptyset}$, $v_a$ and $v_{ab}$ respectively. This will be useful to avoid confusion between vertices of $D_{\Gamma}$ and subgroups of $A_{\Gamma}$.
	\end{notation}
	
	\begin{definition} \label{DefiMoussong}
		Let $A_{\Gamma}$ be a $2$-dimensional Artin group. The \emph{Moussong metric} is the piecewise-Euclidean metric $d$ defined as follows. Let $T_{ab}$ be the base triangle joining $v_{\emptyset}$, $v_a$ and $v_{ab}$. Then we set
		$$d(v_{\emptyset}, v_a) \coloneqq 1, \ \ \angle_{v_a}(v_{\emptyset}, v_{ab}) \coloneqq \frac{\pi}{2}, \ \text{ and } \ \angle_{v_{ab}}(v_{\emptyset}, v_{ab}) \coloneqq \frac{\pi}{2 m_{ab}},$$
		where $m_{ab}$ is the coefficient of the edge $e^{ab} \in V(\Gamma)$. Note that the above information entirely determines the isometry type of $T_{ab}$ as a Euclidean triangle. Then, $d$ is defined on $K_{\Gamma}$ and further on $D_{\Gamma}$ as a piecewise extension of the metrics on the base triangles.
	\end{definition}
	
	\begin{remark} \label{RemFDConvex}
		One can easily show from Definition \ref{DefiMoussong} that the fundamental domain $K_{\Gamma}$ is always convex.
	\end{remark}
	
	Large-type Artin groups are known to have well-behaved parabolic subgroups, in the following sense:
	
	\begin{theorem} \cite[Theorem A]{cumplido2022parabolic} \label{ThmIntersectionParabolics}
		Let $A_{\Gamma}$ be a large-type Artin group. Then the intersection of any family of parabolic subgroups of $A_{\Gamma}$ is a parabolic subgroup of $A_{\Gamma}$.
	\end{theorem}
	
	\begin{definition} \cite[Definition 35]{cumplido2022parabolic} \label{DefiParabolicClosure}
		Let $A_{\Gamma}$ be a large-type Artin group, and let $g \in A_{\Gamma}$. Then there exists a parabolic subgroup $P_g$ that is minimal with respect to the inclusion amongst the parabolic subgroups that contain $g$. $P_g$ is called the \emph{parabolic closure} of $g$.
	\end{definition}
	
	\begin{definition} \label{DefiType}
		The \emph{type} of a parabolic subgroup $g A_{\Gamma'} g^{-1}$ of $A_{\Gamma}$ is the cardinality of the set $|V(\Gamma')|$, that is, the rank of $A_{\Gamma'}$. The \emph{type} of an element $g \in A_{\Gamma}$ is the type of its parabolic closure $P_g$.
	\end{definition}

	As $D_{\Gamma}$ is a simplicial complex with finitely many shapes, the action $A_{\Gamma} \curvearrowright D_{\Gamma}$ is semi-simple, meaning that every element $g \in A_{\Gamma}$ either acts elliptically, with a fixed set $\fix(g)$, or it acts hyperbolically, with a minimal set $\min(g)$. The following lemma describes the behaviour of elliptic elements:
	
	\begin{lemma} \cite[Lemma 8]{crisp2005automorphisms} \label{LemmaCrisp}
		Let $A_{\Gamma}$ be a $2$-dimensional Artin group, and let $g \in A_{\Gamma} \backslash \{1\}$ act elliptically on $D_{\Gamma}$. Then exactly one of the following happens:
		\medskip
		
		\noindent $\bullet$ \underline{$type(g) = 1$.} Equivalently, $g \in h \langle a \rangle h^{-1}$ for some $a \in V(\Gamma)$ and $h \in A_{\Gamma}$. In this case, $\fix(g) = h \fix(a)$ is a tree. We will call such trees \emph{standard trees}.
		\medskip
		
		\noindent $\bullet$ \underline{$type(g) = 2$.} Equivalently, $g \in h A_{ab} h^{-1}$ for some $a, b \in V(\Gamma)$ satisfying $m_{ab} < \infty$ and some $h \in A_{\Gamma}$. In this case, $\fix(g)$ is the single vertex corresponding to the coset $h A_{ab}$.
	\end{lemma}
	
	\begin{remark} \label{RemStandardTrees}
		(1) Standard trees are convex.
		\\(2) Note that while every element of type $1$ is elliptic, elements of type $2$ are elliptic if and only if their parabolic closure $h A_{ab} h^{-1}$ satisfies $m_{ab} < \infty$.
		\\(3) It was showed in \cite[Remark 2.17]{vaskou2023isomorphism} that the standard tree $\fix(a)$ is infinite if and only if $a$ is not an isolated vertex, and does not lie at the tip of an even-labelled leaf in $\Gamma$.
	\end{remark}
	
	\begin{lemma} \cite[Lemmas 2.20 \& 2.21]{vaskou2023isomorphism} \label{LemmaCommuteStabilises}
		Let $g, h \in A_{\Gamma}$ be such that $g$ is elliptic when acting on $D_{\Gamma}$. If $g$ and $h$ commute, then $h$ stabilises $Fix(g)$.
	\end{lemma}
	
	Finally, we briefly prove the following Corollary to Theorem \ref{ThmIntersectionParabolics}, that will be used several times throughout the paper:
	
	\begin{corollary} \label{CorollaryTypeOfIntersections}
		Let $P$ and $P'$ be two parabolic subgroups of $A_{\Gamma}$ such that none contains the other, and suppose that $type(P) = n$ and $type(P') = m$. Then $type(P \cap P') < \min\{n, m\}$.
	\end{corollary}
	
	\begin{proof}
		If $n = m = 1$, the statement follows from the fact that there is no non-trivial element that fixes two distinct standard trees pointwise.
		
		So we suppose without loss of generality that $n \geq 2$. The proof then essentially follows from \cite{cumplido2022parabolic}. In their work, the authors study a simplicial complex called the Artin complex $\mathcal{A}_{\Gamma}$, and show that every parabolic subgroup $P$ of $A_{\Gamma}$ of type $n$ is the pointwise stabiliser of a simplex $\Delta_P$ of codimension $n$ in $\mathcal{A}_{\Gamma}$ (see \cite[Lemma 12]{cumplido2022parabolic}). We now consider the subgroup $Q \coloneqq P \cap P'$, which is a parabolic subgroup of $A_{\Gamma}$ by Theorem \ref{ThmIntersectionParabolics}. By \cite[Lemma 20]{cumplido2022parabolic}, there is a unique maximal simplex $\Delta$ of $\mathcal{A}_{\Gamma}$ such that $Q$ fixes $\Delta$ pointwise. By the above, there are two simplices $\Delta_P$ and $\Delta_{P'}$ of codimension $n$ and $m$ respectively, such that $Q$ fixes both simplices pointwise. By \cite[Lemma 12]{cumplido2022parabolic} again, it can't be that either of these simplices contains the other, or we would have $P \subseteq P'$ or $P' \subseteq P$, contradicting our hypotheses. This means that $Q$ fixes strictly more than just $\Delta_P$ and just $\Delta_{P'}$. It follows that $\Delta$ must have codimension at most $\min\{n, m\} -1$. By \cite[Lemma 12]{cumplido2022parabolic} again, this means that $Q$ is a parabolic subgroup of type at most $\min\{n, m\} -1$. The result follows.
	\end{proof}

	\subsection{Exotic dihedral Artin subgroups and Centralisers}
	
	We first recall a few things regarding dihedral Artin subgroups of exotic type.
	
	\begin{definition} \cite{vaskou2023isomorphism}
		A subgroup $H$ of $A_{\Gamma}$ is called a \emph{dihedral Artin subgroup} if it is abstractly isomorphic to a (non-abelian) dihedral Artin group. The subgroup $H$ is said to be an \emph{exotic} dihedral Artin subgroup if it is not contained in a dihedral Artin parabolic subgroup of $A_{\Gamma}$. Finally, a dihedral Artin subgroup $H$ of $A_{\Gamma}$ is said to be \emph{maximal} if it is not strictly contained in another dihedral Artin subgroup of $A_{\Gamma}$.
	\end{definition}

	\begin{theorem}
		\label{theoremexotic}
		\cite[Theorem D]{vaskou2023isomorphism}
		Let $H$ be a maximal exotic dihedral Artin subgroup of a large-type Artin group $A_{\Gamma}$. Then there are three standard generators $a, b, c \in V(\Gamma)$ satisfying $m_{ab} = m_{ac} = m_{bc} = 3$ such that up to conjugation, $H$ takes the form
		\begin{align*}
			&H = \langle s, t \ | \ stst = tsts \rangle, \\
			&s \coloneqq b^{-1}, \ \ \ \ t \coloneqq babc, \\
			&z \coloneqq stst = tsts = abcabc,
		\end{align*}
		where $z$ generates the centre of $H$. Note that $H = \langle b, abc \rangle$.
	\end{theorem}
	
	The centralisers of elements in large-type Artin groups were completely described in \cite{martin2023characterising}: every non-trivial element $g \in A_{\Gamma}$ has a centraliser that is virtually $\langle g \rangle \times F$ where $F$ is a possibly abelian free group. The non-virtually-abelian centralisers are at the heart of our strategy. We start with the following definition:
	
	\begin{definition}
		We say that a subgroup is $\mathbb{Z} \times F_{\geq 2}$ if it is isomorphic to a direct product $\mathbb{Z} \times F$ where $F$ is a free group of rank at least $2$.
	\end{definition}
	
	We are particularly interested in subgroups that are virtually $\mathbb{Z} \times F_{\geq 2}$, because some of them appear as quasi-isometric invariants for studying large-type Artin groups (see \cite{huang2017quasi}).

	We recall the reader that a (non-abelian) dihedral Artin group always contains a finite index subgroup that is $\mathbb{Z} \times F_{\geq 2}$ (\cite{ciobanu2020equations}). In particular, it is always virtually $\mathbb{Z} \times F_{\geq 2}$ .
	
	\begin{proposition} \label{PropositionCentralisers} (\cite[Remark 3.6]{martin2023characterising}, \cite[Proposition 4.2]{jones2024fixed})
		Suppose that $A_{\Gamma}$ is large-type. Then up to conjugation and permutation of the generators, there are only up to three kinds of centralisers that are virtually $\mathbb{Z} \times F_{\geq 2}$:
		\begin{enumerate}
			\item The centralisers of the form $C(a)$ where $a$ is a standard   generator that is not an isolated vertex, nor the tip of an even-labelled leaf. In this case, $C(a) \cong \langle a \rangle \times F_{\geq 2}$.
			\item The centralisers of the form $C(z_{ab})$ where $z_{ab}$ generates the centre of a dihedral Artin parabolic subgroup $A_{ab}$. In this case, $C(z_{ab})$ is precisely $A_{ab}$.
			\item The centralisers of the form $C(abcabc)$ where $a, b, c \in V(\Gamma)$ satisfy $m_{ab} = m_{ac} = m_{bc} = 3$. In this case, $C(abcabc)$ is an exotic dihedral Artin subgroup. It is isomorphic to the abstract dihedral Artin group $\langle x, y \ | \ xyxy =yxyx \rangle$.
		\end{enumerate}
		All other non-trivial centralisers are isomorphic to $\mathbb{Z}$ or $\mathbb{Z}^2$.
	\end{proposition}
	
	\begin{remark} \label{RemarkFirst}
		(1) When $\Gamma$ is connected and without even-labelled leaves, then every standard generator has a centraliser that is $\mathbb{Z} \times F_{\geq 2}$.
		\\(2) The dihedral Artin subgroup in Proposition \ref{PropositionCentralisers}.3 occur if and only if the large-type Artin group $A_{\Gamma}$ is not of hyperbolic-type, i.e. if and only if $\Gamma$ contains the subgraph $\Delta_{333}$ (\cite[Theorem D]{vaskou2023isomorphism}).
	\end{remark}

	\subsection{The intersection graph.}
	\label{subsec:intersection graph}
	
	Throughout this section we let $A_{\Gamma}$ be a large-type Artin group. The goal is to define the intersection graph associated with $A_{\Gamma}$, which is a quasi-isometric invariant for large-type Artin groups. The original definition of intersection \cite[Definition 10.13 and Definition 10.7]{huang2017quasi} was in terms of certain types of subcomplexes in the Cayley complex of $A_\Gamma$. Here we will introduce a group-theoretical reformulation of this definition, which is easier to work with, and explain how it corresponds to the old definition.
	
	Recall that two subgroups $G_1$ and $G_2$ of $A_\Gamma$ are \emph{commensurable} if $G_1\cap G_2$ is finite index in both $G_1$ and $G_2$. A standard generator $a$ is called \emph{large} if its centraliser is large, i.e. it is not virtually abelian. By Proposition \ref{PropositionCentralisers}, $a$ is large if and only if it is not isolated, nor the tip of an even-labelled leaf.
	
	We say an element in $A_\Gamma$ is \emph{stable} if $g$ is either conjugated to a large standard generator, or $g$ generates the centre of a dihedral Artin parabolic subgroup, or $g$ generates the centre of an exotic maximal dihedral Artin subgroup. In the last case, by Theorem~\ref{theoremexotic}, there exists $a,b,c\in V(\Gamma)$ with $m_{ab}=m_{bc}=m_{ac}=3$ such that $g=h^{-1}(abcabc)^{\pm 1} h$ for some $h\in A_\Gamma$. In the first case, we say the stable element is of \emph{type T}, in the second case, the element is of \emph{type D}, and in the last case, the element is of \emph{type E}. A \emph{stable} $\mathbb Z$-subgroup of $A_\Gamma$ is a subgroup generated by a stable element. The \emph{type} of stable subgroups are defined as the type of the element generating them.
	
	\begin{remark} \label{RemType}
		One should not mix the above notion of type with the notion of type defined in Definition \ref{DefiType}, although stable elements of type $T$ always have type $1$, stable elements of type $D$ have type $2$, and stable elements of type $E$ have type $3$.
	\end{remark}
	
	\begin{corollary} \label{CoroCNVA}
		The centraliser of each stable $\mathbb Z$-subgroup of $A_\Gamma$ is not virtually abelian. Conversely, each $\mathbb Z$-subgroup of $A_\Gamma$ whose centraliser is not virtually abelian is contained in a stable $\mathbb Z$-subgroup of $A_\Gamma$.
	\end{corollary}
	
	\begin{proof}
		For any $\mathbb{Z}$-subgroup $H = \langle h \rangle$ of $A_\Gamma$ we have $C(H) = C(h)$. Thus:
		\medskip
		
		\noindent $\bullet$ If $H$ is stable, then $h$ is a stable element and $C(h)$ is not virtually abelian by Proposition \ref{PropositionCentralisers}. Thus so is $C(H)$.
		\medskip
		
		\noindent $\bullet$ If $C(H)$ is not virtually abelian then $C(h)$ is not virtually abelian, hence $h$ is a non-trivial power of a stable element $h_0$, by Proposition \ref{PropositionCentralisers}. In particular, $H$ is contained in the stable $\mathbb{Z}$-subgroup $\langle h_0 \rangle$.
	\end{proof}

	\begin{lemma}
		\label{lemcommensurableequal}
		If two stable $\mathbb Z$-subgroups of $A_\Gamma$ are commensurable, then they are equal.
	\end{lemma}
	
	\begin{proof}
		First of all, it is easy to see that two stable $\mathbb{Z}$-subgroup of different types amongst $T$, $D$ or $E$ are never commensurable. This directly follows from Remark \ref{RemType}. Let now $H$, $H'$ be two commensurable stable $\mathbb{Z}$-subgroups. By the above, they must have the same type. Now:
		\medskip
		
		\noindent (1) If they are both of type $T$, then $G = \langle g \rangle$ and $H = \langle h \rangle$, where $g$, $h$ are conjugates of standard generators. The fixed sets $\fix(g)$ and $\fix(h)$ are two standard trees, and for any $n \neq 0$ we have $\fix(g^n) = \fix(g)$ and $\fix(h^n) = \fix(h)$ (this follows from Lemma \ref{LemmaCrisp}). By hypothesis, there are some $n, m \neq 0$ such that $g^n = h^m$. Using the above, this gives $\fix(g) = \fix(h)$, which yields $G = H$.
		\medskip
		
		\noindent (2) If they are both of type $D$, then $G = \langle g \rangle$ and $H = \langle h \rangle$, where $g$, $h$ each generate the centre of a dihedral Artin parabolic subgroup of $A_{\Gamma}$. The fixed-sets $Fix(g)$ and $Fix(h)$ are two type $2$ vertices of $D_{\Gamma}$, and for any $n \neq 0$ we have $Fix(g^n) = Fix(g)$ and $Fix(h^n) = Fix(h)$ (this also follows from Lemma \ref{LemmaCrisp}). As above, commensurability implies $G = H$.
		\medskip
		
		\noindent (3) If they are both of type $E$, then $G = \langle g \rangle$ and $H = \langle h \rangle$, where $g$, $h$ each generate the centre of an exotic Artin parabolic subgroup of $A_{\Gamma}$. The minimal sets $Min(g)$ and $Min(h)$ are two subcomplexes of $X_{\Gamma}$, and for any $n \neq 0$, we have $Min(g^n) = Min(g)$ and $Min(h^n) = Min(h)$ (see \cite[Remark 3.31]{vaskou2023isomorphism}). As before, commensurability implies that $Min(g) = Min(h)$. In particular, $g$ and $h$ are hyperbolic elements that act by translations along the same axes. They are both translation of minimal length by maximality of $G$ and $H$, so it must be that $g = h^{\pm1}$. It follows that $G = H$.
	\end{proof}
	
	Let $P_\Gamma$ be the presentation complex of $A_\Gamma$ such that each edge of $P_\Gamma$ is labelled by a standard generator of $A_\Gamma$, and we choose an orientation for each edge of $P_\Gamma$. Let $X_\Gamma$ be the Cayley complex of $A_\Gamma$, i.e. $X_\Gamma$ is the universal cover of $P_\Gamma$, with induced labelling and orientation of edges.
	In \cite{huang2017quasi}, we introduce three types of connected subcomplexes of the Cayley complex $X_\Gamma$ (while these subcomplexes are called ``lines'', they are not exactly homemorphic to $\mathbb R$, but they are quasi-isometric to $\mathbb R$):
	\begin{enumerate}
		\item the first type of subcomplexes are called \emph{diamond lines}, see \cite[Definition 5.1 and Figure 14]{huang2017quasi}; it follows from the definition that for each diamond line $L\subset X_\Gamma$, there is a stable $\mathbb Z$-subgroup of $A_\Gamma$ of type D which stabilizes $L$ and acts cocompactly on $L$; conversely, each stable $\mathbb Z$-subgroup of $A_\Gamma$ of type $D$ stabilizes a diamond line and acts cocompactly on it;
		\item the second type are called \emph{plain lines}, see \cite[Definition 5.4]{huang2017quasi}; it follows from the definition that each diamond line $L\subset X_\Gamma$ is stabilized by a stable $\mathbb Z$-subgroup of type T which acts cocompactly on $L$; conversely, each stable $\mathbb Z$-subgroup of type T stabilizes a plain line and acts cocompactly on it;
		\item the third type are \emph{Coxeter lines}, see \cite[Definition 5.3 and Figure 15]{huang2017quasi}; each Coxeter line $L\subset X_\Gamma$ is stabilized by a stable $\mathbb Z$-subgroup of type E which acts cocompactly on $L$; conversely, each stable $\mathbb Z$-subgroup of type E stabilizes a Coxeter line and acts cocompactly on it.
	\end{enumerate}
	
	A \emph{singular line} in $X_\Gamma$ is either a diamond line, or a plain line, or a Coxeter line. Two singular lines are \emph{parallel} if they have finite Hausdorff distance from each other. It follows from \cite[Corollary 2.4]{MR2867450}  that if two singular lines are parallel, and $Z_1$ and $Z_2$ are stable $\mathbb Z$-subgroups of $A_\Gamma$ stabilizing each of them, then $Z_1$ and $Z_2$ are commensurable. By Lemma~\ref{lemcommensurableequal}, $Z_1=Z_2$. This together with the above discussion imply there is a 1-1 correspondence between the collection of parallel classes of singular lines in $X_\Gamma$, and the collection of stable $\mathbb Z$-subgroups of $A_\Gamma$. This allows us to reformulate \cite[Definition 10.13]{huang2017quasi} in the large-type case as follows.
	
	\begin{definition} \label{DefIntersectionGraphI} [Intersection graph: definition I]
		Suppose $A_\Gamma$ is a large-type Artin group. The \emph{intersection graph} of $A_\Gamma$, denoted $I_\Gamma$, is defined to be the graph whose vertices are the stable $\mathbb Z$-subgroups of $A_\Gamma$, and two vertices are adjacent, if the associated $\mathbb Z$-subgroups commute.
	\end{definition}
	
	One can actually define the edges in the intersection graph $I_{\Gamma}$ in Definition \ref{DefIntersectionGraphI} with a seemingly looser yet equivalent condition (see Lemma \ref{LemmaCommuteIIFVirtuallyCommute}).
	
	\begin{lemma} \label{LemmaCentralisersOfPowers}
		Let $g \in A_{\Gamma}$ be a stable element. Then for any $n \neq 0$, we have $C(g) = C(g^n)$.
	\end{lemma}
	
	\begin{proof}
		The result is standard if $g$ is of type $T$ or of type $D$ (for the reader, it can be proved easily from Proposition \ref{PropositionCentralisers}). If $g$ is of type $E$, this is \cite[Remark 3.31]{vaskou2023isomorphism}.
	\end{proof}
	
	\begin{lemma} \label{LemmaCommuteIIFVirtuallyCommute}
		Two stable $\mathbb{Z}$-subgroups commute if and only if they have finite index subgroups that commute.
	\end{lemma}
	
	\begin{proof}
		The “only if” is trivial, so we prove the “if”. We call the stable $\mathbb{Z}$-subgroups $G$ and $H$. Recall that there are two elements $g, h \in A_{\Gamma}$ such that $G = \langle g \rangle$ and $H = \langle h \rangle$. By hypothesis, there are subgroups $\langle g^n \rangle$ and  $\langle h^m \rangle$ that commute, for some $n, m \neq 0$. In particular, $h^m \in C(g^n)$. By Lemma \ref{LemmaCentralisersOfPowers}, we have $C(g) = C(g^n)$, so $h^m$ actually commutes with $g$. By Lemma \ref{LemmaCentralisersOfPowers} again, we have $C(h) = C(h^m)$, so $g$ and $h$ must commute.
	\end{proof}
	
	\begin{lemma} \label{LemmaCommutingStable}
		If $G$ and $H$ are two distinct stable $\mathbb{Z}$-subgroups, then $G \cap H = \{1\}$. If additionally $G$ and $H$ commute, then $\langle G, H \rangle \cong \mathbb{Z}^2$.
	\end{lemma}
	
	\begin{proof}
		If $G \cap H$ is strictly bigger than $\{1\}$, then it has finite index in both $G$ and $H$. In particular, both subgroups are commensurable, and hence equal by Lemma \ref{lemcommensurableequal}. The second point is clear.
	\end{proof}
	
	\begin{definition} \label{DefIntersectionGraphII} [Intersection graph: definition II]
		The \emph{intersection graph} $I_{\Gamma}$ of $A_{\Gamma}$ is the graph whose vertices are the centralisers $C(g)$ that are virtually $\mathbb{Z} \times F_{\geq 2}$, and two vertices $C(g)$ and $C(h)$ are adjacent if $C(g) \cap C(h) \cong \mathbb{Z}^2$.
	\end{definition}
	
	\begin{lemma} \label{LemmaTwoDefinitionsCoincide}
		The two definitions of the intersection graph $I_{\Gamma}$ coincide.
	\end{lemma}
	
	\begin{proof}
		For the purpose of this proof, we call $I_{\Gamma}$ the intersection graph (definition I) and $J_{\Gamma}$ the intersection graph (definition II). The isomorphism $\Lambda : I_{\Gamma} \rightarrow J_{\Gamma}$ is given by
		$$\Lambda : G \mapsto C(G), \ \ \ \Lambda^{-1} : C(g) \mapsto \langle g \rangle$$
		We first show that $\Lambda$ and $\Lambda^{-1}$ are indeed inverse of each other's. Let $G \in I_{\Gamma}$. Then there is a stable element $g$ such that $G = \langle g \rangle$. We obtain
		$$\Lambda^{-1} (\Lambda(G)) = \Lambda^{-1} (C(G)) = \Lambda^{-1} (C(\langle g \rangle)) \overset{\ref{LemmaCentralisersOfPowers}} = \Lambda^{-1} (C(g)) = \langle g \rangle = G.$$
		Conversely,
		$$\Lambda (\Lambda^{-1} (C(g))) = \Lambda (\langle g \rangle) = C(\langle g \rangle) \overset{\ref{LemmaCentralisersOfPowers}} = C(g).$$
		Now suppose that $G, H \in I_{\Gamma}$ are adjacent. Then $G$ and $H$ are distinct yet commute. In particular, $C(G)$ and $C(H)$ both contain $\langle G, H \rangle$, which is isomorphic to $\mathbb{Z}^2$ by Lemma \ref{LemmaCommutingStable}. We want to show that $C(G) \cap C(H)$ is exactly $\mathbb{Z}^2$.
		
		By Proposition \ref{PropositionCentralisers}, we know that each of $C(G)$ or $C(H)$ is virtually $\mathbb{Z} \times F_{\geq 2}$. Notice that $C(G) \cap C(H) = C_{C(G)}(C(H)) = C_{C(G)}(C(h))$, where $H = \langle h \rangle$.
		
		If one of $C(G)$ or $C(H)$, say $C(G)$ is abstractly isomorphic to a dihedral Artin group, then we can use the classification of centralisers in dihedral Artin groups (see \cite[Remark 3.6]{martin2023characterising} or \cite[Proposition 4.2]{jones2024fixed}) to deduce that $C_{C(G)}(C(h))$ is isomorphic to one of $\mathbb{Z}$, $\mathbb{Z}^2$ or $C(G)$ itself. The first and the third case cannot happen as it would require $G$ and $H$ to intersect non-trivially. Thus $C(G) \cap C(H) = C_{C(G)}(C(h)) \cong \mathbb{Z}^2$.
		
		If none of $C(G)$ or $C(H)$ is a dihedral Artin group, then both take the form $\mathbb{Z} \times F_{\geq 2}$. Since $h \in C(G) = G \times F$, we can write $h = (h_1, h_2)$ with $h_1 \in G$ and $h_2 \in F$. Let now $k \in C(G) \cap C(H)$. Because $k \in C(G)$, we can write $k = (k_1, k_2)$ with $k_1 \in G$ and $k_2 \in F$. Because $k \in C(H)$, we know that $hk = kh$, so in particular, $h_2 k_2 = k_2 h_2$. This forces $h_2$ and $k_2$ to belong to a common maximal $\mathbb{Z}$-subgroup $K$ of $F$. This shows that $C(G) \cap C(H)$ is contained in $G \times K \cong \mathbb{Z}^2$, as wanted.
		
		We conclude that $\Lambda(G)$ and $\Lambda(H)$ are adjacent in $J_{\Gamma}$.
		\medskip
		
		Conversely, if $C(g), C(h) \in J_{\Gamma}$ are adjacent, then $C(g) \cap C(h) \cong \mathbb{Z}^2$. Note that any $\mathbb{Z}^2$-subgroup of a centraliser that is non-virtually abelian is always such that one of its two factors is contained in the centre of the centraliser, because such centralisers always split as a product $\mathbb{Z} \times F$ where the first factor is precisely their centre. In particular, $C(g) \cap C(h)$ contains both $\langle g^n \rangle$ and $\langle h^m \rangle$ for some $n, m \neq 0$. This means $g^n$ commute with $h^m$, and consequently, $g$ commute with $h$ (use Lemma \ref{LemmaCommuteIIFVirtuallyCommute}). Finally, $G = \langle g \rangle$ and $H = \langle h \rangle$ are adjacent in $I_{\Gamma}$.
	\end{proof}
	
	\begin{lemma} \label{LemmaDisjointStandardTreesDontCommute}
		Let $T_g$ and $T_h$ be two standard trees in $D_{\Gamma}$, with respective local groups $G = \langle g \rangle$ and $H = \langle h \rangle$. If $T_g \cap T_h = \emptyset$, then $g$ and $h$ do not commute.
	\end{lemma}
	
	\begin{proof}
		Let $(x, y) \in T_g \times T_h$ be a couple of points minimising the distance between $T_g$ and $T_h$. We suppose that $g$ and $h$ commute, so that $h$ stabilises $T_g$ by Lemma \ref{LemmaCommuteStabilises}. Consequently, the point $h \cdot x$ also belongs to $T_g$. Note that $h \cdot y = y$ by construction. By \cite[Lemma 3.8]{martin2024tits}, the angle that the geodesics $[x, y]$ and $[y, g \cdot x]$ make at $y$ is exactly $\pi$. It follows that the union $[x, y] \cup [y, g \cdot x]$ is a geodesic of $D_{\Gamma}$. As it connects two points of $T_g$ and $T_g$ is convex (Remark \ref{RemStandardTrees}.(1)), it must be that the whole geodesic is contained in $T_g$. This contradicts $y \notin T_g$, which shows that $g$ and $h$ in fact cannot commute.
	\end{proof}
	
	\begin{lemma} \label{LemmaPossibleEdges}
		The only possible edges in $I_{\Gamma}$ are of the form $T-D$, $T-E$ or $E-E$.
	\end{lemma}
	
	\begin{proof}
		We need to show that there can't be edges of the form $T-T$, $D-D$ or $D-E$.
		\medskip
		
		\noindent Suppose that $G, H \in I_{\Gamma}$ are both of type $T$ yet adjacent vertices, and write $G = \langle g \rangle$, $H = \langle h \rangle$. By hypothesis, $G$ and $H$ commute (and thus, so do $g$ and $h$). In particular, $g$ stabilises the standard tree $Fix(H)$ by Lemma \ref{LemmaCommuteStabilises}. We now exhibit a contradiction. Indeed, it can't be that $Fix(G)$ and $Fix(H)$ are disjoint as $g$ and $h$ commute (use Lemma \ref{LemmaDisjointStandardTreesDontCommute}), but it also can't be that $Fix(G)$ and $Fix(H)$ intersect, as \cite[Lemma 2.7]{blufstein2024homomorphisms} would then prove that we can't have $\langle G, H \rangle \cong \mathbb{Z}^2$.
		\medskip
		
		\noindent Suppose that $G, H \in I_{\Gamma}$ are both of type $D$ yet adjacent vertices, and write $G = \langle g \rangle$, $H = \langle h \rangle$. As before, since $g$ and $h$ commute, then $g$ stabilises the type $2$ vertex $Fix(H)$ by Lemma \ref{LemmaCommuteStabilises}. This forces $Fix(G) = Fix(H)$, which then forces $G = H$, giving a contradiction.
		\medskip
		
		\noindent Suppose that $G, H \in I_{\Gamma}$ are such that $G$ is type $D$ and $H$ is type $E$, and write $G = \langle g \rangle$, $H = \langle h \rangle$. Once again, $g$ and $h$ commute, which gives a contradiction as elements of type $2$ such as $g$ cannot commute with hyperbolic elements such as $h$ (see \cite[Lemma 2.29]{vaskou2023isomorphism}).
	\end{proof}
	
	\begin{definition}
		We define $I_{\Gamma}^{TD}$ to be the subgraph of $I_{\Gamma}$ spanned by the vertices of type $T$ and $D$.
	\end{definition}
	
	We already know that $I_{\Gamma}^{TD}$ can be defined in two equivalent ways (see Definition \ref{DefIntersectionGraphI}, Definition \ref{DefIntersectionGraphII} and Lemma \ref{LemmaTwoDefinitionsCoincide}). Thereafter we give yet an equivalent definition of $I_{\Gamma}^{TD}$, that will be used extensively in Section \ref{SectionIntersectionSubgraph}. The following definition is actually equivalent to Crisp's $\Theta$-graph \cite{crisp2005automorphisms}. This is because Crisp considered large-type triangle-free Artin groups, which have no exotic dihedral Artin subgroups. In particular, their intersection graphs have no vertices of type $E$, in which case $I_{\Gamma}^{TD} = I_{\Gamma}$.
	
	\begin{definition} \label{DefIntersectionGraphIII} [Intersection subgraph $I_{\Gamma}^{TD}$: definition III]
		The graph $I_{\Gamma}^{TD}$ is the graph whose vertices are the set $\mathcal{T}$ of unbounded standard trees in $D_{\Gamma}$ and the set $\mathcal{D}$ of type $2$ vertices in $D_{\Gamma}$, and there is an edge between $t \in \mathcal{T}$ and $v \in \mathcal{D}$ if and only if $v \in t$ in $D_{\Gamma}$.
	\end{definition}
	
	\begin{lemma} \label{LemmaThreeDefinitionsCoincide}
		The three definitions of the intersection subgraph $I_{\Gamma}^{TD}$ coincide.
	\end{lemma}
	
	\begin{proof}
		We already know that the first two definitions coincide by Lemma \ref{LemmaTwoDefinitionsCoincide}. The rest of the argument is standard.
	\end{proof}
	
	Before finishing this section, we say a few words regarding loops in $I_{\Gamma}^{TD}$ and in $I_{\Gamma}$.
	
	\begin{definition}
		The \emph{girth} of a graph is the infimum of the length of embedded cycles.
	\end{definition}
	
	\begin{corollary} \label{CoroBipartite}
		$I_{\Gamma}^{TD}$ is a bipartite graph. In particular, its girth is even and at least $6$.
	\end{corollary}
	
	\begin{proof}
		By Lemma \ref{LemmaPossibleEdges}, the only possible edges in $I_{\Gamma}^{TD}$ are of the form $T - D$. This means $I_{\Gamma}^{TD}$ is bipartite, and its girth is even.
		
		Suppose that $I_{\Gamma}^{TD}$ contains a $4$-cycle, which necessarily takes the form $(T - D - T - D)$. Geometrically, this means there is two distinct type $2$ vertices in $D_{\Gamma}$ with two distinct standard trees connecting them. But standard trees are convex and $D_{\Gamma}$ is CAT(0) hence uniquely geodesic (see \ref{RemStandardTrees}.(1) and Theorem \ref{Thm:DeligneCAT(0)}), so this yields a contradiction.
	\end{proof}
	
	\begin{lemma} \label{LemmaNeighoursOfExotic}
		Let $H$ be a stable $\mathbb{Z}$-subgroup of type $E$. We know that up to conjugation, $H = \langle h \rangle$ where $h = abcabc$ and $a, b, c \in V(\Gamma)$ are such that $m_{ab} = m_{ac} = m_{bc} = 3$. Then $H$, seen as a vertex of $I_{\Gamma}$, is contained in the $5$-cycle described in Figure \ref{fig5cycle}.
		
		It follows that every vertex of type $E$ is contained in a $5$-cycle, and every vertex of type $E$ has a neighbour of type $T$ and a neighbour of type $E$.
	\end{lemma}
	
	\begin{proof}
		The first paragraph of the lemma directly follows from Theorem \ref{theoremexotic}. Let $z_{ab}$ be the element $ababab$ generating the centre of $A_{ab}$. We exhibit a $5$-cycle containing $H$:
		
		\begin{figure}[H]
			\centering
			\includegraphics[scale=1]{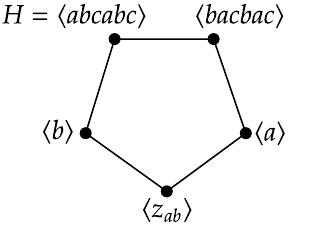}
			\caption{A $5$-cycle of $I_{\Gamma}$ containing $H$.}
			\label{fig5cycle}
		\end{figure}
		
		We need to check that the edges from Figure \ref{fig5cycle} actually exist. Note that:
		\begin{itemize}
			\item The elements $a$ and $b$ commute with $z_{ab}$ as the latter is central in $A_{ab}$.
			\item One can easily see using the relations that the elements $abcabc$ and $b$ commute. The same goes for $bacbac$ and $a$. See \cite[Lemma 4.2]{vaskou2023isomorphism} for details.
			\item The elements $abcabc$ and $bacbac$ commute because they act as pure translations on a common Euclidean plane of $D_{\Gamma}$. See \cite[Proof of Lemma 4.12 \& Figure 17]{vaskou2023isomorphism} for details.
		\end{itemize}
	\end{proof}
	
	In Section 3, we will prove that the $5$-cycle described in Figure \ref{fig5cycle} is induced.
	
	\section{On the intersection graph of the $(3,3,3)$-Artin group}
	
	In this section we establish several properties of the intersection graph of the $(3,3,3)$-Artin group for later uses, via the topological interpretation of this graph as the curve graph of $5$-punctured sphere.

	\subsection{Small cycles in the curve graph of $5$-punctured sphere}
	
	\label{secsmallcycle}
	Let $\Sigma$ be a finite type surface, with possibly punctures and boundary components. The \emph{extended mapping class group} of $\Sigma$, denoted by $\mcg^\pm(\Sigma)$, is made of isotopy classes of homeomorphisms of $\Sigma$ that is identity on $\partial\Sigma$, and the isotopy is taken within such classes of homeomorphisms. The \emph{mapping class group} $\mcg(\Sigma)$ is defined to be the finite index subgroup corresponding to orientation-preserving homeomorphisms. A simple closed curve on $\Sigma$ is \emph{essential} is it is not null-homotopic, and it is not homotopic to a puncture or a boundary component. Let $\mathcal S(\Sigma)$ be the collection of homotopy classes of essential simple closed curves in $\Sigma$. There is a natural action $\mcg(\Sigma)\curvearrowright \mathcal S(\Sigma)$.
	Let $\alpha\neq \beta\in S(\Sigma)$. 
	Take any two (oriented) transverse simple closed curves representing $\alpha$ and $\beta$. 
	Then the \emph{algebraic intersection number} of $\alpha$ and $\beta$ is the sum of the indices of intersection points, where index is $+1$ when the orientation of the intersection coincides with a given orientation of $\Sigma$ and is $-1$ otherwise. The algebraic intersection number does not depend on the choice of the two representatives.
	The \emph{geometric intersection number} of $\alpha$ and $\beta$, denoted by $I(\alpha,\beta)$, is defined to be the minimal possible cardinality of $a\cap b$, with $a\in \alpha$ and $b\in \beta$. We write $\alpha\cap\beta=\emptyset$ if $I(\alpha,\beta)=0$. We write $\alpha \perp_0 \beta$ if $I(\alpha,\beta)=2$ and their algebraic intersection number is $0$.
	
	The \emph{curve graph} for $\Sigma$, is the graph whose vertices are elements in $\mathcal S(\Sigma)$, and two vertices are adjacent if the associated two classes can be represented by disjoint curves.
	
	For an positive integer $n$,  let $\Sigma_n$ be the $n$-punctured sphere. We will be especially interested in $\Sigma_5$. We denote the curve graph of $\Sigma_5$ by $\mathcal C_5$.  
	
	Let $P\mcg(\Sigma_5)$ be the finite index subgroup of $\mcg(\Sigma_5)$ fixing each puncture of $\Sigma_5$. An essential simple closed curve $b$ in $\Sigma_5$ \emph{encloses} two punctures $p$ and $p'$ of $\Sigma_5$ if they are contained in the same connected component $C$ of $\Sigma_5\setminus b$ and $C$ does not contain any other punctures of $\Sigma_5$. In that case we also say that the class of $b$ in $\mathcal S(\Sigma_5)$ \emph{encloses} $p$ and $p'$.
	\begin{lemma}
		\label{lemmatransitive}
		The action of $\mcg(\Sigma_5)$ on $\mathcal C_5$ is vertex transitive and edge transitive. Moreover, given two punctures $p$ and $p'$ of $\Sigma_5$, the group $P\mcg(\Sigma_5)$ acts transitively on the set of elements in $\mathcal C(\Sigma_5)$ that enclose $p$ and $p'$.
	\end{lemma}
	
	\begin{proof}
		For each essential simple curve $b\in \Sigma_5$, we know $\Sigma_5\setminus b$ has two components, one homeomorphic to $\Sigma_3$ and one homeomorphic to $\Sigma_2$. Thus by the change of coordinate principle (see e.g. \cite[Chapter 1.3]{farb2011primer}), the action $\mcg(\Sigma_5)\curvearrowright \mathcal C_5$ is vertex transitive. Other parts of the lemma can be proved by a similar argument.
	\end{proof}
	
	We recall several basic facts about $\mathcal C_5$.
	\begin{lemma}
		\label{lem:girth 5}
		The girth of $\mathcal C_5$ is 5.
	\end{lemma}
	
	\begin{proof}
		This is a well-known fact. We sketch a proof for the convenience of the reader.
		It is not hard to see $\mathcal C_5$ does contain an embedded 5-cycle. $3$-cycles of $\mathcal C_5$ can be ruled out as the associated mapping class group does not contain a subgroup isomorphic to $\mathbb Z^3$.
		Now we rule out embedded 4-cycles. Suppose $\{\beta_i\}_{i=1}^4$ are consecutive vertices of an embedded 4-cycle. Take representatives $b_i\in \beta_i$ for $1\le i\le 4$ such that $b_i\cap b_{i+1}=\emptyset$ with the index $i$ understood as modulo 4. Let $D$ be the closure of the component of $\Sigma_5\setminus b_1$ that contains three punctures. Then $b_i\subset D$ for $i=2,4$. Note that the subgroup of $\mcg(\Sigma_5)$ generated by Dehn twists along $\beta_2,\beta_4$ contain an element $g$ which is pseudo-Anosov on $D$. As the Dehn twist along $\beta_3$ commute with $g$, it can be represented by a homeomorphism whose support is outside $D$. This implies that $\beta_1=\beta_3$, contradicting that the 4-cycle is embedded.
	\end{proof}
	
	The following is a consequence of a work of Feng Luo \cite{luo2000}.
	\begin{lemma} \label{LemmaOneOrbitOf5Cycles}
		Let $\omega$ be an embedded 5-cycle in $\mathcal C_5$ with consecutive vertices corresponding to the homotopy classes $\{\beta_i\}_{i\in \mathbb Z/5\mathbb Z}$. Then there exists a labelling of the 5 punctures $\{p_i\}_{i\in \mathbb Z/5\mathbb Z}$ by elements in $\mathbb Z/5 \mathbb Z$ (this labelling depends on $\omega$) and five simple arcs $\{\gamma_i\}_{i\in \mathbb Z/5\mathbb Z}$ with each $\gamma_i$ connecting $p_i$ and $p_{i+1}$ such that
		\begin{enumerate}
			\item the concatenation of $\{\gamma_i\}_{i\in \mathbb Z/5\mathbb Z}$ gives a simple closed curve on the sphere;
			\item for each $i$, the boundary of a small enough neighborhood of $\gamma_i$ gives a simple closed curve representing one of $\{\beta_i\}_{i=1}^5$.
		\end{enumerate}
	\end{lemma}
	
	\begin{proof}
		We take representative $b_i\in \beta_i$ for each $i$ such that the cardinality of $b_i\cap b_j$ is minimized among all representatives whenever $i\neq j$ (this is a possible by putting a finite volume complete hyperbolic {metric} on $\Sigma_5$ and choosing $b_i$ to be the unique geodesic representative in the homotopy class).
		We claim $\beta_i\cap\beta_{i+2}\neq\emptyset$. Note that  $\Sigma_5\setminus b_i$ has two components $C_1$ and $C_2$ with $C_1$ homeomorphic to the sphere with 4 punctures $\Sigma_4$ and $C_2$ homeomorphic to $\Sigma_3$. Then $b_{i+1},b_{i+2}$ are contained in $C_1$ and $b_{i+1}\cap b_{i+2}=\emptyset$. As $C_1\setminus b_{i+1}$ is a union of two components, each of which is homeomorphic to $\Sigma_3$, we must have $b_{i+1}\cap b_{i+2}\neq\emptyset$, which is a contradiction. Thus the claim is proved. Now it follows from \cite[Lemma 4.2]{luo2000} that $\beta_i\perp_0\beta_{i+2}$ for each $i$.
		
		For each $\beta_i$, let $a_i$ be a simple arc in the $\Sigma_3$-component of $\Sigma_5\setminus b_i$ connecting the two punctures in this component. This definition determines $a_i$ up to isotopy rel endpoints. Note that $a_i\cap a_{i+1}=\emptyset$ as $b_1\cap b_{i+1}=\emptyset$. Moreover we can arrange $a_i\cap a_{i+2}=\emptyset$ except at endpoints for each $i$ as $\beta_i\perp_0\beta_{i+2}$.
		Let $\gamma_1=a_1$, $\gamma_2=a_3$, $\gamma_3=a_5$, $\gamma_4=a_4$ and $\gamma_5=a_2$. Then the lemma follows.
	\end{proof}

	\subsection{The $(3,3,3)$ Artin group and $\mcg^\pm(\Sigma_5)$}
	
	Throughout this section we let $A_{\Gamma}$ be the Artin group where $\Gamma$ is a triangle with coefficients $(3, 3, 3)$.
	
	\begin{notation}
		We let the north pole $N$ and the south pole $S$ be two of the punctures of $\Sigma_5$. The three other punctures of $\Sigma_5$ will be supposed to lie on the equator and be denoted by $X$, $Y$ and $Z$. The set of punctures is called $\mathcal{P}$. The equator is cut by $X,Y$ and $Z$ into three segments $\gamma_{X,Y}$, $\gamma_{Y,Z}$ and $\gamma_{Z,X}$.
		Let $T_{XY}$ be the half-twist permuting $X$ and $Y$ such that $T_{XY}$ is supported in a small neighbourhood of $\gamma_{X,Y}$. Similarly we define $T_{YZ}$ and $T_{ZX}$.
	\end{notation}
	
	Let $\Gamma$ be the presentation graph consisting of a triangle with labels $(3,3,3)$.
	The following is proved in \cite{charney2005automorphism} - see \cite[Figure 3]{charney2005automorphism} and the discussion before that figure. Note that the definition of mapping class group in \cite{charney2005automorphism} corresponds to our extended mapping class group.
	\begin{proposition}
		\label{propiso}
		There exists a natural injection $F : A_{\Gamma} \rightarrow \mcg^\pm(\Sigma_5)$ of the $(3, 3, 3)$-Artin group $A_{\Gamma} = \langle a, b, c \rangle$ into the mapping class group of the $5$-punctured sphere $\Sigma_5$, defined by
		\begin{align*}
			&F(a) = T_{XY}, \\
			&F(b) = T_{YZ}, \\
			&F(c) = T_{ZX}.
		\end{align*}
		The image of this injection has index $120$ in $\mcg^\pm(\Sigma_5)$.
	\end{proposition}
	
	We recall the following result about automorphisms of the $(3, 3, 3)$ Artin group $A_{\Gamma}$:
	
	\begin{theorem} \cite[Theorem A]{vaskou2023automorphisms}
		Let $A_{\Gamma}$ be the $(3, 3, 3)$ Artin group. Then $\aut(A_{\Gamma}) = \aut_{\Gamma}(A_{\Gamma})$, i.e. the automorphism group of $A_{\Gamma}$ is generated by the conjugations, the graph automorphisms, and the global inversion.
	\end{theorem}

	Following the previous theorem, the homomorphism map $F : A_{\Gamma} \rightarrow \mcg^\pm(\Sigma_5)$ extends to an injective homormophism $\overline{F} : \aut(A_{\Gamma}) \rightarrow \mcg^\pm(\Sigma_5)$ as follows:
	\begin{itemize}
		\item If $\varphi_g$ is the conjugation by $g$, then $\overline{F}(\varphi_g) \coloneqq F(g)$;
		\item If $\sigma$ is a graph automorphism of order $2$, then $\overline{F}(\sigma)$ is a rotation of $\Sigma_5$ of angle $\pi$. E.g. if $\sigma : a \mapsto a, b \mapsto c, c \mapsto b$, then it is the rotation along the axis between $Z$ and the midpoint of $X$ and $Y$;
		\item If $\sigma$ is a graph automorphism of order $3$, then $\overline{F}(\sigma)$ is a rotation of $\Sigma_5$ around the axis $(N, S)$. E.g. if $\sigma : a \mapsto b, b \mapsto c, c \mapsto a$, then it is a rotation of angle $(2\pi/3)$;
		\item If $\iota$ is the global inversion, then $\iota$ is the reflection of $\Sigma_5$ along the plane that contains $X, Y, Z$.
	\end{itemize}

	\begin{lemma} \label{lemH}
		Let $H \leq \mcg^\pm(\Sigma_5)$ be the subgroup consisting of the homeomorphisms that preserve the two sets of punctures $\{X, Y, Z\}$ and $\{N, S\}$. Then $\overline{F}(\aut(A_{\Gamma})) = H$.
	\end{lemma}
	
	\begin{proof}
		It is easy to see that $[\mcg^\pm(\Sigma_5) : H] = 10$ because $|S_5| / (|S_3| \cdot |S_2|) = 120 / (6 \cdot 2) = 10$. By definition, $\overline{F}(\aut(A_\Gamma))\subset H$. We have an inclusion of subgroups:
		$$F(A_{\Gamma}) \ \overset{12} \leq \ \overline{F}(\aut(A_{\Gamma})) \ \overset{n} \leq \ H \ \overset{10} \leq \ \mcg^\pm(\Sigma_5),$$
		where the indices of each inclusion are written accordingly. Since $[\mcg^\pm(\Sigma_5) : F(A_{\Gamma})] = 120$, we obtain $n = 1$, thus the lemma follows.
	\end{proof}

	\subsection{Intersection graph and curve graph} \label{SectionI=CurveGraph}
	
	Let $\Gamma$ be a triangle with labels $(3, 3, 3)$. Let $I_\Gamma$ be the associated intersection graph (Definition~\ref{DefIntersectionGraphI}). Recall that two subgroups $G_1$ and $G_2$ of $A_\Gamma$ are \emph{commensurable}, if $G_1\cap G_2$ is finite index in both $G_1$ and $G_2$. 
	
	By Corollary \ref{CoroCNVA}, the intersection graph $I_\Gamma$ can also be characterised as follows: vertices of $I_\Gamma$ are in 1-1 correspondence with commensurability classes of $\mathbb Z$ subgroups of $A_\Gamma$ whose centralisers are not virtually abelian, and two vertices are adjacent if the associated commensurability classes commute, which means that each class has a representative such that the two representatives commute.

	\begin{lemma}
		\label{lemcommensurable}
		Vertices of $\mathcal C_5$ are in 1-1 correspondence with commensurability classes of $\mathbb Z$ subgroups of $\mcg^\pm(\Sigma_5)$ whose centralisers are not virtually abelian, and two vertices are adjacent if the associated classes commute. 
	\end{lemma}
	
	\begin{proof}
		By a criterion of Ivanov \cite[Theorem 7.5B]{ivanov2002mapping}, there is an appropriate finite index subgroup $G$ of $\mcg^\pm(\Sigma_5)$ such that an element $g\in G$ is a non-trivial power of some Dehn twist along an element in $\mathcal S(\Sigma_5)$ if and only if the centre of the centraliser of $g$ in $G$ is isomorphic to $\mathbb Z$ and is not equal to the centraliser of $g$ in $G$. Given a $\mathbb Z$ subgroup $L$ of $\mcg^\pm(\Sigma_5)$ with its centraliser not virtually abelian, then the centraliser of $L\cap G$ in $G$ is not virtually abelian. Recall that $G$ and $A_\Gamma$ share isomorphic finite index subgroups (Proposition~\ref{propiso}), and we know that if a $\mathbb Z$-subgroup of $A_\Gamma$ has its centraliser in $A_\Gamma$ being not virtually abelian, then its centraliser has a finite index subgroup isomorphic to $\mathbb Z \times F_k$ with $k \geq 2$ (see Proposition \ref{PropositionCentralisers}).
		
		Thus the centraliser of $L\cap G$ in $G$ has a finite index subgroup isomorphic to $\mathbb Z \times F_k$ with $k \geq 2$. By Ivanov's criterion as before, $L$ is commensurable to a $\mathbb Z$-subgroup of $\mcg^\pm(\Sigma_5)$ generated by a Dehn twist along an element $\alpha\in \mathcal S(\Sigma_5)$. Thus each $\mathbb Z$ subgroup of $\mcg^\pm(\Sigma_5)$ with non-virtually-abelian centraliser gives rise to an element in $\alpha\in \mathcal S(\Sigma_5)$. Conversely, each element in $\mathcal S(\Sigma_5)$ gives a Dehn twist which generates a $\mathbb Z$-subgroup whose centraliser is not virtually abelian. Thus the lemma follows.
	\end{proof}
	
	\begin{definition}
		\label{defFstar}
		Let $F:A_\Gamma\to \mcg^\pm(\Sigma_5)$ be the injective homomorphism with finite index image as in Proposition~\ref{propiso}. Then $F$ induces a 1-1 correspondence between $\mathbb Z$-subgroups of $A_\Gamma$ with non-virtually-abelian centralisers and $\mathbb Z$-subgroups of $\mcg^\pm(\Sigma_5)$ with non-virtually-abelian centralisers. Thus by the previous discussion and Lemma~\ref{lemcommensurable}, $F$ induces an isomorphism $F_*:I_\Gamma\to \mathcal C_5$.
	\end{definition}
	
	\section{Reducing to automorphisms of $I_{\Gamma}^{TD}$} \label{SectionIsolatingExoticVertices}
	
	Throughout this section we let $A_{\Gamma}$ be a large-type Artin group of rank at least $3$.
	
	\subsection{Pieces and $5$-cycles}
	
	Let $A_{\Gamma'}$ be a (standard) parabolic subgroup of $A_{\Gamma}$. A stable $\mathbb{Z}$-subgroup of $A_{\Gamma'}$ is inherently a stable $\mathbb{Z}$-subgroup of $A_{\Gamma}$. Conversely, a stable $\mathbb{Z}$-subgroup $\langle g \rangle$ of $A_{\Gamma}$ is a stable $\mathbb{Z}$-subgroup of $A_{\Gamma'}$ if and only if $\langle g \rangle$ (or equivalently $g$) is contained in $A_{\Gamma'}$.
	
	Additionally, two stable $\mathbb{Z}$-subgroup of $A_{\Gamma'}$ commute in $A_{\Gamma'}$ if and only if they commute in $A_{\Gamma}$, given that the natural map $A_{\Gamma'} \hookrightarrow A_{\Gamma}$ is injective by \cite{van1983homotopy}.
	
	Consequently, the above means that the intersection graph $I_{\Gamma'}$ associated with $A_{\Gamma'}$ naturally injects in the intersection graph $I_{\Gamma}$ associated with $A_{\Gamma}$. We shall then see $I_{\Gamma'}$ as a subgraph of $I_{\Gamma}$. A corollary of this discussion is the following:
	
	\begin{corollary} \label{CoroVertexOfIntersectionSubgraph}
		A stable $\mathbb{Z}$-subgroup $H = \langle h \rangle$ is a vertex of $I_{\Gamma'}$ if and only if $H$ (equivalently $h$) is contained in $A_{\Gamma'}$.
	\end{corollary}

	If $g A_{\Gamma'} g^{-1}$ is a parabolic subgroup of $A_{\Gamma}$, we denote by $g I_{\Gamma'} g^{-1}$ the associated intersection subgraph of $I_{\Gamma}$. Note that $I_{\Gamma'}$ and $g I_{\Gamma'} g^{-1}$ are naturally isomorphic.
	
	\begin{definition}
		A subgraph of $I_{\Gamma}$ of the form $g I_{\Gamma'} g^{-1}$ for some $g \in A_{\Gamma}$, where $\Gamma'$ is a $3$-cycle subgraph of $\Gamma$ is called a \emph{piece} of $I_{\Gamma}$. We say that $g I_{\Gamma'} g^{-1}$ is a \emph{k-piece} if it is a piece and its girth as a graph is $k$.
	\end{definition}
	
	Our next goal is to prove Lemma \ref{LemmaPentagonInSinglePiece} below. We start with the following:
	
	\begin{lemma} \label{Lemma5CyclesContainExotic}
		Every $5$-cycle in $I_{\Gamma}$ contains a vertex of type $E$.
	\end{lemma}
	
	\begin{proof}
		Let $\gamma$ be a cycle of $I_{\Gamma}$ that does not contain a vertex of type $E$. Then $\gamma$ is contained in $I_{\Gamma}^{TD}$. In particular, $\gamma$ has an even number of edges by Corollary \ref{CoroBipartite}.
	\end{proof}
	
	\begin{lemma} \label{LemmaTypeEinUniquePiece}
		A vertex of type $E$ is contained in a single $5$-piece.
	\end{lemma}
	
	\begin{proof}
		Existence was proved in Lemma \ref{LemmaNeighoursOfExotic}, so we only prove unicity. Let $v_e$ be our vertex, and let $g I_{\Gamma'} g^{-1}$ and $h I_{\Gamma''} h^{-1}$ be two $5$-pieces containing $v_e$. Up to conjugation, $v_e$ corresponds to a stable $\mathbb{Z}$-subgroup of the form $\langle abcabc \rangle$ for appropriate $a, b, c \in V(\Gamma)$. By Corollary \ref{CoroVertexOfIntersectionSubgraph}, the element $abcabc$ is both contained in $g A_{\Gamma'} g^{-1}$ and in $h A_{\Gamma''} h^{-1}$. Since $abcabc$ has type $3$, the intersection $g A_{\Gamma'} g^{-1} \cap h A_{\Gamma''} h^{-1}$ has type at least $3$. But $g A_{\Gamma'} g^{-1}$ and $h A_{\Gamma''} h^{-1}$ both have type $3$, so by Corollary \ref{CorollaryTypeOfIntersections} it must be that one of the two parabolic subgroup contains the other. Since they have the same type, we obtain $g A_{\Gamma'} g^{-1} = h A_{\Gamma''} h^{-1}$. Using Corollary \ref{CoroVertexOfIntersectionSubgraph} again, this proves that $g I_{\Gamma'} g^{-1} = h I_{\Gamma''} h^{-1}$, as wanted.
	\end{proof}
	
	\begin{lemma} \label{LemmaNeighboursInSamePiece}
		If a vertex $v \in I_{\Gamma}$ of type $D$ or $E$ belongs to a piece then every neighbour $v'$ of $v$ belongs to that piece.
	\end{lemma}
	
	\begin{proof}
		By Lemma \ref{LemmaPossibleEdges}, the only possible edges in $I_{\Gamma}$ are of type $T - D$, of type $T - E$ or of type $E - E$. 
		
		If $v$ is of type $E$, then up to conjugating it corresponds to the stable $\mathbb{Z}$-subgroup $\langle abcabc \rangle$ for some appropriate generators $a, b, c \in V(\Gamma)$. By Corollary \ref{CoroVertexOfIntersectionSubgraph}, $v$ belongs to $I_{\Gamma'}$ thus $abcabc \in A_{\Gamma'}$. As $v$ and $v'$ are adjacent, $v'$ corresponds to a stable $\mathbb{Z}$-subgroup $\langle h \rangle$ and $h \in C(abcabc)$. By Theorem \ref{theoremexotic}, we have $C(abcabc) =  \langle b, abc \rangle \subseteq A_{\Gamma'} = \langle a, b, c \rangle$. Thus $h \in A_{\Gamma'}$. By Corollary \ref{CoroVertexOfIntersectionSubgraph}, this means that $v' \in I_{\Gamma'}$.
		
		If $v$ is of type $D$, we can proceed as above. Since $v'$ must be of type $T$, we obtain $h \in C(z_{ab}) = A_{ab} \subseteq A_{\Gamma'}$, so we must also have $h \in A_{\Gamma'}$, i.e. $v' \in I_{\Gamma'}$.
	\end{proof}
	
	\begin{lemma} \label{LemmaPentagonInSinglePiece}
		Each $5$-cycle in $I_{\Gamma}$ is contained in a single $5$-piece.
	\end{lemma}
	
	\begin{proof}
		It is enough to prove that each $5$-cycle in $I_{\Gamma}$ is contained is some $5$-piece, as unicity then follows from the fact that each $5$-cycle contains a vertex of type $E$ (Lemma \ref{LemmaNeighoursOfExotic}) and that each vertex of type $E$ is contained in a single $5$-piece (Lemma \ref{LemmaTypeEinUniquePiece}).
		
		Let $\gamma$ be a $5$-cycle of $I_{\Gamma}$. By Lemma \ref{Lemma5CyclesContainExotic}, we know that $\gamma$ contains a vertex $v_e$ of type $E$. By Lemma \ref{LemmaPossibleEdges}, we also know that the only possible edges in $I_{\Gamma}$ are of type $T - D$, of type $T - E$ or of type $E - E$. We label the other vertices of $\gamma$ by $v_1, v_2, v_3, v_4$, and we proceed case by case, according to the following picture:
		
		\begin{figure}[H]
			\centering
			\includegraphics[scale=0.71]{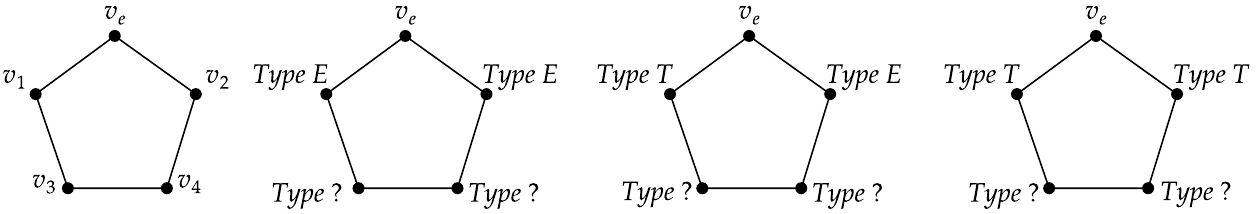}
			\caption{From left to right: The general case, Case 1, Case 2, and Case 3.}
		\end{figure}
		
		Note that the neighbours $v_1$ and $v_2$ of $v_e$ are always contained in $I_{\Gamma'}$ by Lemma \ref{LemmaNeighboursInSamePiece}. We will use Lemma \ref{LemmaNeighboursInSamePiece} extensively.
		\medskip
		
		\noindent \underline{Case 1: $v_1$ and $v_2$ are of type $E$.} By Lemma \ref{LemmaNeighboursInSamePiece} the vertices $v_3$ and $v_4$ also belong to $I_{\Gamma'}$, so we are done. 
		\medskip
		
		\noindent \underline{Case 2: $v_1$ is of type $T$ and $v_2$ are of type $E$.} By Lemma \ref{LemmaNeighboursInSamePiece} we know that $v_4$ belongs to $I_{\Gamma'}$. We also know by Lemma \ref{LemmaPossibleEdges} that $v_3$ is either of type $D$ or $E$, and that $v_4$ is either of type $T$ or $E$. If $v_4$ is of type $E$ then we can use Lemma \ref{LemmaNeighboursInSamePiece} to obtain that $v_3$ also belongs to $I_{\Gamma'}$, and we are done.
		
		The last case to consider is when $v_4$ is of type $T$, and we want to show that $v_3$ also belongs to $I_{\Gamma'}$. So we now assume that $v_4$ is of type $T$, and let $\langle g \rangle$ be the stable $\mathbb{Z}$-subgroup corresponding to $v_3$. If $g \in A_{\Gamma'}$, then $v_3 \in I_{\Gamma'}$ by Corollary \ref{CoroVertexOfIntersectionSubgraph}, and we are done. So we suppose that $g \notin A_{\Gamma'}$, and we consider the subgroup $P \coloneqq P_g \cap A_{\Gamma'}$. Note that $P$ is a parabolic subgroup by Theorem \ref{ThmIntersectionParabolics}. Note that none of $P_g$ or $A_{\Gamma'}$ contains the other, thus we can use Corollary \ref{CorollaryTypeOfIntersections}: $type(P) < \min\{type(g), 3\} = type(g)$. There are two cases:
		\medskip
		
		\noindent $\bullet$ If $v_3$ is of type $D$, then $type(P) < 2$ hence $P$ has rank at most $1$. In particular, it is cyclic. This contradicts the fact that $C(g)$ contains the two distinct stable $\mathbb{Z}$-subgroups associated with $v_1$ and $v_4$, because both vertices are adjacent to $v_3$.
		\medskip
		
		\noindent $\bullet$ If $v_3$ is of type $E$, then $P$ has rank at most $2$. It cannot be that $P$ has rank at most $1$, by the same reason as above. So $P$ has rank $2$, and in particular, there is a type $D$ vertex $v_d$ in $I_{\Gamma}$ that corresponds to $P$ (i.e. the corresponding stable $\mathbb{Z}$-subgroup is the centre of $P$). Call $H_1$ and $H_4$ the two stable $\mathbb{Z}$-subgroups corresponding to $v_1$ and $v_4$ respectively. On one hand, $H_1, H_4 \subseteq P_g$ because $v_1$ and $v_4$ are attached to $v_3$. On the other hand, we know that $H_1, H_4 \subseteq A_{\Gamma'}$ as $v_1, v_4 \in I_{\Gamma'}$ (see Corollary \ref{CoroVertexOfIntersectionSubgraph}). This means we actually have $H_1, H_4 \subseteq P$. Altogether, the string $(v_1, v_3, v_4, v_d)$ forms a $4$-cycle in the $5$-piece associated with $P_g$, which contradicts Definition \ref{defFstar} and Lemma~\ref{lem:girth 5}.
		\medskip
		
		\noindent \underline{Case 3: $v_1$ and $v_2$ are of type $T$.} By Lemma \ref{LemmaPossibleEdges}, $v_3$ and $v_4$ can only be or type $D$ or $E$. But by the same Lemma, there is no $D - D$ or $D - E$ edges in $I_{\Gamma}$. This forces $v_3$ and $v_4$ to both be of type $E$. We are now in the same situation as in Case 2, but with $v_3$ playing the role of $v_e$. So we conclude similarly.
	\end{proof}

	\subsection{Characterising vertices of type $E$}
	
	\begin{definition}
		Let $v$ be a vertex of $I_{\Gamma}$. We define the \emph{modified link} of $v$ as the graph $\widetilde{lk}(v)$ defined as follows:
		\begin{itemize}
			\item The vertices of $\widetilde{lk}(v)$ are the vertices of $lk(v)$ that are contained in a $5$-cycle of $I_{\Gamma}$ ;
			\item We put an edge between two vertices $w$ and $w'$ if there is      an embedded $5$-cycle in $I_{\Gamma}$ that contain $w$, $v$ and         $w'$ as consecutive vertices.
		\end{itemize}
		If $G$ is a subgraph of $I_{\Gamma}$, we write $\widetilde{lk}_G(v)$ to denote the modified link of $v$ in $G$ (where we require both vertices and $5$-cycles to be contained in $G$).
	\end{definition}
	
	The following lemma directly follows from the previous definition:
	
	\begin{lemma} \label{LemmaDiscreteModLink}
		Let $v \in I_{\Gamma}$. If $v$ is not contained in any $5$-cycle, then $\widetilde{lk}(v)$ is empty.
	\end{lemma}
	
	\begin{lemma} \label{LemmaLinkOfExoticIsRestricted}
		Let $v_e$ be a vertex of type $E$, and up to conjugation let $I_{\Gamma'}$ be the unique $5$-piece containing it (see Lemma \ref{LemmaTypeEinUniquePiece}). Then $\widetilde{lk}(v) = \widetilde{lk}_{I_{\Gamma'}}(v)$.
	\end{lemma}
	
	\begin{proof}
		First of all, by Lemma \ref{LemmaNeighboursInSamePiece}, every neighbour $v$ or $v_e$ also belongs to $I_{\Gamma'}$. Let us now consider an edge in $\widetilde{lk}(v)$. By definition, this edge corresponds to a $5$-cycle $\gamma$ containing the three vertices $v$, $v_e$ and $v'$, where $v$ and $v'$ are two neighbours of $v_e$. On the one hand, Lemma \ref{LemmaPentagonInSinglePiece} implies that $\gamma$ is contained in a single $5$-piece. On the second hand, this piece has to be $I_{\Gamma'}$ because it is the unique one that contains $v_e$ (see Lemma \ref{LemmaTypeEinUniquePiece}). This proves our edge already belonged to $\widetilde{lk}_{I_{\Gamma'}}(v)$.
	\end{proof}

	Recall that the \emph{Farey graph} is a simplicial graph whose vertices correspond to rational numbers in $\mathbb Q\cup \{\infty\}$. Given two vertices $x = a/c$ and $y = b/d$ written in irreducible forms, i.e. $a,b,c,d\in \mathbb Z\cup \{\infty\}$ and $\gcd(a,c)=\gcd(b,d)=1$ (the irreducible forms of $\infty$ is defined to be $1/0$ and $-1/0$), we define $x$ and $y$ are adjacent if $ad-bc=\pm 1$. See below for a picture of Farey graph, which we take from \cite{luo2000}.
	
	\begin{figure}[H]
		\centering
		\includegraphics[scale=0.75]{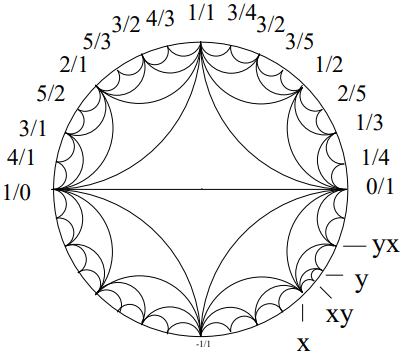}
		\caption{Farey graph.}
	\end{figure}

	Note that the Farey graph can also be viewed as the 1-skeleton of tessellation of the hyperbolic plane by ideal triangles, (the vertices of the Farey graph are contained in the boundary at infinity). It follows from this description that the link of each vertex in Farey graph is a bi-infinite line, which is a fact we will use several times later.
	
	\begin{lemma} \label{LemmaModLinkve}
		Let $I_{\Gamma'}$ be a $5$-piece. For every vertex $v \in I_{\Gamma'}$, the modified link $\widetilde{lk}(v,I_{\Gamma'})$ is isomorphic to the Farey graph.
	\end{lemma}
	
	\begin{proof}
		By Definition~\ref{defFstar}, $I_{\Gamma'}$ and the curve graph $\mathcal C_5$ of the 5-punctured sphere are isomorphic, so it suffices to prove the same statement for $\mathcal C_5$. Let $\alpha$ be the isotopy class of simple closed curves corresponding to $v$. Take a representative $a\in \alpha$, and let $\mathcal D_3$ be the disk with three punctures in $\Sigma_5$ bounded by $a$. As vertices in $\lk_{\mathcal C_5}(v)$ can be represented by simple closed curves disjoint from $a$, and these representatives must be contained in $\mathcal D_3$, we know vertices in $\lk_{\mathcal C_5}(v)$ are in 1-1 correspondence with isotopy classes of essential simple closed curves in $\mathcal D_3$. Let $(\mathcal C(\mathcal D_3),\perp_0)$ be the graph whose vertices are isotopy classes of essential simple closed curves in $\mathcal D_3$, and edges correspond to the relation $\perp_0$ defined in the beginning of Section~\ref{secsmallcycle}. Given $v_1,v_2$ adjacent to $v$ in $\mathcal C_5$ represented by classes $\alpha_1$ and $\alpha_2$, viewed as vertices in $(\mathcal C(\mathcal D_3),\perp_0)$. If $\alpha_1\perp_0 \alpha_2$, then $v_1,v,v_2$ are contained in an embedded 5-cycle in $\mathcal C_5$. Conversely, by the first paragraph of the proof of Lemma~\ref{LemmaOneOrbitOf5Cycles}, if $v_1,v,v_2$ are consecutive vertices in an embedded 5-cycle in $\mathcal C_5$, then $\alpha_1\perp_0\alpha_2$. Thus $\widetilde{\lk}_{I_{\Gamma'}}(v)$ is isomorphic to $(\mathcal C(\mathcal D_3),\perp_0)$. Note that $(\mathcal C(\mathcal D_3),\perp_0)$ is isomorphic to $(\mathcal C(\Sigma_4),\perp_0)$ with $\Sigma_4$ being the sphere with 4 punctures, and it is a classical observation due to Max Dehn that $(\mathcal C(\Sigma_4),\perp_0)$ is isomorphic to the Farey graph. Hence the lemma follows.
	\end{proof}

	\begin{remark} \label{RemarkModifiedLinkIn5Pieces}
		Let $I_{\Gamma'}$ be a $5$-piece. Because $I_{\Gamma'}$ is vertex transitive by Lemma \ref{lemmatransitive}, then for every $v \in I_{\Gamma'}$ the modified link $\widetilde{lk}_{I_{\Gamma'}}(v)$ is the same. Using Lemma \ref{LemmaModLinkve}, $\widetilde{lk}_{I_{\Gamma'}}(v)$ is actually isomorphic to the Farey graph.
	\end{remark}
	
	\begin{definition}
		Let $G$ be a graph. For a vertex $v \in G$, the \emph{star} $st(v)$ is the full subgraph of $G$ spanned by $v$ and all the vertices that are adjacent to $v$. 
	\end{definition}
	
	\begin{lemma} \label{LemmaIntersectionOfPieces}
		Let us consider two distinct pieces of $I_{\Gamma}$, with no assumption on their girth. Up to conjugation, these pieces take the form $I_{\Gamma'}$ and $g I_{\Gamma''} g^{-1}$ for some $g \in A_{\Gamma}$ and some (non-necessarily distinct) triangular subgraphs $\Gamma', \Gamma''$ of $\Gamma$. Then exactly one of the following happens:
		\begin{enumerate}
			\item $I_{\Gamma}$ and $g I_{\Gamma''} g^{-1}$ do not intersect ;
			\item $I_{\Gamma} \cap g I_{\Gamma''} g^{-1}$ is a single vertex $v_t$ of type $T$ ;
			\item $I_{\Gamma} \cap g I_{\Gamma''} g^{-1}$ is the star $st(v_d)$ of a vertex $v_d$ of type $D$.
		\end{enumerate}
	\end{lemma}
	
	\begin{proof}
		We consider the parabolic subgroups $A_{\Gamma'}$ and $g A_{\Gamma''} g^{-1}$ associated with $I_{\Gamma'}$ and $g I_{\Gamma''} g^{-1}$ respectively. By Corollary \ref{CoroVertexOfIntersectionSubgraph}, a vertex $v$ of $I_{\Gamma}$ belongs to $Q \coloneqq I_{\Gamma'} \cap g I_{\Gamma''} g^{-1}$ if and only if its corresponding stable $\mathbb{Z}$-subgroup $H = \langle h \rangle$ is such that $h \in P \coloneqq A_{\Gamma'} \cap g A_{\Gamma''} g^{-1}$. By Theorem \ref{ThmIntersectionParabolics}, the subgroup $P$ is also parabolic. The parabolic subgroups $A_{\Gamma'}$ and $g A_{\Gamma''} g^{-1}$ are distinct and both have type $3$. Consequently, none contains the other one and we can use Corollary \ref{CorollaryTypeOfIntersections}: the rank of $P$ is at most $2$. There are two possibilities:
		\medskip
		
		\noindent $\bullet$ If $P$ has rank $1$, then $h$ is an element of type $T$ by Remark \ref{RemType}, and the corresponding stable $\mathbb{Z}$-subgroup $H$ is the unique vertex in $Q$.
		\medskip
		
		\noindent $\bullet$ If $P$ has rank $2$, then $h$ is an element of type $T$ or $D$, by Remark \ref{RemType}. The subgraph $Q$ contains a single vertex $v_d$ of type $D$: the one that corresponds to the centre of $P$.  Note that $H \subseteq P$, so either $H = P$ corresponds to $v_d$, or $H$ corresponds to a vertex of type $T$ that's adjacent to $v_d$.
	\end{proof}
	
	\begin{lemma} \label{LemmaModLinkvd}
		If $v_d \in I_{\Gamma'}$ is a vertex of type $D$ in a $5$-piece, then $\widetilde{lk}(v_d)$ is isomorphic to the Farey graph.
	\end{lemma}
	
	\begin{proof}
		Up to conjugation, let $I_{\Gamma'}$ be any $5$-piece containing $v_d$. By Remark \ref{RemarkModifiedLinkIn5Pieces} we already know that $\widetilde{lk}_{I_{\Gamma'}}(v_d)$ is isomorphic to the Farey graph. 
		
		Let $v$ be any vertex in $\widetilde{lk}(v_d)$. By Lemma \ref{LemmaPossibleEdges}, $v$ must be of type $T$. Note that $v$ also belongs to $I_{\Gamma'}$ by Lemma \ref{LemmaNeighboursInSamePiece}.
		
		Let now $e = (v, v')$ be any edge in $\widetilde{lk}(v_d)$. We claim that $e$ already existed in $\widetilde{lk}_{I_{\Gamma'}}(v_d)$. This edge corresponds to a $5$-cycle $\gamma$ in $I_{\Gamma}$ that contains the vertices $v$, $v_d$ and $v'$. By Lemma \ref{LemmaPentagonInSinglePiece}, $\gamma$ is contained in a single $5$-piece $g I_{\Gamma''} g^{-1}$. By Lemma~\ref{LemmaIntersectionOfPieces}, $I_{\Gamma'}\cap gI_{\Gamma''}g^{-1}=st(v_d)$. In what follows, we will construct a graph isomorphism $\varphi':I_{\Gamma'}\to  gI_{\Gamma''}g^{-1}$ such that $\varphi'|_{st(v_d)}$ is the identity map (though we do not require $\varphi'$ extends to an automorphism of $I_\Gamma$). This will imply the claim, hence finishes the proof of the lemma.
		
		The standard parabolic subgroup $A_{\Gamma'}$ is generated by three standard generators $a, b, c \in V(\Gamma)$. Up to further conjugating by an element of $A_{\Gamma'}$, we can assume that $v_d$ corresponds to the dihedral standard parabolic subgroup $A_{ab} \coloneqq \langle a, b \rangle$. The parabolic subgroup $g A_{\Gamma''} g^{-1}$ is generated by the elements $g d g^{-1}$, $g e g^{-1}$ and $g f g^{-1}$ for some standard generators $d, e, f \in V(\Gamma)$ satisfying $m_{de} = m_{df} = m_{ef} = 3$.
		
		As $v_d \in g I_{\Gamma''} g^{-1}$, then using Corollary \ref{CoroVertexOfIntersectionSubgraph} we know that the parabolic subgroup $A_{ab}$ is contained in the parabolic subgroup $g A_{\Gamma''} g^{-1}$. By \cite[Theorem 1.1]{blufstein2023parabolic}, it follows that $A_{ab}$ is also a parabolic subgroup of $g A_{\Gamma''} g^{-1}$. In particular, there exists some $h \in g A_{\Gamma''} g^{-1}$ such that $a = h (g d g^{-1}) h^{-1}$ and $b = h (g e g^{-1}) h^{-1}$.
		
		To construct the graph isomorphism $\varphi'$ as above, by Definition~\ref{DefIntersectionGraphII}, it is enough to exhibit an isomorphism $\varphi : A_{\Gamma'} \rightarrow g A_{\Gamma''} g^{-1}$ that restricts to the identity on $A_{ab}$.
		The isomorphism is given by
		$$\varphi(a) = a, \ \ \varphi(b) = b, \ \ \varphi(c) = (hg) f (hg)^{-1}.$$
	\end{proof}
	
	\begin{lemma} \label{LemmaModLinkvt}
		We assume that $\Gamma$ is connected, and does not simply consist of a $(3, 3, 3)$ cycle. If $v_t \in I_{\Gamma'}$ is vertex of type $T$ in a $5$-piece, then $\widetilde{lk}(v_t)$ is a graph that strictly contains a subgraph isomorphic to the Farey graph (in particular, it is not isomorphic to the Farey graph).
	\end{lemma}
	
	\begin{proof}
		Let $I_{\Gamma'}$ be any $5$-piece containing $v_t$. Once again by Remark \ref{RemarkModifiedLinkIn5Pieces} we know that $\widetilde{lk}_{I_{\Gamma'}}(v_t)$ is isomorphic to the Farey graph. We want to show that $\widetilde{lk}(v_t)$ is not isomorphic to a Farey graph plus some discrete set of points. We start with the following:
		\medskip
		
		\noindent \underline{Claim:} $v_t$ is contained in infinitely many $5$-pieces.
		\medskip
		
		\noindent \underline{Proof of the claim:} By hypothesis $\Gamma$ contains a $(3, 3, 3)$-cycle $\Gamma'$, and we name the associated standard generators $a$, $b$ and $c$. Up to conjugation, we suppose that $v_t$ corresponds to the stable $\mathbb{Z}$-subgroup $\langle a \rangle$. By hypothesis again, $\Gamma$ is not just a $(3, 3, 3)$ triangle. Together with $\Gamma$ being connected, this means there is a fourth standard generator $d$ that is attached to either $a$, $b$ or $c$.
		
		If $d$ is attached to $a$, we let $z_{ad}$ be the element generating the centre of $A_{ad}$. Then for all $k \in \mathbb{Z}$, the pieces $z_{ad}^k I_{\Gamma'} z_{ad}^{-k}$ respectively associated with the parabolic subgroups $z_{ad}^k A_{\Gamma'} z_{ad}^{-k}$ are all distinct, and they all contain $v_t$, as $z_{ad}^k$ commutes with $a$.
		
		If $d$ is attached to $b$, we let $\Delta_{ab}$ be the Garside element corresponding to $a$ and $b$, and we let $z_{bd}$ be the element generating the centre of $A_{bd}$. Let $Z \coloneqq \Delta_{ab} z_{bd} \Delta_{ab}^{-1}$. Then for all $k \in \mathbb{Z}$, the pieces $Z^k I_{\Gamma'} Z^{-k}$ respectively associated with the parabolic subgroups $Z^k A_{\Gamma'} Z^{-k}$ are all distinct, and they all contain $v_t$, as $Z$ commutes with $a$. The case where $d$ is attached to $c$ is similar. This finishes the proof of the claim.
		\medskip
		
		We now consider two pieces $I_{\Gamma'}$ and $g I_{\Gamma''} g^{-1}$ containg $v_t$, for some appropriate $g \in A_{\Gamma}$. By Lemma \ref{LemmaIntersectionOfPieces}, there are two possibilities:
		\medskip
		
		\noindent $\bullet$ Either $I_{\Gamma'} \cap g I_{\Gamma''} g^{-1}$ is exactly $v_t$. Then $\widetilde{lk}_{I_{\Gamma'}}(v_t)$ and $\widetilde{lk}_{g I_{\Gamma''} g^{-1}}(v_t)$ are disjoint. This means that $\widetilde{lk}(v_t)$ contains two copies of a Farey graph. In particular, it cannot be that $\widetilde{lk}(v_t)$ is a Farey graph, as a Farey graph does not strictly contain a copy of itself - this is a consequence of the fact that each edge in the Farey graph is contained in exactly two 3-cycles of the Farey graph.
		\medskip
		
		\noindent $\bullet$ Or $I_{\Gamma'} \cap g I_{\Gamma''} g^{-1}$ is the star of a vertex $v_d$ of type $D$. If $\widetilde{lk}(v_t)$ was a Farey graph, then for every vertex $v$ adjacent to $v_t$ the link $lk_{\widetilde{lk}(v_t)}(v)$ would be an bi-infinite line. But the links $lk_{\widetilde{lk}_{I_{\Gamma'}}(v_t)}(v_d)$ and $lk_{\widetilde{lk}_{g I_{\Gamma''} g^{-1}}(v_t)}(v_d)$ are two disjoint components of $lk_{\widetilde{lk}(v_t)}(v_d)$ and they both are infinite line. This means that $lk_{\widetilde{lk}(v_t)}(v_d)$ is strictly more than a single infinite line, and thus $\widetilde{lk}(v_t)$ is not a Farey graph.
	\end{proof}
	
	\begin{corollary} \label{CorollaryModifiedLinks}
		We assume that $\Gamma$ is connected, and does not simply consist of a $(3, 3, 3)$-cycle.
		Let $v$ be any vertex of $I_{\Gamma}$. Then:
		\begin{enumerate}
			\item If $v$ is of type $T$, then $\widetilde{lk}(v)$ strictly contains a Farey graph if $v$ belongs to a $5$-piece, and it is empty otherwise.
			\item If $v$ is of type $D$, then $\widetilde{lk}(v)$ is a Farey graph if $v$ belongs to a $5$-piece, and it is empty otherwise.
			\item If $v$ is of type $E$, then $\widetilde{lk}(v)$ is a Farey graph.
		\end{enumerate}
	\end{corollary}
	
	\begin{proof}
		This follows from Lemma \ref{LemmaModLinkvt}, Lemma \ref{LemmaModLinkvd} and Lemma \ref{LemmaModLinkve}.
	\end{proof}
	
	\begin{proposition} \label{PropExoticIIFSomeCondition}
		We assume that $\Gamma$ is connected, and does not simply consist of a $(3, 3, 3)$-cycle.
		Let $v \in I_{\Gamma}$ be any vertex. Then $v$ is of type $E$ if and only if it has two neighbours $v'$ and $v''$ such that $\widetilde{lk}_{I_{\Gamma}}(v')$ strictly contains a Farey graph, and $\widetilde{lk}_{I_{\Gamma}}(v'')$ is a Farey graph.
	\end{proposition}
	
	\begin{proof}
		(“Only if”) Let $v$ be of type $E$. We know by Lemma \ref{LemmaNeighoursOfExotic} that $v$ has a neighbour $v'$ of type $T$ and a neighbour $v''$ of type $E$. Then Corollary \ref{CorollaryModifiedLinks} gives the desired result.
		
		(“If”) If $v$ is of type $T$, then by Lemma \ref{LemmaPossibleEdges} any neighbour $v'$ is necessarily of type $D$ or $E$. By Corollary \ref{CorollaryModifiedLinks}, this means that $\widetilde{lk}_{I_{\Gamma}}(v')$ is a Farey graph or it is empty. In particular, $v$ has no neighbour whose modified link strictly contains a Farey graph. So $v$ is not of type $T$.
		
		If $v$ is of type $D$, then by Lemma \ref{LemmaPossibleEdges} any neighbour $v'$ is necessarily of type $T$. By Corollary \ref{CorollaryModifiedLinks}, this means that $\widetilde{lk}_{I_{\Gamma}}(v')$ either strictly contains a Farey graph, or it is empty. In particular, $v$ has no neighbour whose modified link is a Farey graph. So $v$ is not of type $D$.
		
		We conclude that $v$ is of type $E$.
	\end{proof}
	
	Given that the condition described in Proposition \ref{PropExoticIIFSomeCondition} is purely graphical, we directly obtain the following result:
	
	\begin{corollary} \label{CorollaryExoticSentToExotic}
		Let $A_{\Gamma_1}$ and $A_{\Gamma_2}$ be two large-type Artin groups where both $\Gamma_1$ and $\Gamma_2$ are connected, and neither $\Gamma_1$ nor $\Gamma_2$ consists of just a $(3, 3, 3)$-cycle. If there exists an isomorphism $\psi : I_{\Gamma_1} \rightarrow I_{\Gamma_2}$, then $\psi$ send the vertices of type $E$ onto vertices of type $E$.
	\end{corollary}

	\section{$\aut(I_\Gamma)$ versus $\aut(D_\Gamma)$} \label{SectionIntersectionSubgraph}
	
	The main goal of this section is to prove that, under certain assumptions, the automorphism group of the intersection graph is isomorphic to the automorphism group of the modified Deligne complex. This is carried out in several steps: first we prove an extra property of $\aut(I_\Gamma)$ in Section~\ref{subsecTD}, which sets the stage for us to describe a procedure to transfer information from $\aut(I_\Gamma)$ to $\aut(D_\Gamma)$ in Section~\ref{subsecFund}. In Section~\ref{subsecComparison} we compare different automorphism groups.
	
	\subsection{Preserving $T$ and $D$}
	\label{subsecTD}
	
	Let $A_{\Gamma_1}$ and $A_{\Gamma_2}$ be two large-type Artin groups with intersection graphs $I_{\Gamma_1}$ and $I_{\Gamma_2}$, and assume that both defining graphs are connected and do not consist of a single $(3, 3, 3)$-cycle. In Corollary \ref{CorollaryExoticSentToExotic}, we have proved that any isomorphism $\psi : I_{\Gamma_1} \rightarrow I_{\Gamma_2}$ sends the vertices of type $E$ to vertices of type $E$. In particular, $\psi$ reduces to an isomorphism $\overline{\psi} : I_{\Gamma_1}^{TD} \rightarrow I_{\Gamma_2}^{TD}$. The goal of this section is to strengthen this result and prove the following:
	
	\begin{proposition} \label{PropAutomorphismsPreserveType}
		Let $A_{\Gamma_1}$ and $A_{\Gamma_2}$ be as above, and suppose that neither $\Gamma_1$ nor $\Gamma_2$ is a tree. Then any isomorphism $\psi : I_{\Gamma_1} \rightarrow I_{\Gamma_2}$ preserves the type of vertices ($T$, $D$ or $E$).
	\end{proposition}
	
	Throughout the section, we will write $A_{\Gamma}$ to denote any of the two Artin groups $A_{\Gamma_1}$ or $A_{\Gamma_2}$.
	
	The graph $I_{\Gamma}^{TD}$ is the subgraph of the intersection graph $I_{\Gamma}$ that is spanned by the vertices of type $T$ and $D$, i.e. forgetting about the vertices of type $E$. By Lemma \ref{LemmaThreeDefinitionsCoincide}, we can freely choose any of the three definitions for $I_{\Gamma}^{TD}$ (see Definitions \ref{DefIntersectionGraphI}, \ref{DefIntersectionGraphII} and \ref{DefIntersectionGraphIII}).
	
	By Corollary \ref{CoroBipartite}, we know that the splitting $\mathcal{T} \sqcup \mathcal{D}$ gives a bipartition for the vertex set of $I_{\Gamma}^{TD}$, where $\mathcal{T}$ corresponds to the set of (unbounded) standard trees in $D_{\Gamma}$, and $\mathcal{D}$ corresponds to the set of type $2$ vertices in $D_{\Gamma}$. By Lemma \ref{LemmaThreeDefinitionsCoincide}, the vertices $t \in \mathcal{T}$ and $v \in \mathcal{D}$ are connected if and only if $v \in t$.
	\medskip
	
	Because $I_{\Gamma}^{TD}$ is bipartite, it is enough for proving Proposition \ref{PropAutomorphismsPreserveType} to show that there is a vertex $v$ of type $D$ in $I_{\Gamma_1}^{TD}$ such that $\psi(v)$ is of type $D$ in $I_{\Gamma_2}^{TD}$.
	
	If $\Gamma$ is triangle-free, the result of Proposition \ref{PropAutomorphismsPreserveType} follows from \cite[Proposition 41]{crisp2005automorphisms}. Therefore, throughout the rest of this section, we always suppose that $\Gamma$ contains at least one triangle.
	
	\begin{lemma} \label{LemmaBipartiteGirth6}
		$I_{\Gamma}^{TD}$ is bipartite and has girth $6$.
	\end{lemma}
	
	\begin{proof}
		This essentially comes from Corollary \ref{CoroBipartite}. The girth is exactly $6$ because by hypothesis $\Gamma$ contains $3$-cycle with vertices $a, b, c \in V(\Gamma)$, hence the $6$-cycle
		$$(\langle a \rangle, \langle z_{ab} \rangle, \langle b \rangle, \langle z_{bc} \rangle, \langle c \rangle, \langle z_{ac} \rangle)$$
		is contained in $I_{\Gamma}^{TD}$, where $z_{st}$ is an element generating the centre of $A_{st}$.
	\end{proof}
	
	Before stating the next lemma, we make the following observation, that is true for any $2$-dimensional Artin group:
	\medskip
	
	\noindent \textbf{Observation:} Let $\Gamma_{bar}$ denote the barycentric subdivision of $\Gamma$. Then the boundary of the fundamental domain $K_{\Gamma}$ is isomorphic to $\Gamma_{bar}$.
	\medskip
	
	We will thus often identify $\Gamma_{bar}$ with the boundary of $K_{\Gamma}$.
	\medskip
	
	To make the notation clearer, we will throughout the rest of the section denote by $t$ or by $t_i$ a vertex of type $T$, and by $v$ or by $v_i$ a vertex of $D$.
	
	\begin{lemma} \label{Lemma6CyclesinFD}
		Let $\gamma$ be a $6$-cycle in $I_{\Gamma}^{TD}$. Then up to conjugation, there are three standard generators $a, b, c \in V(\Gamma)$ such that
		$$\gamma = (\langle a \rangle, \langle z_{ab} \rangle, \langle b \rangle, \langle z_{bc} \rangle, \langle c \rangle, \langle z_{ac} \rangle).$$
	\end{lemma}
	
	\begin{proof} 
		As $I_{\Gamma}^{TD}$ is bipartite, $\gamma$ takes the form
		$$\gamma = (t_1, v_{12}, t_2, v_{23}, t_3, v_{13}),$$
		where the $t_i$'s are standard trees and the $v_{ij}$'s are type $2$ vertices. By construction, the vertices $v_{12}$ and $v_{13}$ both lie on the standard tree $t_1$. Since standard trees are convex by \ref{RemStandardTrees}.(1), the geodesic from $v_{12}$ to $v_{13}$ lie entirely inside of $t_1$. The same goes for $t_2$ and $t_3$.
		
		We now consider the geodesic triangle $T_{\gamma}$ formed by concatenating these three geodesics. By \cite[Claim in the proof of Proposition 3.5]{vaskou2023automorphisms}, $T_{\gamma}$ is contained in a single translate $g K_{\Gamma}$ of the fundamental domain. It is actually contained in the boundary of that translate, and thus $\gamma$ corresponds to a $6$-cycle in $g \Gamma_{bar}$, which precisely means that there are three standard generators $a, b, c \in V(\Gamma)$ such that
		$$\gamma = g (\langle a \rangle, \langle z_{ab} \rangle, \langle b \rangle, \langle z_{bc} \rangle, \langle c \rangle, \langle z_{ac} \rangle) g^{-1}.$$
	\end{proof}
	
	\begin{notation}
		For every $6$-cycle $\gamma$ in $I_{\Gamma}^{TD}$, we will denote by $\widetilde{\gamma}$ the associated $6$-cycle in $D_{\Gamma}$, as described in Lemma \ref{Lemma6CyclesinFD}.
	\end{notation}
	
	\begin{notation}
		We will at times say that “two cycles $\gamma$ and $\gamma'$ of $I_{\Gamma}$ share a $t_1 - v_1 - t_2$ subgraph”. What this will mean, is that there are distinct vertices $t_1, v_1, t_2 \in I_{\Gamma}$ of type T, D and T respectively, such that $t_1$ is adjacent to $v_1$ and $v_1$ is adjacent to $t_2$, and such that those three vertices belong to both $\gamma$ and $\gamma'$. We will similarly sometimes say that $\gamma$ and $\gamma'$ share a $t_1 - v_1$ subgraph.
	\end{notation}
	
	\begin{lemma} \label{Lemmaf1v1f2}
		Up to conjugation, two $6$-cycles in $I^{TD}_\Gamma$ intersecting (exactly) along a $f_1-v_1-f_2$ subgraph correspond to one of the two following situations in $D_{\Gamma}$, where $\Delta_{v_1}$ is the Garside element of the dihedral Artin parabolic subgroup corresponding to $v_1$.
		\begin{figure}[H]
			\centering
			\includegraphics[scale=1]{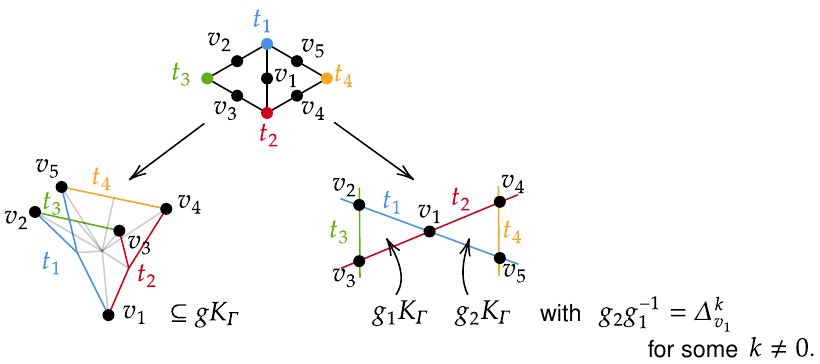}
			\caption{$6$-cycles intersection along a $f_1-v_1-f_2$ subgraph.}
			\label{Fig3}
		\end{figure}
	\end{lemma}
	
	\begin{proof}
		By Lemma \ref{Lemma6CyclesinFD}, each of the two $6$-cycles are contained into a translate of the fundamental domain $K_{\Gamma}$. So either these two fundamental domains are the same, and the situation is described on the left of Figure \ref{Fig3}. Or they are distinct, and the situation is described on the right of Figure \ref{Fig3}. What is left to show if that the element $g_2 g_1^{-1}$ is as wanted.
		
		Up to conjugation, we can suppose that $t_1 = Fix(a)$ and $t_2 = Fix(b)$ for two standard generators $a, b \in V(\Gamma)$ with $m_{ab} < \infty$. In particular, the local group at $v_1$ is $A_{ab}$. The element $g_2$ is non-trivial because $K_{\Gamma} \neq g_2 K_{\Gamma}$ by hypothesis. By \cite[Lemma 4.3]{martin2022acylindrical}, there is some $k \neq 0$ such that $g_2 \in \Delta_{ab}^k \langle a \rangle$, where $\Delta_{ab}$ is the Garside element in $a$ and $b$. This means there is some $q \in \mathbb{Z}$ such that
		$$g_2 = \Delta_{ab}^k a^q.$$
		A similar reasoning on the standard there $t_2$ shows that
		$$g_2 = \Delta_{ab}^{k'} b^{q'}.$$
		Dihedral Artin groups are known to have a normal form called the Garside normal form (see \cite{brieskorn1972artin}). Note that the two above expressions are in Garside normal forms. By unicity of this form, it must be that $\Delta_{ab}^k = \Delta_{ab}^{k'}$ and $a^q = b^{q'}$. This forces $k = k'$ and $q = q' = 0$. Finally, we obtain $g_2 = \Delta_{ab}^k$, as wanted.
	\end{proof}
	
	\begin{lemma} \label{Lemmaf1v1}
		Two $6$-cycles intersecting (exactly) along a $t_1-v_1$ subgraph correspond to the following situation in $D_{\Gamma}$, where $\Delta_{v_1}$ is the Garside element of the dihedral Artin parabolic subgroup corresponding to $v_1$.
		\begin{figure}[H]
			\centering
			\includegraphics[scale=1]{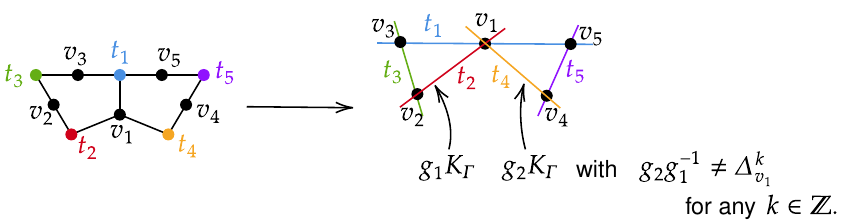}
			\caption{$6$-cycles intersection along a $f_1-v_1$ subgraph.}
			\label{Fig4}
		\end{figure}
	\end{lemma}
	
	\begin{proof}
		By Lemma \ref{Lemma6CyclesinFD}, each of the two $6$-cycles are contained into a translate of the fundamental domain. An argument similar to that of Lemma \ref{Lemmaf1v1f2} shows that if $g_2 g_1^{-1}$ belonged to the centre of the local group at $v_1$, then the standard trees $f_2$ and $f_4$ would coincide. So the two fundamental domains are as described on Figure \ref{Fig4}.
	\end{proof}
	
	\begin{lemma} \label{LemmaChainOf6Cycles}
		Consider a family $\gamma_1, \cdots, \gamma_n$ of $6$-cycles and a vertex $t_* \in T$ satisfying $t_* = \bigcap\limits_{i=1}^n \gamma_i$ such that for every $i, j \in \{1, \cdots, n\}$ with $i < j$ we have
		$$\gamma_i \cap \gamma_j = \begin{cases} \text{some } t_i - v_i - t_* \text{ subgraph if } j = i+1 ;\\ \text{at least } t_* \text{ if } i=1, j=n ;\\ t_* \text{ otherwise}. \end{cases}$$
		(See Figure \ref{FigComb}). By Lemma \ref{Lemma6CyclesinFD}, there is for every $i \in \{1, \cdots, n\}$ a unique element $g_i \in A_{\Gamma}$ such that $\widetilde{\gamma_i} \in g_i K_{\Gamma}$. Let $\{ \lambda_1, \cdots, \lambda_m\} \subseteq \{1, \cdots, n-1\}$ be the (possibly empty) set of indexes for which $g_i \neq g_{i+1}$. Then the following is a geodesic path in $D_{\Gamma}$:
		$$[v_{\lambda_1}, v_{\lambda_2}] \cup [v_{\lambda_2}, v_{\lambda_3}] \cup \cdots \cup [v_{\lambda_{m-1}}, v_{\lambda_m}].$$
		Consequently, if $m \geq 1$, then the combinatorial distance between $v_1$ and $v_{n-1}$ in $D_{\Gamma}$ is $2m-2$. In particular, if $v_0 = v_n$ then we must have $m \leq 1$.
		\begin{figure}[H]
			\centering
			\includegraphics[scale=1]{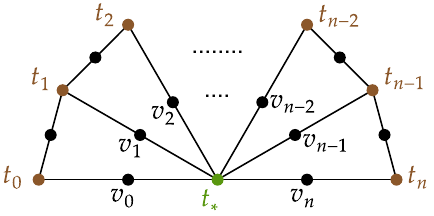}
			\caption{The situation of Lemma \ref{LemmaChainOf6Cycles} as seen in $I_{\Gamma}^{TD}$. We do not necessarilly assume that $v_0 \neq v_n$ or that $t_0 \neq t_n$.}
			\label{FigComb}
		\end{figure}
	\end{lemma}
	
	\begin{proof}
		The statement is empty if $m = 0$, so we assume that $m \geq 1$. For any $k \in \{1, \cdots, m-1\}$, the vertices $v_{\lambda_k}$ and $v_{\lambda_{k+1}}$ are two type $2$ vertices both contained in $g_{\lambda_k} K_{\Gamma}$. So the combinatorial distance between them is precisely $2$, and the combinatorial length of $[v_{\lambda_k}, v_{\lambda_{k+1}}]$ is also $2$. By construction, and using Lemma \ref{Lemmaf1v1f2}, the $6$-cycles $\gamma_{\lambda_k}$ and $\gamma_{\lambda_{k+1}}$ are in the situation described below:
		\begin{figure}[H]
			\centering
			\includegraphics[scale=1]{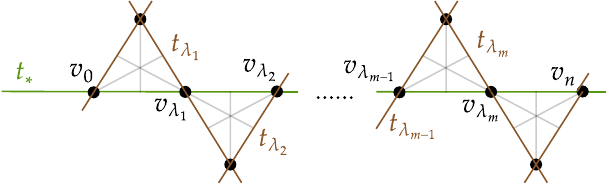}
			\caption{The situation of Lemma \ref{LemmaChainOf6Cycles} as seen in the Deligne complex.}
			\label{FigCombProof}
		\end{figure}
		
		To prove the main statement, it is now enough to prove that for any $k \in \{2, \cdots, m-1\}$ the two geodesics $[v_{\lambda_{k-1}}, v_{\lambda_k}]$ and $[v_{\lambda_k}, v_{\lambda_{k+1}}]$ concatenate into a geodesic.
		The result is now clear because these two segments are both contained in the standard tree $t_*$ with is convex by Lemma \ref{RemStandardTrees}.(1), and they meet at $v_{\lambda_k}$.
	\end{proof}
	
	The following lemma is crucial to show that vertices of type $T$ and vertices of type $D$ cannot be sent to each others.
	
	\begin{lemma} \label{LemmaDifference}
		Suppose that $\Gamma$ contains at least one triangle. Then there is some $n \geq 3$ (as defined in the figure below) such that $I_{\Gamma}^{TD}$ contains a subgraph isomorphic to $G_1$, but no subgraph isomorphic to $G_2$:
		\begin{figure}[H]
			\centering
			\includegraphics[scale=0.8]{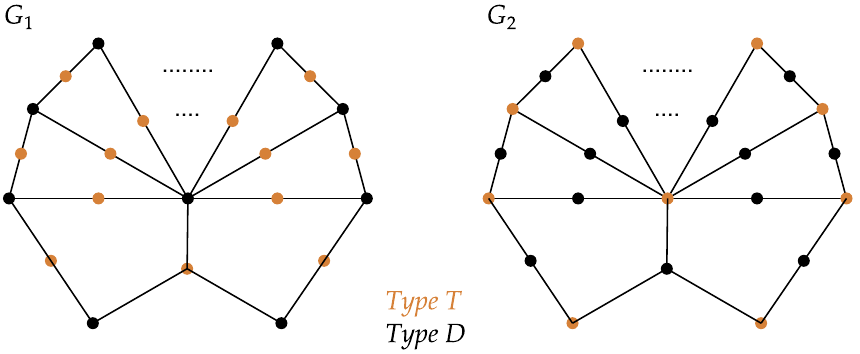}
			\caption{The subgraphs $G_1$ and $G_2$. Both graphs are the union of exactly $n$ $6$-cycles, attached as pictured. The vertices of type $T$ and type $D$ are drawn in different colours.}
			\label{FigureTwoSubgraphs}
		\end{figure}
	\end{lemma}
	
	\begin{proof}
		By hypothesis $\Gamma$ contains a triangle, so there are three vertices $a, b, c \in V(\Gamma)$ that satisfy $m_{ab}, m_{ac}, m_{bc} < \infty$. In the Deligne complex, we let $T_{st}$ be the $2$-cell formed by the three vertices $\{1\}$, $\langle s \rangle$ and $A_{st}$. We first consider the subcomplex $K$ of the fundamental domain $K_{\Gamma}$ described by
		$$K \coloneqq T_{ab} \cup T_{ba} \cup T_{bc} \cup T_{cb} \cup T_{ca} \cup T_{ac}.$$
		Then, we consider the subcomplex $Y$ of $D_{\Gamma}$ defined as
		$$Y \coloneqq K \cup aK \cup abK \cup abaK \cup \cdots \cup \underbrace{aba \cdots x}_{m_{ab}-1 \text{ terms}} K,$$
		where $x = a$ or $x = b$ depending on the parity of $m_{ab}$.
		\begin{figure}[H]
			\centering
			\includegraphics[scale=0.8]{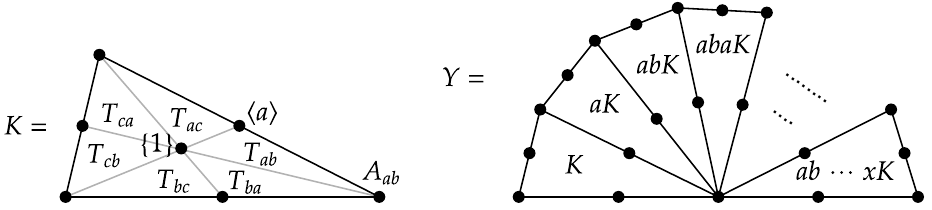}
			\caption{The subcomplexes $K$ and $Y$ in the Deligne complex. The horizontal segments of $Y$ are contained in the same standard tree.}
		\end{figure}
		One can easily notice from the above figure that this situation gives rise to a subgraph in $I_{\Gamma}$ that is isomorphic to the subgraph $G_1$ of Figure \ref{FigureTwoSubgraphs}, where $n = m_{ab}$. This proves the first statement of the lemma.
		\medskip
		
		We now want to show that $I_{\Gamma}^{TD}$ cannot contain any subgraph isomorphic to $G_2$, for $n = m_{ab}$. So we suppose it does, and we will point out a contradiction. We label the $6$-cycles appearing in $G_2$ as follows: the bottom-left one is $\gamma_1$, the one it shares three vertices with is $\gamma_2$, etc. until we reach $\gamma_n$ in the bottom-right corner.
		
		By construction, $\widetilde{\gamma_1}$ and $\widetilde{\gamma_n}$ share a type $2$ vertex $v$, so we can use Lemma \ref{LemmaChainOf6Cycles} and conclude that the subset $\{ \lambda_1, \cdots, \lambda_m\} \subseteq \{1, \cdots, n-1\}$ of indexes for which $g_i \neq g_{i+1}$ satisfies $m \leq 1$. By Lemma \ref{Lemmaf1v1f2}, this means that $g_{\lambda_1} g_{\lambda_1+1}^{-1} = \Delta_{v'}^k$, where $\Delta_{v'}$ is the Garside element associated with the type $2$ vertex $v'$ at which the two distinct translates of $K_{\Gamma}$ intersect, and $k \neq 0$. As $g_1 = \cdots = g_{\lambda_1 -1}$ and $g_{\lambda +2} = \cdots = g_n$, we also obtain $g_n g_1^{-1} = \Delta_{v'}^k$. In particular, $g_n g_1^{-1}$ has type $2$.
		
		The $6$-cycles $\gamma_1$ and $\gamma_n$ intersect (exactly) along a $t - v$ subgraph where $t$ is the type $T$ vertex that is at the centre of $G_2$ on Figure \ref{FigureTwoSubgraphs}, and $v$ is the type $D$ vertex below $t$. This means that we can apply Lemma \ref{Lemmaf1v1}, which states that $g_n g_1^{-1}$ belongs to $G_v$ but not to $\langle \Delta_v \rangle$, where $\Delta_v$ is the Garside element associated with $v$.
		
		Note that $v' \neq v$ by hypothesis, as $v'$ belongs to the upper-half of $G_2$ on Figure \ref{FigureTwoSubgraphs}, while $v$ lies in the lower-half of the picture (below the central vertex $t$). Consequently, the parabolic subgroups $G_v$ and $G_{v'}$ are distinct. Their intersection is exactly the subgroup $G_t$ fixing the standard tree $t$ pointwise. In particular, $g_n g_1^{-1}$ must be an element of type $1$. This gives a contradiction to the above argument that proved that $g_n g_1^{-1}$ had type $2$.
	\end{proof}
	
	We can now prove Proposition \ref{PropAutomorphismsPreserveType}:
	
	\begin{proof}
		Let $\psi : I_{\Gamma_1} \rightarrow I_{\Gamma_2}$ be an isomorphism. We already know by Corollary \ref{CorollaryExoticSentToExotic} that $\psi$ reduces to an isomorphism $\overline{\psi} : I_{\Gamma_1}^{TD} \rightarrow I_{\Gamma_2}^{TD}$.
		
		As already mentioned, if $\Gamma$ is triangle-free then the result follows from \cite[Proposition 41]{crisp2005automorphisms}. So we assume that $\Gamma$ contains at least one triangle.
		
		By Lemma \ref{LemmaBipartiteGirth6} the graphs $I_{\Gamma_1}^{TD}$ and $I_{\Gamma_2}^{TD}$ are bipartite, so it is enough to exhibit one vertex $v$ of type $D$ in $I_{\Gamma_1}^{TD}$ such that $\psi(v)$ is also of type $D$ in $I_{\Gamma_2}^{TD}$. By Lemma \ref{LemmaDifference} it is enough to let $v$ be the vertex of type $D$ located at the centre of the graph $G_1$ (see Figure \ref{FigureTwoSubgraphs}).
	\end{proof}

	\subsection{Fundamental subgraphs}
	\label{subsecFund}
	Throughout this section we suppose that Artin groups and presentation graphs are large-type.
	
	\begin{definition} \label{DefiFundGraphI} [Fundamental graph: definition I]
		Let $\Gamma$ and $\Gamma'$ be two connected presentation graphs, and write $V(\Gamma') = \{s_1, \cdots, s_n\}$. Then $\Gamma'$ is called \emph{fundamental} relatively to $D_{\Gamma}$ if for every distinct (infinite) standard trees $t_1, \cdots, t_n$ in $D_{\Gamma}$ satisfying
		\begin{enumerate}
			\item If $s_i$ and $s_j$ are adjacent in $\Gamma'$ then $t_i \cap t_j$ is a single vertex $v_{ij}$, and all $v_{ij}$'s are distinct ;
			\item If $s_i$ and $s_j$ are not adjacent in $\Gamma'$ then $t_i \cap t_j = \emptyset$ ;
		\end{enumerate}
		there is a unique element $g \in A_{\Gamma}$ such that the translate $g K_{\Gamma}$ contains all the vertices of the form $v_{ij}$ and contains a type $1$ vertex $x_i \in t_i$ for every $i \in \{1, \cdots, n\}$.
		
		If $\Gamma' = \Gamma$, we will simply say that $\Gamma$ is \emph{fundamental}.
	\end{definition}
	
	Being fundamental for a graph $\Gamma$ has strong consequences, that we will make explicit in Section \ref{subsecComparison}. In this section, we focus on proving the following proposition.
	
	\begin{proposition}
		\label{prop:fund}
		Let $A_{\Gamma}$ be a large-type Artin group. If $\Gamma$ admits a twistless hierarchy terminating at twistless stars (see Definition \ref{DefiTwistlessHierarchy}), then $\Gamma$ is fundamental.
	\end{proposition}
	
	\begin{proof}
		Every twistless star is fundamental by Lemma \ref{LemmaNoSepVer}. The fact that this extends to twistless hierarchies terminating at twistless stars then comes from Lemma \ref{LemmaExtendingToTwistlessHierarchy}.
	\end{proof}
	
	\begin{remark}
		\label{rmk:example}
		Note that one scenario where Proposition~\ref{prop:fund} applies is if there are vertices $v_1, \cdots, v_n$ of $\Gamma$ such that:
		\begin{enumerate}
			\item $st(v_i)$ does not have separating vertices or edges ;
			\item $\Gamma=\cup_{i=1}^nst(v_i)$ ;
			\item for any $st(v_i)$ and $st(v_j)$, there is a finite chain of elements in $\{st(v_i)\}_{i=1}^n$ with the first member being $st(v_i)$ and last member being $st(v_j)$ such that adjacent members in the chain have a non-trivial intersection which is not a vertex or an edge.
		\end{enumerate}
		One readily checks that in this case, $\Gamma$ has a twistless hierarchy terminating at $\{st(v_i)\}_{i=1}^n$. As a consequence, if $\Gamma$ is the 1-skeleton of the triangulation of a closed manifold of dimension $\ge 2$, then $\Gamma$ satisfies the assumption of Proposition~\ref{prop:fund}.
	\end{remark}
	
	In what follows we give an equivalent definition of what is means to be fundamental.
	
	\begin{notation}
		Let $\Gamma$ be a presentation graph. Then we write $\Gamma_{bar}$ to denote the barycentric subdivision of $\Gamma$. Moreover, we will denote a vertex in $\Gamma$ and the corresponding vertex in $\Gamma_{bar}$ by the same name.
	\end{notation}
	
	\begin{definition}
		Let $\Gamma$ and $\Gamma'$ be two connected presentation graphs. A subgraph $G$ of $I_{\Gamma}^{TD}$ is called \emph{$\Gamma'$-characteristic} if there is a graph isomorphism $f : G \rightarrow \Gamma'_{bar}$ where for every vertex $t \in G$ of type $T$ the image $f(t)$ is a vertex of $\Gamma'$.
		
		If $\Gamma' = \Gamma$, we will simply say that $\Gamma$ is \emph{characteristic}.
	\end{definition}
	
	\begin{definition} \label{DefiFundGraphII} [Fundamental graph: definition II]
		Let $\Gamma$ and $\Gamma'$ be two connected presentation graphs. Then $\Gamma'$ is called \emph{fundamental} relatively to $D_{\Gamma}$ if for every $\Gamma'$-characteristic subgraph $G \subseteq I_{\Gamma}^{TD}$ there is a unique element $g \in A_{\Gamma}$ such that the translate $g K_{\Gamma}$ contains all the vertices of type $D$ of $G$, as well as one type $1$ vertex $x_i \in t_i$ of $D_{\Gamma}$ for every vertex $t_i$ of type $T$ of $G$.
		
		If $\Gamma' = \Gamma$, we will simply say that $\Gamma$ is \emph{fundamental}.
	\end{definition}
	
	\begin{lemma}
		Definitions \ref{DefiFundGraphI} and \ref{DefiFundGraphII} are equivalent.
	\end{lemma}
	
	\begin{proof}
		First assume that $\Gamma'$ is fundamental relatively to Definition \ref{DefiFundGraphI}, and let $G \subseteq I_{\Gamma}^{TD}$ be $\Gamma'$-characteristic. Let $t_1, \cdots, t_n$ be the vertices of type $T$ in $G$, where each $t_i$ correspond to a standard generator $s_i \in V(\Gamma')$, and let the $v_{ij}$'s be the vertices of type $D$ of $G$, where $v_{ij}$ is the only vertex connecting $t_i$ and $t_j$, if it exists. Because $G$ is $\Gamma'$-characteristic, the vertex $v_{ij}$ exists if and only if $s_i$ and $s_j$ are connected in $\Gamma'$. As $\Gamma'$ is fundamental, there is by Definition \ref{DefiFundGraphI} a unique $g \in A_{\Gamma}$ and some $x_i \in t_i$ for each $i \in \{1, \cdots, n\}$, such that all $x_i$'s and all $v_{ij}$'s belong to $g K_{\Gamma}$. Therefore $\Gamma'$ is fundamental relatively to Definition \ref{DefiFundGraphII}.
		
		We now assume that $\Gamma'$ is fundamental relatively to Definition \ref{DefiFundGraphII}, and we let $t_1, \cdots, t_n$ and the $v_{ij}$'s be as in Definition \ref{DefiFundGraphI}. By construction, the $t_i$'s and the $v_{ij}$'s are all distinct vertices of $I_{\Gamma}^{TD}$. Note that $v_{ij}$ exists if and only if $s_i$ and $s_j$ are adjacent in $\Gamma'$, so there is a double-edge between $v_i$ and $v_j$ if and only if $s_i$ and $s_j$ are adjacent in $\Gamma'$. Together with the above, this shows that the subgraph $G$ of $I_{\Gamma}^{TD}$ spanned by the $t_i$'s and the $v_{ij}$'s is $\Gamma'$-characteristic. By Definition \ref{DefiFundGraphII}, there is a unique $g \in A_{\Gamma}$.  and some $x_i \in t_i$ for each $i \in \{1, \cdots, n\}$, such that all $x_i$'s and all $v_{ij}$'s belong to $g K_{\Gamma}$. Therefore $\Gamma$ is fundamental relatively to Definition \ref{DefiFundGraphI}.
	\end{proof}
	
	The following lemma from \cite{vaskou2023isomorphism} will be our starting point towards proving Proposition \ref{prop:fund}.
	
	\begin{lemma} \label{Lemma6CyclesFundamental}
		\cite[Claim in the proof of Proposition 3.5]{vaskou2023automorphisms}
		Let $A_{\Gamma}$ be a large-type Artin group. If $\Gamma'$ is a $3$-cycle, then $\Gamma'$ is fundamental relatively to $D_{\Gamma}$.
	\end{lemma}
	
	The following lemma allows to extend fundamentality from $3$-cycles to cones over cycles.
	
	\begin{lemma} \label{LemmaConeOverCycle}
		Let $A_{\Gamma}$ be a large-type Artin group. If $\Gamma'$ is the cone over an induced simple cycle, then $\Gamma'$ is fundamental relatively to $D_{\Gamma}$.
	\end{lemma}
	
	\begin{proof}
		Let $G$ be a $\Gamma'$-characteristic subgraph of $I_{\Gamma}^{TD}$. By hypothesis, $G$ takes the form described in Figure \ref{FigComb} with $t_0 = t_n$ and $v_0 = v_n$. Let $m$ and $g_1, \cdots, g_n$ be as in Lemma \ref{LemmaChainOf6Cycles}. Then by Lemma \ref{LemmaChainOf6Cycles}, it must be that $m \leq 1$.
		
		If $m = 1$, then $g_{\lambda_1} \neq g_{\lambda_1+1}$ and $g_n \neq g_1$. Then we can rotate all the indices by $+1$ modulo $n$ (or by $-1$ if $\lambda_1 = n-1$) so that $g_{\lambda_1+1} \neq g_{\lambda_1+2}$ and $g_1 \neq g_2$. In this new setting, we have $m = 2$, which gives a contradiction to Lemma \ref{LemmaChainOf6Cycles}.
		
		So we must have $m = 0$, i.e. all the $g_i$'s are the same. This proves that $G$ is $\Gamma'$-fundamental, as wanted.
	\end{proof}
	
	The next two lemmas extend fundamentality from cones over cycles to cones over some unions of cycles, or equivalently, to cones that have no separating edges or vertices.
	
	\begin{lemma} \label{LemmaConditionsC1C2}
		Let $G$ be a simplicial graph. Then the two following conditions are equivalent:
		\begin{align*}
			(C1) \ \ &G \text{ is the union of some induced cycles } \gamma_1, \cdots, \gamma_n \text{ for which for } \\
			&\text{every } i, j \in \{1, \cdots, n\}, \text{ there are some } r_1 = i, r_2, \cdots, r_m = j \\
			&\text{such that every intersection } \gamma_{r_k} \cap \gamma_{r_{k+1}} \text{ contains an edge.} \\
			(C2) \ \ &G \text{ is connected and has no separating vertex.}
		\end{align*}
	\end{lemma}
	
	\begin{proof}
		(C1) $\Rightarrow$ (C2). Pick any two cycles $\gamma_i, \gamma_j$. If we remove any vertex $v$ from $G$, the intersection $\gamma_{r_k} \cap \gamma_{r_{k+1}}$ still contains at least one vertex. Hence $G \backslash \{v\}$ is still connected, which proves that $G$ does not have any separating vertex.
		\medskip
		
		\noindent (C2) $\Rightarrow$ (C1). This is \cite[Lemma 6.4]{blufstein2024homomorphisms}.
	\end{proof}
	
	\begin{lemma} \label{LemmaNoSepVer}
		Let $A_{\Gamma}$ be a large-type Artin group. Suppose that $\Gamma'$ is the cone over a graph that is connected and has no separating vertex (equivalently, $\Gamma'$ is a cone and has no separating vertex or edge). Then $\Gamma'$ is fundamental relatively to $D_{\Gamma}$. 
	\end{lemma}
	
	\begin{proof}
		By Lemma \ref{LemmaConditionsC1C2} we know that $G$ satisfies condition (C1). Let $t \in V(\Gamma)$ denote the cone point, and let $G_i$ be the subgraph spanned by $t$ and $\gamma_i$. We know by Lemma \ref{LemmaConeOverCycle} that for each $i \in \{1, \cdots, n\}$ the subgraph $G_i$ is fundamental relatively to $D_{\Gamma}$, each $G_i$ having a corresponding unique element $g_i \in A_{\Gamma}$. Since $\gamma_i$ and $\gamma_{i+1}$ share an edge $e = (t_1, t_2)$, the subgraphs $G_i$ and $G_{i+1}$ share the triangle $T \coloneqq (t, t_1, t_2)$. This triangle is fundamental by Lemma \ref{Lemma6CyclesFundamental}, with a unique corresponding element $g \in A_{\Gamma}$. But since $T \subseteq G_i, G_{i+1}$, it must be that $g_1 = g = g_2$ by unicity. Proceeding through all $i \in \{1, \cdots, n-1\}$, the above shows that $g_1 = g_2 = \cdots = g_n$, i.e. $\Gamma'$ is fundamental.
	\end{proof}
	
	Before proving the next result, we need the following technical lemma:
	
	\begin{lemma} \label{LemmaTwoStdTreesBorderingTwoFD}
		Let $A_{\Gamma}$ be a large-type Artin group, let $g_1, g_2 \in A_{\Gamma}$ and let $t_1, t_2$ be two non-intersecting standard trees in $D_{\Gamma}$. If each of $g_1 K_{\Gamma}$ and $g_2 K_{\Gamma}$ intersects each of $t_1$ and $t_2$ along at least a type $1$ vertex, then we must have $g_1 = g_2$.
	\end{lemma}
	
	\begin{proof}
		Up to conjugation, we will assume that $g_1 = 1$ and that $t_1 = Fix(a)$, $t_2 = Fix(b)$ where $m_{ab} = \infty$. Call $v_a \in Fix(a)$ and $v_b \in Fix(b)$ the type $1$ vertices corresponding to the cosets $\langle a \rangle$ and $\langle b \rangle$ respectively. By construction of the Moussong metric $d$, we know that
		$$d(v_a, v_{\emptyset}) = 1 = d(v_b, v_{\emptyset}).$$
		Moreover, we know that the angle $\angle_{v_a}(Fix(a), v_{\emptyset})$ is precisely $\pi/2$, and similarly for $\angle_{v_b}(Fix(b), v_{\emptyset})$.
		
		If $\Gamma$ is disconnected, then $v_{\emptyset}$ disconnects the fundamental domain $K_{\Gamma}$, so it lies on the geodesic path from $v_a$ to $v_b$ by convexity of $K_{\Gamma}$ (Remark \ref{RemFDConvex}) and $d(v_a, v_b) = 2$. We now suppose that $\Gamma$ is connected. We want to compute the angle $\angle_{v_{\emptyset}}(v_a, v_b)$. Note that this angle is equal to the distance between $v_a$ and $v_b$ in the link of $v_{\emptyset}$ when given the angular metric. In particular, $lk_{X_{\Gamma}}(v_{\emptyset})$ is a graph isomorphic to the barycentric subdivision $\Gamma_{bar}$, where each edge $(v_s, v_{st})$ has length the angle $\angle_{v_{\emptyset}}(v_s, v_{st})$, that is $\pi / 2 - \pi / 2 m_{st}$. Because $\Gamma$ is connected, but $a$ and $b$ are not adjacent, the distance between these two vertices in $\Gamma_{bar}$ is at least $4$. Because $A_{\Gamma}$ is large-type, this yields:
		$$\angle_{v_{\emptyset}}(v_a, v_b) \geq 4 \cdot \left( \frac{\pi}{2} - \frac{\pi}{6} \right) = \frac{4 \pi}{3} > \pi.$$
		By convexity of $K_{\Gamma}$ (Remark \ref{RemFDConvex}), the geodesic between $v_a$ and $v_b$ in $D_{\Gamma}$ lies in $K_{\Gamma}$. So the previous inequation implies that $v_{\emptyset}$ also lies on the geodesic between $v_a$ and $v_b$, with an angle strictly greater than $\pi$.
		
		The standard trees $Fix(a)$ and $Fix(b)$ intersect $g_2 K_{\Gamma}$ along some type $1$ vertices $g_2 v_s$ and $g_2 v_t$ respectively, for appropriate $s, t \in V(\Gamma)$. A similar reasoning shows that $g_2 v_{\emptyset}$ lies on the geodesic between $g_2 v_s$ and $g_2 v_t$. Of course, we also have $\angle_{g_2 v_s}(Fix(a), g_2 v_{\emptyset}) = \angle_{g_2 v_t}(Fix(b), g_2 v_{\emptyset}) = \pi / 2$.
		
		We now suppose that $g_1 \neq g_2$. We consider the quadrilateral $Q \coloneqq (v_a, g_2 v_s, g_2 v_t, v_b)$, which is not degenerate as $g_1 \neq g_2$. Note that the sum of the four angles of $Q$ is exactly $2 \pi$. We can now apply \cite[Theorem II.2.11]{BridsonHaefliger1999}: the convex hull $c(Q)$ of $Q$ is isometric to the convex hull of a quadrilateral in the Euclidean plane. In particular, every point $p$ that lies on the interior of the geodesic $[v_a, v_b]$ must satisfy $\angle_p(v_a, v_b) = \pi$. This gives a contradiction as $\angle_{v_{\emptyset}}(v_a, v_b) > \pi$. Thus $g_1 = g_2$, as wanted.    
	\end{proof}
	
	\begin{lemma} \label{LemmaExtendingToTwistlessHierarchy}
		Let $A_{\Gamma}$ be a large-type Artin group, and let $\Gamma_1, \Gamma_2$ be two graphs that are fundamental relatively to $D_{\Gamma}$. Let now $\Gamma'$ be a graph satisfying the following: there exists $\varphi_1 : \Gamma_1 \hookrightarrow \Gamma'$ and $\varphi_2 : \Gamma_2 \hookrightarrow \Gamma'$ such that $\Gamma' = \varphi_1(\Gamma_1) \cup \varphi_2(\Gamma_2)$ and $\varphi_1(\Gamma_1) \cap \varphi_2(\Gamma_2)$ is none of the empty set, a single vertex, or a single edge. Then $\Gamma'$ is fundamental relatively to $D_{\Gamma}$.
	\end{lemma}

	\begin{proof}
		To simplify notations in the proof, we will identify $\Gamma_i$ with $\varphi_i(\Gamma_i)$ for both $i \in \{1, 2\}$.
		
		We want to show that $\Gamma' = \Gamma_1 \cup \Gamma_2$ is fundamental, using Definition \ref{DefiFundGraphII}. Thus we let $G \subseteq I_{\Gamma}^{TD}$ be a $\Gamma'$-characteristic subgraph, and we let $G_1 \subseteq I_{\Gamma}^{TD}$ be a $\Gamma_1$-characteristic subgraph of $G$ and $G_2 \subseteq I_{\Gamma}^{TD}$ be a $\Gamma_2$-characteristic subgraph of $G$, such that $G = G_1 \cup G_2$ but $G_1 \cap G_2$ is none of the empty set, a single vertex of type $T$, or a double-edge connecting two vertices of type $T$.
		
		We denote the vertices of type $T$ of $G_1$ by $t_1, \cdots, t_n$, and similarly we denote the vertices of type $T$ of $G_2$ by $s_1, \cdots, s_m$. By hypothesis, $\Gamma_1$ is fundamental, so we adopt the terminology of Definition \ref{DefiFundGraphII} and let $g_1 \in A_{\Gamma}$ be the (unique) associated group element, $x_1, \cdots, x_n$ be the type $1$ vertices satisfying $x_i \in t_i$, and for every vertex $v_{ij}$ of type $D$ of $G_1$ connecting $t_i$ and $t_j$, we also view $v_{ij}$ as a type $2$ vertex in $D_{\Gamma}$. We proceed similarly for $G_2$, with the element $g_2 \in A_{\Gamma}$, the type $1$ vertices $y_1, \cdots, y_m$ and the type $2$ vertices $w_{ij}$'s. We split the argument in two cases:
		\medskip
		
		\noindent \underline{Case 1: $\Gamma_1$ and $\Gamma_2$ share at least two edges.}
		We call these edge $e$ and $e'$. Each edge corresponds to a double-edge in $I_{\Gamma}^{TD}$ connecting two vertices of type $T$ in $G$, so without loss of generality we assume that $t_1 = s_1$, $t_2 = s_2$ and $v_{12} = w_{12}$. In particular, Lemma \ref{Lemmaf1v1f2} stipulates that we must have $g_2 g_1^{-1} = \Delta_{v_{12}}^k$ for some $k \in \mathbb{Z}$. If $g_2 g_1^{-1}$ is not trivial, it would then have type $2$ and its fixed set on the Deligne complex would be $v_{12}$. Proceeding similarly for $e'$, we obtain a type $2$ vertex $v'$ such that $g_2 g_1^{-1} = \Delta_{v'}^q$ for some $q \in \mathbb{Z}$. Unless $g_2 g_1^{-1}$ is trivial, we have $Fix(g_2 g_1^{-1}) = \{v'\}$. Note that $v_{12}$ and $v'$ are distinct because $e$ and $e'$ are. So the above gives a contradiction unless $g_2 g_1^{-1} = 1$, which proves $g_1 = g_2$. This means $G \subseteq g_1 K_{\Gamma}$, i.e. $\Gamma'$ is fundamental.
		\medskip
		
		\noindent \underline{Case 2: $\Gamma_1$ and $\Gamma_2$ share two non-adjacent vertices.} Call these vertices $a$ and $b$. Each of $a$ and $b$ corresponds to a vertex of type $T$ in $G$, and without loss of generality, we suppose that these are $t_1 = s_1$ and $t_2 = s_2$. By construction, this means the translates $g_1 K_{\Gamma}$ and $g_2 K_{\Gamma}$ both contain a type $1$ vertex in the standard tree $t_1$, and a type $1$ vertex in the standard tree $t_2$. By Lemma \ref{LemmaTwoStdTreesBorderingTwoFD}, this forces $g_1 = g_2$, i.e. $G \subseteq g_1 K_{\Gamma}$, and $\Gamma'$ is fundamental.
	\end{proof}
	
	\begin{remark}
		We leave a curious remark which will not be used in the later part of the article.
		Lemma~\ref{LemmaExtendingToTwistlessHierarchy} breaks down if $\Gamma_1\cap \Gamma_2$ is empty, or a single vertex, or a single edge. The empty case is clear. If $\Gamma = \Gamma_1 \cup \Gamma_2$ with $\Gamma_1\cap \Gamma_2$ is a vertex of an edge, then we can consider a twist with respect to the splitting $A_{\Gamma}=A_{\Gamma_1}*_{A_{\Gamma_1\cap\Gamma_2}}A_{\Gamma_2}$. This twist induces an automorphism of the intersection graph of $A_{\Gamma}$, though it sends a copy of $\Gamma_{bar}$ in $I_\Gamma$ corresponding to a single fundamental domain of $D_\Gamma$ to another subgraph of $I_\Gamma$ which does not correspond to a single fundamental domain of $D_\Gamma$. In other words, Lemma \ref{LemmaExtendingToTwistlessHierarchy} is optimal.
	\end{remark}
	
	At last, we want to highlight a result from \cite{blufstein2024homomorphisms}, which is closely related to the work in this section. Their result only applies to XXXL Artin groups (where all coefficients are $\geq 6$), but it applies to all twistless graphs, and not only to twistless hierarchies terminating at twistless stars. We reformulate the result in our setting:
	
	\begin{proposition} \cite[Proposition 5.1, Lemma 6.7]{blufstein2024homomorphisms}
		Let $A_{\Gamma}$ be an XXXL Artin group  such that $\Gamma$ is connected and twistless. Then $\Gamma$ is fundamental.
	\end{proposition}
	
	\subsection{Comparison homomorphisms}
	\label{subsecComparison}
	Let $C_\Gamma$ be the Cayley complex of $A_\Gamma$, i.e. the universal cover of the presentation complex $P_\Gamma$ of $A_\Gamma$. The 1-skeleton of $C_\Gamma$ is the Cayley graph of $A_\Gamma$.
	We identify $A_\Gamma$ with the vertex set of $C_\Gamma$. A
	block of $C_\Gamma$ is the full subcomplex spanned by all vertices in a left coset of the form
	$gA_e$, where $g\in A_\Gamma$ and $e\subset \Gamma$ is an edge; it is a large block if $e$ has label at least 3. A standard line of $C_\Gamma$ is a line which covers a circle in $P_\Gamma$ 
	associated with a generator (corresponding to a vertex $v$ of $\Gamma$). Thus each standard line is labelled by a vertex of $\Gamma$. With this definition, note that type 1 vertices in the modified Deligne complex $D_\Gamma$ are in 1-1 correspondence with standard lines in $C_\Gamma$, and type 2 vertices in $D_\Gamma$ are in 1-1 correspondence with blocks in $C_\Gamma$.
	
	\begin{definition}(Comparison homomorphism I)
		\label{def:first comparison}
		Suppose $A=A_\Gamma$ is an Artin group such that each vertex
		of $\Gamma$ is contained in an edge of $\Gamma$ with label $\ge 3$ (this holds true when $A_\Gamma$ is of large-type). Let $C_\Gamma$ be the Cayley complex of $A_\Gamma$.
		Then there is an injective homomorphism $h_1:\aut(D_\Gamma)\to \aut(C_\Gamma)$, called the first comparison homomorphism,
		defined as follows \cite[Section 8.1]{horbez2020boundary}. Given an automorphism $\alpha:D_\Gamma\to D_\Gamma$, as $\alpha$ preserves the types of vertices in $D_\Gamma$ \cite[Lemma 2.4]{horbez2020boundary}, it follows from Definition~\ref{def:md} that
		an automorphism of $D_\Gamma$ gives a bijection $f:A_\Gamma\to A_\Gamma$ (viewed as a bijection of the vertex set of $A_\Gamma$) such that $f$ sends vertices in a standard line bijectively to vertices in a standard line; and $f$ sends vertices in a block bijectively to vertices in a block.
		By \cite[Proposition 40]{crisp2005automorphisms}, the restriction of $f$ to the vertex of each large block can be extended to a unique
		cellular automorphism from this block onto its image. By the assumption on $\Gamma$, each standard line 
		of $C_\Gamma$ is contained in a large block, so $f$ sends adjacent vertices to adjacent vertices. As each 2-cell of $C_\Gamma$ is contained in a block, so $f$ sends vertices in a 2-cell to vertices in a 2-cell. Thus $f$ extends to a cellular automorphism. This gives $h:\aut(D_\Gamma)\to \aut(C_\Gamma)$, which is clearly a group homomorphism. Any element of $\aut(D_\Gamma)$ fixing type 0 vertices of $D_\Gamma$
		pointwise will also fix other vertices as they correspond to certain left cosets of $A_\Gamma$. Thus $h_1$ is injective.  
	\end{definition}

	\begin{proposition}
		\label{prop:first comparison}
		Suppose $A_\Gamma$ is large-type. Then the first comparison homomorphism $h_1:\aut(D_\Gamma)\to \aut(C_\Gamma)$ is an isomorphism.
	\end{proposition}
	
	\begin{proof}
		We already see that $h_1$ is injective. To see $h_1$ is surjective, it suffices to show any automorphism of $C_\Gamma$ sends blocks to blocks and standard links to standard lines. However, this follows from \cite[Lemma 8.1]{horbez2020boundary} (note that \cite[Lemma 8.1]{horbez2020boundary} is under the assumption that $A_\Gamma$ is 2-dimensional hyperbolic-type, however, the same proof there also work in the large-type case). 
	\end{proof}
	
	\begin{lemma}
		\label{lem:second comparison}
		Suppose $A_\Gamma$ is of large-type. Then each automorphism of $D_\Gamma$ sends standard trees to standard trees. This gives a homomorphism $h_2:\aut(D_\Gamma)\to \aut(I^{TD}_\Gamma)$, called the second comparison homomorphism. 
		
		Suppose $\Gamma$ is connected and does not have leaf vertices. Then $h_2$ is injective.
	\end{lemma}
	
	\begin{proof}
		This is a variation of \cite[Lemma 8.4 and Lemma 8.5]{horbez2020boundary}. 
		Let $\alpha:D_\Gamma\to D_\Gamma$ be an automorphism of $D_\Gamma$. As $\alpha$ preserves type of vertices \cite[Lemma 2.4]{horbez2020boundary}, and type 0 vertices of $D_\Gamma$ are identified with elements of $A_\Gamma$, we obtain a bijection $A_\Gamma\to A_\Gamma$, which is also denoted by $\alpha$.
		Let $x\in D_\Gamma$ be a type 2 vertex associated with a coset $gA_e$ with $e\subset \Gamma$ being an edge.
		Take two type 1 vertices $x_1$ and $x_2$ adjacent to $x$. We first claim that $x_1$ and $x_2$ belong to the same standard
		tree if and only if $\alpha(x_1)$ and $\alpha(x_2)$ belong to the same standard tree. Indeed, \cite[Proposition 40]{crisp2005automorphisms} implies that $\alpha$ restricted to $gA_e$ (which is identified
		as the collection of type 0 vertices adjacent to $x$) can be extended to a cellular map on
		the block of $C_\Gamma$ containing $gA_e$, thus parallelism of the standard lines inside the blocks
		is preserved (two standard lines are defined to be parallel if they have finite Hausdorff distance). By Lemma~\ref{lemcommensurableequal}, two standard lines are parallel if and only if they have
		the same stabiliser, and this is precisely the condition for the corresponding type one
		vertices of $D_\Gamma$ to belong to the same standard tree. Thus the claim follows.
		
		Note that for a type 1 vertex $x$ in a standard tree $T$, the collection of type 2
		vertices in $T$ which are adjacent to $x$ coincides with the collection of type 2 vertices of
		$D_\Gamma$ adjacent to $x$. As vertices in a standard tree alternate between type 1 and type 2,
		we deduce that $\alpha$ sends a standard tree to a standard tree. As $\alpha$ also maps type 2 vertices to type 2 vertices (\cite[Lemma 2.4]{horbez2020boundary}), it follows from Lemma~\ref{LemmaThreeDefinitionsCoincide} that $\alpha$ induces an automorphism of $I^{TD}_\Gamma$. This gives the group homomorphism $h_2:\aut(D_\Gamma)\to \aut(I^{TD}_\Gamma)$.
		
		Now we show $h_2$ is injective. If an automorphism $\alpha: D_\Gamma\to D_\Gamma$ induces the trivial automorphism of $I^{TD}_\Gamma$, then $\alpha$ fixes each type 2 vertex of $D_\Gamma$. As $\Gamma$ does not have leaf vertex, each vertex in $C_\Gamma$ is the intersection of all blocks containing this vertex. As $\alpha$ restricted to type $0$ vertices of $D_\Gamma$ gives a bijection from $A_\Gamma$ to $A_\Gamma$ sending blocks to blocks, we know $\alpha$ fixes each type 0 vertex of $D_\Gamma$. Hence $\alpha$ fixes each vertex of $D_\Gamma$ by definition of $D_\Gamma$. It follows that $\alpha$ is the identity map. 
	\end{proof}
	
	We give a criterion for the second comparison homomorphism to be an isomorphism.
	
	\begin{theorem}
		\label{thm:2ndiso}
		Let $A_{\Gamma}$ be a large-type Artin group such that $\Gamma$ is connected and does not have leaf vertices. If $\Gamma$ is fundamental, then the map $h_2: \aut(D_\Gamma) \to \aut(I^{TD}_\Gamma)$ from Lemma \ref{lem:second comparison} is an isomorphism.
	\end{theorem}
	
	\begin{proof}
		We will now show how to construct the inverse map $h_2^{-1} : \aut(I^{TD}_\Gamma) \to \aut(D_\Gamma)$. Let $\beta \in \aut(I^{TD}_\Gamma)$. Then we define $h_2^{-1}(\beta)$ as follows:
		\medskip
		
		For any vertex $v \in D_{\Gamma}$ of type $2$, we see $v$ as a vertex of type $D$ in $I_{\Gamma}^{TD}$, and we let $h_2^{-1}(\beta)(v) \coloneqq \beta(v)$.
		\medskip
		
		For any vertex $y \in D_{\Gamma}$ of type $0$, we consider the unique element $g \in A_{\Gamma}$ such that $y \in g K_{\Gamma}$. We now consider the family $D_y$ of type $2$ vertices attached to $y$, and the family $T_y$ of standard trees that have a vertex attached to $y$. Note that both families are finite, and in fact, the subgraph $G_y \coloneqq span(D_y \cup T_y)$ of $I_{\Gamma}^{TD}$ is isomorphic to the boundary of $g K_{\Gamma}$, which is itself isomorphic to $g \Gamma_{bar}$. In other words, $G_y$ is characteristic. As it is a purely graphical condition, the image $\beta(G_y)$ is also a characteristic subgraph of $I_{\Gamma}^{TD}$. We now use the fact that $\Gamma$ is fundamental: there exists a unique element $g_y \in A_{\Gamma}$ such that $g_y K_{\Gamma}$ contains all the vertices of type $D$ of $\beta(G_y)$, and one type $1$ vertex in every vertex of type $T$ of $\beta(G_y)$ (seen as standard trees in $D_{\Gamma}$). We then set $h_2^{-1}(\beta)(y)$ to be the only type $0$ vertex in $g_y K_{\Gamma}$.
		\medskip
		
		If a type $2$ vertex $v$ and a type $0$ vertex $y$ are adjacent, then the corresponding vertex $v \in I_{\Gamma}^{TD}$ belongs to the characteristic subgraph $G_y$ corresponding to $y$. It follows that $\beta(v) \in \beta(G_y)$, and thus $h_2^{-1}(\beta)(v) \in g_y K_{\Gamma}$, i.e. $h_2^{-1}(\beta)(v)$ is attached to $h_2^{-1}(\beta)(y)$. This shows that adjacency between vertices of type $2$ and type $0$ is preserved.
		\medskip
		
		For any vertex $x \in D_{\Gamma}$ of type $1$, we consider the family $D_x$ of type $2$ vertices attached to $x$, and the family $Y_x$ of type $0$ vertices attached to $x$. By the previous points, we know how to define $h_2^{-1}(\beta)(D_x)$ and $h_2^{-1}(\beta)(Y_x)$. One can easily check that there is a unique type $1$ vertex $p_x$ that is adjacent to all the elements of $h_2^{-1}(\beta)(D_x)$ and of $h_2^{-1}(\beta)(Y_x)$, thus we set $h_2^{-1}(\beta)(x) \coloneqq p_x$. Adjacency between type $1$ vertices and other vertices is preserved by definition of $h_2^{-1}(\beta)$ on type $1$ vertices.
		\medskip
		
		One readily verifies that $h_2$ and $h^{-1}_2$ are indeed inverses of each other. This finishes the proof.
	\end{proof}

	\section{Proof of main results and applications to rigidity} \label{SectionProvingMainResults}

	\subsection{QI rigidity results}

	\begin{definition}[The $\Phi$-homomorphism]
		\label{def:Phi}
		Let $A_\Gamma$ be an Artin group of large-type. Then it is known from \cite[Theorem 10.11]{huang2017quasi} that each quasi-isometry $q:A_\Gamma\to A_\Gamma$ send a stable $\mathbb Z$-subgroup $L$ of $A_\Gamma$ (as defined in Section~\ref{subsec:intersection graph}), there is another stable $\mathbb Z$-subgroup $L'\subset A_\Gamma$ such that $d_H(Q(L),L')<\infty$ where $d_H$ denotes the Hausdorff distance. Note that if $L'$ exists, then it is the unique stable $\mathbb Z$-subgroup such that $d_H(Q(L),L')<\infty$, as different stable $\mathbb Z$-subgroups have infinite Hausdorff distance from each other. We also know from \cite[Theorem 10.16]{huang2017quasi} that if two stable $\mathbb Z$-subgroups of $A_\Gamma$ commute, then their $q$-images are Hausdorff close to a pair of commuting stable $\mathbb Z$-subgroup. Thus each quasi-isometry of $A_\Gamma$ induces an automorphism of the intersection graph $I_\Gamma$. This gives a group homomorphism $\Phi:\qi(A_\Gamma)\to \aut(I_\Gamma)$ from the quasi-isometry group of $A_\Gamma$ to the graph automorphism group of the intersection graph $I_\Gamma$.
	\end{definition}
	
	The following is a consequence of \cite[Corollary 10.17]{huang2017quasi}.
	\begin{proposition}
		\label{prop:Phi injective}
		Suppose $A_\Gamma$ is an Artin group of large-type. Suppose $\Gamma$ does not contain a leaf vertex. If $q:A_\Gamma\to A_\Gamma$ is an $(L,A)$-quasi-isometry with $\Phi(q)$ is the identity automorphism, then 
		there exists $C = C(L, A, \Gamma)$ such that
		$d(q(x), x) \le  C$ for any $x\in A_\Gamma$. In particular, $\Phi$ is injective.
	\end{proposition}

	\begin{lemma}
		\label{lem:third comparison}
		Let $A_\Gamma$ be an Artin group of large-type such that $\Gamma$ is connected and it is not a $(3,3,3)$-triangle. As each automorphism of $I_\Gamma$ preserve the set of exotic vertices (Corollary~\ref{CorollaryExoticSentToExotic}), this induces an automorphism $h_3:\aut(I_\Gamma)\to \aut(I^{TD}_\Gamma)$. Then $h_3$ is injective.
	\end{lemma}
	
	\begin{proof}
		It suffices to show if an automorphism $\alpha$ of $I_\Gamma$ fixes vertices of type $T$ and $D$ pointwise, then it fixes each vertex of type $E$. Given a vertex $v\in I_\Gamma$ of type $E$. By Lemma~\ref{LemmaNeighoursOfExotic}, $v$ is adjacent to a vertex of type $T$. Actually, the proof of Lemma~\ref{LemmaNeighoursOfExotic} implies that $v$ is adjacent to two distinct vertices $v_1,v_2$ of type $T$. As $I_\Gamma$ has girth $\ge 5$, $v$ is the unique vertex which is adjacent to each of $v_1$ and $v_2$. Thus $\alpha(v)=v$, as desired.
	\end{proof}
	
	Let $C_\Gamma$ be the Cayley complex of $A_\Gamma$.
	We equipped the group $\aut(C_\Gamma)$ with the compact open topology, which makes it a locally compact second countable topological group (as $C_\Gamma$ is locally finite).
	
	\begin{lemma}
		\label{lem:star rigid}
		Suppose $A_\Gamma$ is large-type and $\Gamma$ is connected. The topological group $\aut(C_\Gamma)$ is discrete if and only if $\Gamma$ is star rigid in the sense that each label-preserving automorphism of $\Gamma$ fixing the closed star of a vertex of $\Gamma$ is identity. Moreover, if $\Gamma$ is not star rigid, then $\aut(I_\Gamma)$ is not discrete.
	\end{lemma}
	
	\begin{proof} 
		We assume $\Gamma$ is not a single vertex, otherwise the lemma is clear.
		Let $\aut_0(C_\Gamma)$ be the subgroup of $\aut(C_\Gamma)$ fixing the identity vertex $v_0\in C_\Gamma$. We claim that $\aut_0(C_\Gamma)$ is finite if and only if $\Gamma$ is star rigid. Then the lemma follows.
		
		Let $F\subset D_\Gamma$ be the closed star of the type 0 vertex $v_0$ corresponding to the identity of $A_\Gamma$. Note that $D_\Gamma/A_\Gamma$ can be identified with $F$, i.e. $F$ is a strict fundamental domain for the action $A_\Gamma\actson D_\Gamma$. This gives a map $\pi:D_\Gamma\to F$.
		Let $\aut_0(D_\Gamma)$ be the collection of automorphisms of $D_\Gamma$ fixing $F$ setwise.
		By Proposition~\ref{prop:first comparison}, the above claim reduces to showing that, $\aut_0(D_\Gamma)$ is finite if and only if $\Gamma$ is star rigid.
		
		We will first show if $\Gamma$ is not star rigid, then $\aut_0(D_\Gamma)$ is infinite. Then there is a non-trivial label-preserving automorphism $f$ of $\Gamma$ fixing the closed star a vertex $a\in \Gamma$. Then $f$ induces an automorphism of $A_\Gamma$, which gives an automorphism $\alpha_f$ of $D_\Gamma$.
		
		Let $T$ be the
		Bass-Serre tree of the splitting $A_\Gamma=A_{st(v)}*_{A_{\lk(v)}}A_{\Gamma\setminus\{v\}}$. We need the following alternative description of $T$ in terms of $D_\Gamma$.
		Let $F_a$ be the full subcomplex of $F$ spanned by $v_0$ and type $1$ vertices corresponding the identity cosets $\langle a'\rangle$ with $a'$ ranging over vertices in $\lk(a,\Gamma)$. It follows from the way we define metric on $D_\Gamma$ that each connected component of $\pi^{-1}(F_a)$ is a convex subset which is isometric to a tree. Then edges of $T$ are in 1-1 correspondence with components of $\pi^{-1}(F_a)$, vertices of $T$ are in 1-1 correspondence with components of $D_\Gamma\setminus \pi^{-1}(F_a)$, and an edge of $T$ contain a vertex of $T$ if the component of $\pi^{-1}(F_a)$ corresponding to the edge is contained in the closure of the component of $D_\Gamma\setminus \pi^{-1}(F_a)$ corresponding to the vertex.

		Let $T_a$ be the component of $\pi^{-1}(F_a)$ that passes through $v_0$. Then the above description of $T$ allows us to define an automorphism $\alpha\in \aut_0(D_\Gamma)$ which is equal to $\alpha_f$ on some components of $D_\Gamma\setminus T_a$, and equal to identity on the remain components of $D_\Gamma\setminus T_a$. Note that we can find infinitely many different conjugates of $\alpha$ by isometries of $D_\Gamma$ which are also in $\aut_0(D_\Gamma)$. Thus $\aut_0(D_\Gamma)$ is infinite.

		Now we show if $\Gamma$ is star rigid, then $\aut_0(D_\Gamma)$ is finite. This uses an argument of Crisp \cite[Section 11]{crisp2005automorphisms}, which we reproduce here. Given an automorphism $\alpha\in \aut_0(D_\Gamma)$. Then $\alpha$ induces an automorphism $\alpha'$ of $\lk(v_0,D_\Gamma)$, which is identified with the barycentric subdivision of $\Gamma$. As $\alpha$ preserves type of vertices, $\alpha'$ arises from an automorphism $\alpha''$ of $\Gamma$. This gives a group homomorphism $\aut_0(D_\Gamma)\to \aut(\Gamma)$. It suffices to show the kernel of this homomorphism is finite. So we will assume $\alpha$ is identity on $F$. We define a \emph{chamber} of $D_\Gamma$ to some $A_\Gamma$ of $F$, and two chambers are \emph{adjacent} if they share a type 1 vertex. If a chamber $F'$ is adjacent to $F$, then by the fact that $\alpha$ is identity on the closed star of each type 2 vertex of $F$ and that $\Gamma$ is star rigid, we know $\alpha$ is identity on $F'$. Then we can repeat the previous argument to deduce that $\alpha$ is identity on the closed star of each type 2 vertex in $F'$. 
		Note that $D_\Gamma$ is a union of chambers, and any two chambers are connected by a chain of chambers such that adjacent members in the chain are adjacent chambers. This implies that $\alpha$ is identity on $D_\Gamma$, as desired.
	\end{proof}
	We recall several definitions from coarse geometry.
	\begin{definition}
		An $(L,A)$-quasi-action of a group $G$ on a metric space $Z$ is a map $\rho:G\times Z\to Z$ so that $\rho(\gamma,\cdot):Z\to Z$ is an $(L,A)$ quasi-isometry for every $\gamma\in G$, $d(\rho(\gamma_1,\rho(\gamma_2,z)),\rho(\gamma_1\gamma_2,z))<A$ for every $\gamma_1,\gamma_2\in G$, $z\in Z$, and $d(\rho(e,z),z)<A$ for every $z\in Z$.
	\end{definition}
	
	The action $\rho$ is \textit{discrete} if for any point $z\in Z$ and any $R>0$, the set of all $\gamma\in G$ such that $\rho(\gamma,z)$ is contained in the ball $B_{R}(z)$ is finite; $\rho$ is \textit{cobounded} if $Z$ coincides with a finite tubular neighbourhood of the \textquotedblleft orbit\textquotedblright\ $\rho(G,z)$. If $\rho$ is a discrete and cobounded quasi-action of $G$ on $Z$, then the orbit map $\gamma\in G\to \rho(\gamma,z)$ is a quasi-isometry. Conversely, given a quasi-isometry between $G$ and $Z$, it induces a discrete and cobounded quasi-action of $G$ on $Z$.
	
	Two quasi-actions $\rho$ and $\rho'$ are \textit{equivalent} if there exists a constant $D$ so that 
	\begin{equation*}
		\sup_{\gamma\in G}\sup_{z\in Z}d(\rho(\gamma,z),\rho'(\gamma,z))<D.
	\end{equation*}

	\begin{proposition}
		\label{prop:qi}
		Suppose $A_\Gamma$ is a large-type Artin group, such that $\Gamma$ is connected without valence one vertices, and $\Gamma$ is not a triangle with each edge labelled by $3$.	Suppose $\Gamma$ is fundamental with respect to $D_\Gamma$ in the sense of Definition~\ref{DefiFundGraphI}. Then the following holds true:
		\begin{enumerate}
			\item any self quasi-isometry of $A_\Gamma$ is at bounded distance from an automorphism of the Cayley complex $C_\Gamma$;
			\item the quasi-isometry group of $A_\Gamma$ is isomorphic to $\aut(C_\Gamma)$;
			\item any finitely generated group quasi-isometric to $A_\Gamma$ is virtually isomorphic to a uniform lattice in the locally compact topological group $\aut(C_\Gamma)$;
			\item the outer automorphism group of $A_\Gamma$ is finite, and it is generated by global inversion and label-preserving automorphisms of the defining graph $\Gamma$.
		\end{enumerate}
		If in addition the topological group $\aut(C_\Gamma)$ is discrete (equivalently $\Gamma$ is star rigid), then 
		\begin{enumerate}
			\item any self quasi-isometry of $A_\Gamma$ is at bounded distance from an automorphism of $A_\Gamma$;
			\item any finitely generated group quasi-isometric to $A_\Gamma$ is virtually isomorphic to $A_\Gamma$;
			\item the quasi-isometry group of $A_\Gamma$ is isomorphic to $\aut_{\Gamma}(A_{\Gamma})$.
		\end{enumerate}
	\end{proposition}
	
	\begin{proof}
		Given $q\in \qi(A_\Gamma)$, let $\alpha=h_3\circ \Phi (q)$ where $\Phi$ is defined in Definition~\ref{def:Phi} and $h_3$ is defined in Lemma~\ref{lem:third comparison}. 
		By our assumption, Theorem~\ref{thm:2ndiso} and Lemma~\ref{lem:second comparison}, there is a unique $\beta\in \aut(D_\Gamma)$ such that $h_2(\beta)=\alpha$. Let $\gamma\in \aut(C_\Gamma)$ be $\gamma=h_1(\alpha)$ with $h_1$ as in Definition~\ref{def:first comparison}. It follows from the definition of $h_1,h_2,h_3$ and $\Phi$ that $h_3\circ \Phi (\gamma)=h_3\circ \Phi(q)$ (we view $\gamma$ as a self quasi-isometry of $A_\Gamma$). Since $h_3,h_2,h_1, \Phi$ are all injective (Lemma~\ref{lem:third comparison}, Lemma~\ref{lem:second comparison} and Proposition~\ref{prop:first comparison}), we know that $q=\gamma$ in $\qi(A_\Gamma)$, thus the first assertion follows. The second assertion then follows from the injectivity of $h_3\circ \Phi$. From the first two assertions and Proposition~\ref{prop:Phi injective}, we know each $(L,A)$-self quasi-isometry of $A_\Gamma$ is at distance at most $C=C(L,A,\Gamma)$ from a unique automorphism of $C_\Gamma$. 
		
		Now we prove Assertion 3. Given a finitely generated group $H$ quasi-isometric to $A_\Gamma$, then there is a discrete and cobounded $(L,A)$-quasi-action $\rho: H\times A_\Gamma\to A_\Gamma$ of $H$ on $A_\Gamma$ for some choice of $L$ and $A$. By the previous paragraph, there exists $C>0$ such that for each $h\in H$, we can replace $\rho(h,\cdot):A_\Gamma\to A_\Gamma$ by a unique automorphism of $C_\Gamma$ at distance at most $C$ from $\rho(h,\cdot)$. This gives an quasi-action $\rho':H\curvearrowright C_\Gamma$ which is equivalent to $\rho$, moreover, the uniqueness of the replacement implies that $\rho'$ is an action.
		We also know that $\rho'$ is proper and cocompact. Thus $H$ is virtually a uniform lattice in $\aut(C_\Gamma)$. If in addition $\aut(C_\Gamma)$ is discrete, any two uniform lattice in $\aut(C_\Gamma)$ are commensurable, in particular, $H$ is virtually isomorphic to $A_\Gamma$.
		
		For Assertion 4, let $F\subset D_\Gamma$ and $v_0$ be as in the proof of Lemma~\ref{lem:star rigid}. Given an automorphism $\alpha$ of $A_\Gamma$, viewed as a quasi-isometry, there is a unique automorphism $\alpha'\in \aut(D_\Gamma)$ such that $h_2(\alpha')=h_3\circ \Phi$ as before. Up to modify $\alpha$ by an inner automorphism, we can assume $\alpha'(F)=F$. Then $\alpha'$ induces an automorphism of $\alpha'_{v_0}$ of $\lk(v_0,D_\Gamma)$, which comes from an automorphism $\alpha''$ of $\Gamma$ as in the proof of Lemma~\ref{lem:star rigid}. By \cite[Lemma 39]{crisp2005automorphisms}, $\alpha'_{v_0}$ sends a type 2 vertex of $F$ corresponding to $A_e$ to a possibly different type 2 vertex of $F$ corresponding to $A_{e'}$, with the label of $e$ and $e'$ being the same. Thus $\alpha''$ preserves labels of edges of $\Gamma$. Hence up to post-compose $\alpha'$ by an automorphism of $A_\Gamma$ induced by a label-preserving automorphism of $\Gamma$, we can assume $\alpha'$ fixes $F$ pointwise. This implies (by unwinding the definition of $\alpha'$) that the Hausdorff distance between $\alpha(A_e)$ and $A_e$ is finite for each edge $e\subset \Gamma$. As $\alpha(A_e)$ belongs to the list in \cite[Theorem D]{vaskou2023isomorphism}, the only possibly which is Hausdorff close to $A_e$ is $A_e$ itself. By \cite[Lemma 40]{crisp2005automorphisms}, $\alpha|_{A_e}$ is either identity or an inversion sending each generator to its inverse. As $\Gamma$ is connected, $\alpha$ is either identity or a global inversion, which finishes the proof of Assertion 4.
		
		Now we prove the “in addition” part. It follows from the argument in the last paragraph of the proof of Lemma~\ref{lem:star rigid} that for any $\beta\in \aut(D_\Gamma)$, $h^{-1}_1(\beta)$ is an automorphism of $A_\Gamma$. Now we repeat the argument in the first paragraph of the proof to see that any self quasi-isometry of $A_\Gamma$ is at bounded distance from an automorphism of $A_\Gamma$. Now the other components of the in addition part follows.
	\end{proof}

	Let $\Delta_{333}$ denotes the triangle with each edge labelled by $3$.
	
	\begin{corollary}
		\label{cor:qi}
		Suppose $A_\Gamma$ is a large-type Artin group with $\Gamma\neq\Delta_{333}$ such that $\Gamma$ has a twistless hierarchy terminating on twistless stars. 
		Then \begin{enumerate}
			\item any self quasi-isometry of $A_\Gamma$ is at bounded distance from an automorphism of the Caylye complex $C_\Gamma$;
			\item the quasi-isometry group of $A_\Gamma$ is isomorphic to $\aut(C_\Gamma)$;
			\item any finitely generated group quasi-isometric to $A_\Gamma$ is virtually isomorphic to a uniform lattice in the locally compact topological group $\aut(C_\Gamma)$;
			\item the outer automorphism group of $A_\Gamma$ is finite, and it is generated by global inversion and label-preserving automorphisms of the defining graph $\Gamma$.
		\end{enumerate}
		If in addition $\Gamma$ is star rigid, then 
		\begin{enumerate}
			\item any self quasi-isometry of $A_\Gamma$ is at bounded distance from an automorphism of $A_\Gamma$;
			\item any finitely generated group quasi-isometric to $A_\Gamma$ is virtually isomorphic to $A_\Gamma$;
			\item the quasi-isometry group of $A_\Gamma$ is isomorphic to $\aut_{\Gamma}(A_{\Gamma})$.
		\end{enumerate}
		If $\Gamma=\Delta_{333}$, then the outer automorphism group of $A_\Gamma$ is still generated by global inversion and label-preserving automorphisms of the defining graph. Moreover, there is a finite index super-group $H$ containing $A_\Gamma$ such that 
		\begin{enumerate}
			\item any self-quasi-isometry of $A_\Gamma$ is at bounded distance from a left translation of $H$, hence the quasi-isometry group of $A_\Gamma$ is isomorphic to $H$;
			\item any finitely generated group quasi-isometric to $A_\Gamma$ is virtually isomorphic to $A_\Gamma$.
		\end{enumerate}
	\end{corollary}
	
	\begin{proof}
		The case when $\Gamma$ is not a triangle with each edge labelled by $3$ is a consequence of Proposition~\ref{prop:qi} and Proposition~\ref{prop:fund}. If $\Gamma$ is a triangle with each edge labelled by $3$, then $A_\Gamma$ has finite index in the mapping class group $H$ of $5$-punctured sphere. The corollary follows directly from quasi-isometric rigidity results on mapping class groups \cite{hamenstaedt2005geometry,behrstock2008quasi}. The statement about outer automorphism group of $A_\Gamma$ follows from \cite{vaskou2023automorphisms}.
	\end{proof}

	\subsection{On Conjecture~\ref{conj:main}}

	\begin{proposition}
		\label{lem:onlyif}
		The only if direction of Conjecture~\ref{conj:main} holds.
	\end{proposition}
	
	\begin{proof}
		We first show if $\Gamma$ has a separating edge or a separating vertex, then $\aut(I_\Gamma)$ is not locally compact. We argue by contradiction and assume $\aut(I_\Gamma)$ is locally compact, then for any vertex $v\in I_\Gamma$, the stabiliser $G_v\le\aut(I_\Gamma)$ of $v$ is compact. If $\Gamma$ has a separating $e$, then there are full subgraphs $\Gamma_1,\Gamma_2$ of $\Gamma$ such that $\Gamma_1\cap\Gamma_2=e$,  $\Gamma_1\cup\Gamma_2=\Gamma$ and $\Gamma_1$ and $\Gamma_2$ properly contain $e$. Consider the splitting $A_{\Gamma}=A_{\Gamma_1}*_{A_e}A_{\Gamma_2}$ and let $z$ be a non-trivial element in the centre of $A_e$. Let $\phi_n$ be the automorphism of $A_\Gamma$ which is identity on $A_{\Gamma_1}$ and is the conjugation by $z^n$ on $A_{\Gamma}$. One readily checks that $\phi_n$ induces a sequence of automorphisms $\varphi_n:I_\Gamma\to I_\Gamma$. There is a vertex $v\in I_\Gamma$ which is fixed by any $\varphi_n$, although the sequence $\{\varphi_n\}_{n\ge 1}$ does not have a limit, which contradicts that $G_v$ is compact.
		
		Now we assume $\Gamma$ is twistless. It remains to show that, if $\Gamma$ is not star rigid, then $\aut(I_\Gamma)$ is not discrete. If $\aut(I_\Gamma)$ is discrete, then $G_v$ defined in the previous paragraph is finite. However, the infinitely many different elements we constructed in $\aut_0(D_\Gamma)$ in the proof of Lemma~\ref{lem:star rigid} induces infinitely many different elements in $\aut(I_\Gamma)$ which fix a common vertex, which is a contradiction.
	\end{proof}

	\begin{proposition}
		\label{prop:convert}
		Suppose $A_\Gamma$ is a large-type Artin group with $\Gamma\neq\Delta_{333}$ such that $\Gamma$ is connected and does not have leaf vertices.
		Suppose $\Gamma$ is fundamental with respect to $D_\Gamma$ in the sense of Definition~\ref{DefiFundGraphI}. Then
		$\aut(I_\Gamma)$ is locally compact and it is isomorphic to $\aut(C_\Gamma)$. If in addition that $\Gamma$ is star rigid, then $\aut(I_\Gamma)$ is discrete and it is isomorphic to $\aut_{\Gamma}(A_{\Gamma})$ (with $\aut_{\Gamma}(A_{\Gamma})$ defined in the beginning of Section~\ref{subsec:intro2}).
	\end{proposition}
	
	\begin{proof}
		By Theorem~\ref{thm:2ndiso}, the second comparison homomorphism $h_2:\aut(D_\Gamma)\to\aut(I^{TD}_\Gamma)$ is an isomorphism. By Lemma~\ref{lem:third comparison} and Proposition~\ref{prop:first comparison}, the homomorphism $h_1\circ h^{-1}_2\circ h_3:\aut(I_\Gamma)\to \aut(C_\Gamma)$ is injective. One readily checks this homomorphism is also continuous However, as each element in $\aut(C_\Gamma)$ can be viewed as an element in $\qi(A_\Gamma)$, hence induces an automorphism of $I_\Gamma$ via Definition~\ref{def:Phi}. It follows that $h_1\circ h^{-1}_2\circ h_3$ is surjective, hence an isomorphism. The in addition part is similar to the proof of the in addition part of Proposition~\ref{prop:qi}.
	\end{proof}
	
	The following is a consequence of Proposition~\ref{prop:convert} and Proposition~\ref{lem:onlyif} (note that $\Gamma$ being twistless implies that $\Gamma$ is connected and does not have leaf vertices).
	\begin{corollary}
		\label{cor:convert}
		Suppose $\Gamma$ is fundamental in the Deligne complex $D_\Gamma$ in the sense of Definition~\ref{DefiFundGraphI} whenever $\Gamma$ is twistless. Then Conjecture~\ref{conj:main} holds.
	\end{corollary}
	
	This motivates the following conjecture.
	\begin{conjecture}
		\label{conj:deligne}
		Suppose $A_\Gamma$ is of large-type and $\Gamma$ is twistless. Then $\Gamma$ is fundamental in the Deligne complex $D_\Gamma$ in the sense of Definition~\ref{DefiFundGraphI}
	\end{conjecture}
	
	The following is a consequence of Proposition~\ref{lem:onlyif}, Proposition~\ref{prop:convert} and Proposition~\ref{prop:fund} (the assumption of the following theorem implies $\Gamma$ is connected and does not have leaf vertices).
	\begin{theorem}
		\label{thm:main}
		Suppose $A_\Gamma$ is of large-type. If $\Gamma$ admits a twistless hierarchy terminating in twistless stars, then Conjecture~\ref{conj:main} holds.
	\end{theorem}
	
	\begin{theorem}
		\label{thm:qi reduction}
		If Conjecture~\ref{conj:deligne} holds. Then Conjecture~\ref{conj:qi} holds.
	\end{theorem}
	
	\begin{proof}
		The only if direction of Conjecture~\ref{conj:qi} can be proved in a similar way to the only if direction of Conjecture~\ref{conj:main}. The if direction follows from Proposition~\ref{prop:qi}.
	\end{proof}

	\subsection{ME and OE rigidity}	
	Recall that an Artin group is of hyperbolic-type if the associated Coxeter group is Gromov hyperbolic. A large-type Artin group $A_\Gamma$ is of \emph{hyperbolic-type} if $\Gamma$ does not contain a triangle with each edge labelled by $3$. In this case, $I^{TD}_\Gamma=I_\Gamma$.
	\begin{proposition}
		\label{prop:ME}
		Suppose $A_\Gamma$ is a large-type, hyperbolic-type Artin group such that $\Gamma$ is connected and does not have leaf vertices. Suppose $\Gamma$ is fundamental with respect to $D_\Gamma$ in the sense of Definition~\ref{DefiFundGraphI}. Then any countable group measure equivalent to $A_\Gamma$ is virtually a lattice in the locally compact topological group $\aut(C_\Gamma)$.
		
		If in addition the topological group $\aut(C_\Gamma)$ is discrete (equivalently $\Gamma$ is star rigid), then any countable group measure equivalent to $A_\Gamma$ is virtually isomorphic to $A_\Gamma$.
	\end{proposition}
	
	\begin{proof}
		By Lemma~\ref{lem:second comparison} and Theorem~\ref{thm:2ndiso}, $h_2$ is an isomorphism. Then $h_2\circ h^{-1}_1:\aut(C_\Gamma)\to \aut(I_\Gamma)$ is an isomorphism, and it is the comparison map between $\aut(C_\Gamma)$ and $\aut(I_\Gamma)$ as defined in \cite[Section 8]{horbez2020boundary}. Thus the first part of the proposition follows from \cite[Theorem 10.4]{horbez2020boundary}.
		The in addition part follows from \cite[Theorem 10.5]{horbez2020boundary}, as $\aut(C_\Gamma)$ is discrete implies that the natural map $A_\Gamma\to \aut(C_\Gamma)$ induced by action of $A_\Gamma$ on $C_\Gamma$ by left translations has finite index image.
	\end{proof}

	\begin{proposition}
		\label{prop:SOE}
		Suppose $A_\Gamma$ is a large-type, hyperbolic-type Artin group such that $\Gamma$ is connected and does not have leaf vertices. Suppose $\Gamma$ is fundamental with respect to $D_\Gamma$ in the sense of Definition~\ref{DefiFundGraphI}.
		Assume in addition that $\Gamma$ is star rigid.
		
		Then for any ergodic measure-preserving essentially free $A_\Gamma\actson X$ on a probability measure space $X$, and any countable group $G'$ with an ergodic measure-preserving essentially free $G'$-action on a standard probability space $X'$, as long as actions $G\actson X$ and $G'\actson X'$ are stably orbit equivalent, then they are virtually conjugate.
	\end{proposition}
	
	\begin{proof}
		By Proposition~\ref{prop:convert}, $\aut(I_\Gamma)$ is isomorphic to $\aut_{\Gamma}(A_{\Gamma})$. Now the proposition follows from \cite[Theorem 10.1]{horbez2020boundary}, and \cite[Lemma 4.18]{furman2009survey}.
	\end{proof}
	
	\begin{corollary}
		\label{cor:ME}
		Suppose $A_\Gamma$ is a large-type, hyperbolic-type Artin group such that $\Gamma$ has a twistless hierarchy terminating in twistless stars. Then any countable group measure equivalent to $A_\Gamma$ is virtually a lattice in the locally compact topological group $\aut(C_\Gamma)$.
		
		If in addition the graph $\Gamma$ is star rigid, then any countable group measure equivalent to $A_\Gamma$ is virtually isomorphic to $A_\Gamma$. Moreover, then for any ergodic measure-preserving essentially free $A_\Gamma\actson X$ on a probability measure space $X$, and any countable group $G'$ with an ergodic measure-preserving essentially free $G'$-action on a standard probability space $X'$, as long as actions $G\actson X$ and $G'\actson X'$ are stably orbit equivalent, then they are virtually conjugate.
	\end{corollary}
	
	\begin{proof}
		Proposition~\ref{prop:ME}, Proposition~\ref{prop:SOE} and Proposition~\ref{prop:fund}.
	\end{proof}
	
	Now we discuss some consequences on lattice embedding of Artin groups. The following two theorems can be proved in a similar way as Proposition~\ref{prop:ME}, using 	\cite[Theorem 10.9]{horbez2020boundary} and 	\cite[Theorem 10.10]{horbez2020boundary} respectively.
	
	\begin{theorem}\label{theo:lattice-embeddings1}
		Suppose $A_\Gamma$ is a large-type, hyperbolic-type Artin group such that $\Gamma$ has a twistless hierarchy terminating in twistless stars. Suppose $G'$ is a finite index subgroup of $A_\Gamma$ which embeds into a locally compact second countable group $\mathsf{H}$ as a lattice via $f:G'\to \mathsf{H}$.
		
		Then there exists a continuous homomorphism $\psi:\mathsf{H}\to \aut(C_\Gamma)$ with compact kernel such that $\psi\circ f$ coincides with the natural map $G'\to\aut(C_\Gamma)$. 
	\end{theorem}

	\begin{theorem}\label{theo:lattice-embeddings2}
		Suppose $A_\Gamma$ is a large-type, hyperbolic-type Artin group such that $\Gamma$ has a twistless hierarchy terminating in twistless stars. Then there exists $G_0$ which is virtually isomorphic to $G$ such that the following holds true.
		
		Suppose $A_\Gamma$ or one of its finite index subgroup embeds into a locally compact second countable group $\mathsf{H}$ as a lattice. Then there exists a continuous homomorphism $g:\mathsf{H}\to G_0$ with compact kernel.
	\end{theorem}

	\subsection{Rigidity of cross-product von Neumann algebra}
	Given a countable group $G$ and a standard probability space $(X,\mu)$ equipped with an ergodic measure-preserving free $G$-action by Borel automorphisms, let $L(G\actson X)$ be the associated cross-product von Neumann algebra.
	
	\begin{proposition}
		\label{prop:von-neumann-strong}
		Let $A_\Gamma$ be a large-type, hyperbolic-type Artin group. Suppose that
		\begin{enumerate}
			\item $\Gamma$ is connected without valence one vertex and it is not a complete graph;
			\item $\Gamma$ is star rigid, and Suppose $\Gamma$ is fundamental with respect to $D_\Gamma$ in the sense of Definition~\ref{DefiFundGraphI}.
		\end{enumerate}
		Let $H$ be a countable group, and let $A_\Gamma\actson X$ and $H\actson Y$ be free, ergodic, measure-preserving actions by Borel automorphisms on standard probability spaces.  
		
		If these two actions are stably $W^*$-equivalence, then they are virtually conjugate (in particular $A_\Gamma$ and $H$ are almost isomorphic).
	\end{proposition}
	
	\begin{proof}
		As $\Gamma$ is not a complete graph, by \cite[Theorem 11.1]{horbez2020boundary} (relying on \cite{Ioa2}), $L^\infty(X)$ is the unique Cartan subalgebra of $L(A_\Gamma\actson X)$ up to unitary conjugation. Thus the stably $W^*$-equivalence assumption implies that $A_\Gamma\actson X$ and $H\actson Y$ are stably orbit equivalent (see, e.g. the proof of \cite[Corollary 6.4]{horbez2022measure} for an explanation). Now we are done by Proposition~\ref{prop:SOE}.
	\end{proof}

	The following is a consequence of Proposition~\ref{prop:von-neumann-strong} and Proposition~\ref{prop:fund}.	
	\begin{corollary}
		\label{cor:von-neumann}
		Let $A_\Gamma$ be a large-type Artin group of hyperbolic-type. Suppose that
		\begin{enumerate}
			\item $\Gamma$ admits a twistless hierarchy terminating in twistless stars. 
			\item $\Gamma$ is star rigid, and $\Gamma$ is not a complete graph.
		\end{enumerate}
		Let $H$ be a countable group, and let $A_\Gamma\actson X$ and $H\actson Y$ be free, ergodic, measure-preserving actions by Borel automorphisms on standard probability spaces.  
		
		If these two actions are stably $W^*$-equivalence, then they are virtually conjugate (in particular $A_\Gamma$ and $H$ are almost isomorphic).
	\end{corollary}

	\subsection{Classification results}
	\begin{lemma}
		\label{lem:classification}
		Let $A_{\Gamma_1}$ and $A_{\Gamma_2}$ be two large-type Artin groups. Suppose $\Gamma_1$ is connected without leaf vertices, and $\Gamma_1$ is fundamental with respect to $D_{\Gamma_2}$ (Definition~\ref{DefiFundGraphI}).
		
		If $I_{\Gamma_1}$ and $I_{\Gamma_2}$ are isomorphic, then there is a label-preserving isomorphism
		$\Gamma_1\to\Gamma_2$.
	\end{lemma}
	
	\begin{proof}
		First we consider the case $\Gamma_1=\Delta_{333}$. Then $I_{\Gamma_1}$ is isomorphic to the curve graph of the $5$-punctured sphere $\Sigma_5$, see Definition~\ref{defFstar}.
		As the mapping class group of $\Sigma_5$ acts transitively on the set of homotopy classes of essential simple close curves on $\Sigma_5$, we know the automorphism group of $I_{\Gamma_1}$ acts transitively on the vertex set of $I_{\Gamma_1}$. Note that $\Gamma_2$ must be connected, otherwise $I_{\Gamma_2}$ is not connected. If $\Gamma_2\neq\Delta_{333}$, then Corollary~\ref{CorollaryExoticSentToExotic} implies that $\aut(I_{\Gamma_2})$ does not act transitively on the vertex set of $I_{\Gamma_2}$, which contradicts that $I_{\Gamma_1}$ and $I_{\Gamma_2}$ are isomorphic as graphs. Thus $\Gamma_2=\Delta_{333}$.
		
		Now we assume $\Gamma_1\neq\Delta_{333}$. By the argument in the previous paragraph, we can also assume $\Gamma_2\neq\Delta_{333}$.
		Let $\alpha:I_{\Gamma_1}\to I_{\Gamma_2}$ be an isomorphism. It follows from Corollary~\ref{CorollaryExoticSentToExotic} that $\alpha$ induces an bijection of vertices of type $E$. Thus $\alpha$ induces an isomorphism of graphs $I^{TD}_{\Gamma_1}\to I^{TD}_{\Gamma_2}$, which we still denote by $\alpha$. As $\alpha$ preserves types of vertices by Proposition~\ref{PropAutomorphismsPreserveType}, we know it gives a 1-1 correspondence between type 2 vertices of $D_{\Gamma_1}$ and type 2 vertices of $D_{\Gamma_2}$. By the same argument as in \cite[pp. 1435]{crisp2005automorphisms}, this correspondence between type 2 vertices extends to an isometric embedding $\varphi:D_{\Gamma_1}\to D_{\Gamma_2}$.
		
		Next we show $\Gamma_2$ does not have any leaf vertices. Suppose $\Gamma_2$ has a leaf vertex $a$. Let $e=\overline{ab}$ be the edge of $\Gamma_2$ containing $e$. Let $v\in D_{\Gamma_2}$ be a vertex corresponding to $gA_e$ for some $g\in A_{\Gamma_2}$. Take $v'\in D_{\Gamma_1}$ with $\varphi(v')=v$ (such $v'$ exists as $\alpha$ induces a bijection of type 2 vertices). Then $\varphi$ sends the closed star of $v'$ to closed star of $v$ isometrically. As $\Gamma_1$ does not contain leaf vertices, each triangle of $D_{\Gamma_1}$ in the star of $v$ cannot contain any free edges (a free edge is an edge which is contained in exactly one triangle). On the other hand, as $a$ is a leaf vertex of $\Gamma_2$, there is a triangle in the closed star of $v$ such that this triangle contains a free edge. As $\varphi$ sends triangles without free edges to triangles without free edges, this leads to a contradiction. Thus $\Gamma_2$ does not have any leaf vertices.
		
		It follows that $D_{\Gamma_2}$ does not contain any free edges, hence it has geodesic extension property. Now we argue as \cite[Theorem 8.7]{horbez2020boundary} to see that $\varphi$ is actually surjective, hence it is an isometry. Now we restrict $\varphi$ to the closed star of a type 0 vertex and argue as in the last paragraph of the proof of Proposition~\ref{prop:qi} to see that $\varphi$ induces a label-preserving isomorphism between $\Gamma_1$ and $\Gamma_2$.
	\end{proof}
	
	\begin{corollary}
		\label{cor:classification1}
		Suppose $A_\Gamma$ and $A_{\Gamma'}$ are two large-type Artin groups such that $\Gamma$ admits a twistless hierarchy terminating in twistless stars. Then the following are equivalent:
		\begin{enumerate}
			\item there is a label-preserving isomorphism between $\Gamma$ and $\Gamma'$;
			\item $A_\Gamma$ and $A_{\Gamma'}$ are isomorphic;
			\item $A_\Gamma$ and $A_{\Gamma'}$ are commensurable;
			\item $A_\Gamma$ and $A_{\Gamma'}$ are quasi-isometric.
		\end{enumerate}
		Then there is a label-preserving isomorphism between their defining graphs.
	\end{corollary}
	
	\begin{proof}
		$1$ implies 2, $2$ implies $3$, and $3$ implies $4$ are trivial. It suffices to prove $4$ implies 1, which follows from Proposition~\ref{prop:fund}, Lemma~\ref{lem:classification} and \cite[Theorem 10.16]{huang2017quasi}.
	\end{proof}
	
	\begin{corollary}
		\label{cor:classification2}
		Suppose $A_\Gamma$ and $A_{\Gamma'}$ are two large-type and hyperbolic-type Artin groups such that $\Gamma$ admits a twistless hierarchy terminating in twistless stars. 
		Then the following are equivalent.
		\begin{enumerate}
			\item there is a label-preserving isomorphism between $\Gamma$ and $\Gamma'$ ;
			\item $A_\Gamma$ and $A_{\Gamma'}$ admits orbit equivalent or stably orbit equivalent free, ergodic, measure-preserving actions by Borel automorphisms on standard probability spaces;
			\item $A_\Gamma$ and $A_{\Gamma'}$ are measure equivalent.
		\end{enumerate}
		If in addition both $\Gamma$ and $\Gamma'$ are not complete graphs, then the above items are equivalent to that the cross product von-Neumann algebras $L(A_\Gamma\actson X)$ and $L(A_{\Gamma'}\actson Y)$ have isomorphic amplifications, for some free, ergodic, measure-preserving actions by Borel automorphisms on standard probability spaces $A_\Gamma\actson X$ and $A_{\Gamma'}\actson Y$.
	\end{corollary}
	
	\begin{proof}
		For the first part, it suffices to prove $3$ implies 1, which follows from Lemma~\ref{lem:classification}, Proposition~\ref{prop:fund} and \cite[Theorem 10.1]{horbez2020boundary}. The in addition part is similar to the proof of Proposition~\ref{prop:von-neumann-strong}.
	\end{proof}

	\bibliographystyle{alpha}
	\bibliography{mybib}

\end{document}